\documentclass[a4paper, 12pt]{amsart}

\usepackage{amsmath}
\usepackage{amssymb,esint}
\usepackage{colortbl}
\usepackage{mathtools}
\usepackage{aligned-overset}
\usepackage{hyperref} 
\usepackage[top=3cm, bottom=2.75cm, left=3cm, right=3cm]{geometry}
\newtheorem{theorem}{Theorem}[section]

\newtheorem{corollary}[theorem]{Corollary}

\newtheorem{lemma}[theorem]{Lemma}

\newtheorem{proposition}[theorem]{Proposition}
\theoremstyle{definition}
\newtheorem{definition}[theorem]{Definition}
\theoremstyle{remark}
\newtheorem{remark}[theorem]{Remark}

\numberwithin{equation}{section}

\newcommand{ \R }{ \mathbb{R} }
\def\tail{\operatorname{Tail}}
\DeclareMathOperator*{\osc}{osc}

\begin{document}
\title{Wolff potentials and nonlocal equations of Lane--Emden type}
\thanks{The first author was supported by the Academy of Mathematics and Systems Science, Chinese Academy of Sciences startup fund; CAS Project for Young Scientists in Basic Research, Grant No. YSBR-031; the National Natural Science Foundation of China (No. 12288201);  and the National Key R$\&$D Program of China under grant 2021YFA1000800. The second author was supported by the National Research Foundation of Korea(NRF) grant funded by the Korea government(MSIT) (RS-2025-24533680). The third author was supported by a KIAS individual grant (MG091702) at Korea Institute for Advanced Study.}

\author{Quoc-Hung Nguyen}
\address{Academy of Mathematics and Systems Science, Chinese Academy of Sciences, Beijing 100190, China}
\email{qhnguyen@amss.ac.cn}

\author{Jihoon Ok}
\address{Department of Mathematics, Institute for Mathematical and Data Science, Sogang University, Seoul 04107, Republic of Korea}
\email{jihoonok@sogang.ac.kr}

\author{Kyeong Song}
\address{School of Mathematics, Korea Institute for Advanced Study, Seoul 02455, Republic of Korea}	
\email{kyeongsong@kias.re.kr}




\begin{abstract}
We consider nonlocal equations of the type
\[ (-\Delta_{p})^{s}u = \mu \quad \text{in}\;\; \Omega, \]
where $\Omega \subset \mathbb{R}^{n}$ is either a bounded domain or the whole $\mathbb{R}^{n}$, $\mu$ is a Radon measure on $\Omega$, $0 < s < 1$ and $1 < p < n/s$. In particular, we extend the existence, regularity and Wolff potential estimates for SOLA (Solutions Obtained as Limits of Approximations), established by Kuusi, Mingione, and Sire (Comm. Math. Phys. 337(3):1317--1368, 2015), to the strongly singular case $1 < p \le 2-s/n$. 
Moreover, using Wolff potentials and Orlicz capacities, we present both a sufficient condition and a necessary condition for the existence of SOLA to nonlocal equations of the type 
\[ (-\Delta_{p})^{s}u = P(u) + \mu \quad \text{in}\;\; \Omega, \]
where $P(\cdot)$ is either a power function or an exponential function.
\end{abstract}

\maketitle

\setcounter{tocdepth}{1}
\tableofcontents

\section{Introduction and main results}
In this paper, we study the existence, regularity and Wolff potential estimates for SOLA (Solutions Obtained as Limits of Approximations) to the following nonlocal elliptic problem:
\begin{equation}\label{EQp}
\left\{
\begin{aligned}
-\mathcal{L}_\Phi u&=P(u)+\mu &\text{in }&\Omega,\\
u&=0& \text{in }&\mathbb{R}^n\setminus\Omega,
\end{aligned}
\right.
\end{equation}
where $\Omega\subset\mathbb{R}^n$ is either a bounded domain or the whole $\mathbb{R}^n$ with $n\geq 2$, $\mu$ is a Radon measure defined on $\Omega$ with finite total mass (we can always assume that $\mu$ is defined on the whole $\mathbb{R}^n$ by letting $|\mu|(\mathbb{R}^n\setminus \Omega)=0$) and $P(\cdot)$ is either  the zero function, a power function or an exponential function. 
The nonlinear integro-differential operator $-\mathcal L_\Phi$ is defined by
\begin{equation}\label{def.L}
\langle  -\mathcal{L}_\Phi u,\varphi\rangle = \int_{\mathbb{R}^n}\int_{\mathbb{R}^n}\Phi(u(x)-u(y))(\varphi(x)-\varphi(y))K(x,y)\,dx\,dy
\end{equation}
for every $\varphi \in C^{\infty}_{c}(\Omega)$. Here, $\Phi: \mathbb{R}\to \mathbb{R}$ is a continuous function satisfying the growth/coercivity properties
\begin{equation}\label{mono}
\Lambda^{-1}|t|^{p} \le \Phi(t)t \le \Lambda|t|^{p}
\end{equation} 
for all $t \in \mathbb{R}$. 
For the comparison principle argument in Proposition \ref{pro1}, which will be used in Section \ref{sec:LE}, we additionally assume the monotonicity property
\begin{equation}\label{mono2}
(\Phi(t)-\Phi(t'))(t-t') > 0
\end{equation}
for all $t, t' \in \mathbb{R}$ with $t \neq t'$. Moreover, $K:\mathbb{R}^n\times\mathbb{R}^n\to \mathbb{R}$ is a measurable kernel satisfying the ellipticity/coercivity properties
\begin{equation} \label{c2}
\frac{ \Lambda^{-1} }{|x-y|^{n+sp}}\leq K(x,y)\leq \frac{ \Lambda }{|x-y|^{n+sp}} 
\end{equation} 
for all $(x,y)\in \mathbb{R}^n\times\mathbb{R}^n$ with $x\neq y$, where $\Lambda\geq 1$ is a fixed constant. In this paper we assume that
\begin{equation}\label{s.p.range} 
0<s<1 \quad \text{and} \quad 1< p < \frac{n}{s}.
\end{equation}
Note that when $\Phi(t) = |t|^{p-2}t$ and $K(x,y) = |x-y|^{-(n+sp)}$, the operator $-\mathcal{L}_{\Phi}$ reduces to the ($s$-)fractional $p$-Laplacian $(-\Delta_{p})^{s}$. 
In particular, we obtain Wolff potential estimates for SOLA to the following nonlocal problem
\begin{equation}\label{EQ}
\left\{
\begin{aligned}
-\mathcal{L}_\Phi u&= \mu &\text{in }&\Omega,\\
u&=0&\text{in }&\mathbb{R}^n\setminus\Omega.
\end{aligned}
\right.
\end{equation}

In the case of local measure data problems involving the classical $p$-Laplacian $\Delta_{p}$, 
Boccardo and Gallou\"et \cite{BG89,BG92} introduced the notion of SOLA to the following Dirichlet problem
\begin{equation}\label{2hvEQ5}
\left\{
\begin{aligned}
-\Delta_{p}u & = \mu &\text{in }&\Omega, \\
u & = 0 & \text{in }& \partial\Omega
\end{aligned}
\right.
\end{equation}
with $p>2-1/n$, and proved its existence. 
Note that the lower bound $p > 2-1/n$ is unavoidable when considering SOLA to the local problem \eqref{2hvEQ5}, since such solutions are required to be $W^{1,1}$-distributional solutions. On the other hand, Dal Maso, Murat, Orsina and Prignet  \cite{22DMOP} systematically developed the notion of renormalized solution to \eqref{2hvEQ5}  with $p>1$ and a general measure $\mu$. For other notions of solutions to \eqref{2hvEQ5}, see \cite{BBGGPV95,BGO96,CiMa,Lin86}. 
In the context of nonlinear potential theory, Kilpel\"ainen and Mal\'y \cite{22KiMa1,22KiMa2} obtained pointwise estimates for solutions to  \eqref{2hvEQ5} in terms of the (truncated) Wolff potential ${\bf W}^{T}_{1,p}[\mu]$ of $\mu$, see \eqref{2hvEQ8} below for the definition. We further refer to \cite{DM10,DM11,KK10,KM12,KM13,KM14,Min11,22TW4} for potential estimates for solutions, and their gradient, to local problems of the type \eqref{2hvEQ5}.

Subsequently, the existence theory for the Lane--Emden type problem 
\begin{equation}\label{EQ:pcase}
\left\{
\begin{aligned}
-\Delta_p u&= P(u) +  \mu &\text{in }&\Omega,\\
u&=0&\text{on }&\partial\Omega
\end{aligned}
\right.
\end{equation}
was established in \cite{HV1,22PhVe}, where the necessary conditions and sufficient conditions for the existence of renormalized solutions were obtained in terms of Bessel capacities and Wolff potentials. For instance, when $P(u)= u^\gamma$ and $\Omega$ is bounded, Phuc and Verbitsky \cite{22PhVe} showed that for a nonnegative measure $\mu$ with compact support in $\Omega$, 
it is equivalent to solve  
\begin{equation*}
\left\{
\begin{aligned}
-\Delta_pu&=u^\gamma+\mu&\text{in }&\Omega,\\
u&=0&\text{on }&\partial\Omega,
\end{aligned}
\right.
\qquad u\ge 0,
\end{equation*}
or to have
\begin{equation*}
\mu(E)\leq c\mathrm{Cap}_{\mathbf{G}_p,\frac{\gamma}{\gamma+1-p}}(K) \quad \text{for any compact set $K\subset\Omega$,}
\end{equation*}
where $c>0$ is a universal constant and $\mathrm{Cap}_{\mathbf{G}_p,\gamma/(\gamma+1-p)}$ is the Bessel capacity, or to have
\begin{equation*}
\int_B\left({\bf W}^{2\mathrm{diam}(\Omega)}_{1,p}[\chi_{B}\mu](x)\right)^{\gamma}\,dx\leq c\mu (B) \quad \text{for any ball $B$ with  $B\cap\mathrm{supp}\, \mu\neq\emptyset$},
\end{equation*}
where $\mathrm{diam}(\Omega)$ is the diameter of $\Omega$ and $\mathrm{supp}\,\mu$ is the support of $\mu$. On the other hand, when $P(u)$ is an exponential function given in \eqref{2hvEQ11} below, V\'eron and the first author of this paper \cite{HV1} obtained a sufficient condition expressed in terms of fractional maximal functions and a necessary condition expressed in terms of Orlicz capacities. The constructions in both papers \cite{HV1,22PhVe} are based upon the sharp pointwise estimates for \eqref{2hvEQ5} with a nonnegative measure $\mu$, obtained in \cite{22KiMa1,22KiMa2}. For more related results, see \cite{22Bi2,22Bi3,22VHV,23VH} and also \cite{22V}.
 
In recent years, nonlocal problems have attracted considerable attention.  
The growing interest is  motivated not only by various models in the applied sciences, but also by the analysis of fractional Sobolev spaces and related variational problems. We refer to \cite{DiPaVa1,palatucci2018,Si2} and references therein for some early works in this direction. 
In particular, much attention has been devoted to the nonlocal counterparts of regularity results for classical local equations. In this paper, we deal with nonlocal Dirichlet problems of the type
\begin{equation}\label{EQg}
\left\{
\begin{aligned}
 -\mathcal{L}_\Phi u&=\mu &\text{in }&\Omega,\\
 u&=g& \text{in }&\mathbb{R}^n\setminus\Omega.
\end{aligned}
\right.
\end{equation}
Consider first the case when $s \in (0,1)$, $p \in (1,\infty)$, $g\in \mathbb{W}^{s,p}(\Omega)$ and $\mu \in W^{-s,p'}(\Omega)$, where $p'=p/(p-1)$ and the space $\mathbb{W}^{s,p}(\Omega)$ is defined by
\[ \mathbb{W}^{s,p}(\Omega) \coloneqq \left\{ f:\mathbb{R}^{n}\rightarrow \mathbb{R} : f|_{\Omega} \in L^{p}(\Omega), \ \iint_{\mathcal{C}_{\Omega}}\frac{|f(x)-f(y)|^{p}}{|x-y|^{n+sp}}\,dx\,dy < \infty \right\} \]
with $\mathcal{C}_{\Omega} \coloneqq (\mathbb{R}^{n}\times\mathbb{R}^{n})\setminus((\mathbb{R}^{n}\setminus\Omega)\times(\mathbb{R}^{n}\setminus\Omega))$. 
Recalling \eqref{def.L}, \eqref{mono} and \eqref{c2}, we say that $u\in \mathbb{W}^{s,p}(\Omega)$ is a weak solution to \eqref{EQg}  if 
$u$ satisfies
\[
\langle -\mathcal L_\Phi u,\varphi \rangle = \langle \mu, \varphi \rangle
\]
for every $\varphi \in C^{\infty}_{c}(\Omega)$, and  $u=g$ a.e. in $\mathbb R^n\setminus \Omega $. 
Note that, in this setting, $u$ is a weak solution to \eqref{EQg} if and only if it is a minimizer of the following functional
\[  w \mapsto \iint_{\mathcal{C}_{\Omega}}\Psi(w(x)-w(y))K(x,y)\,dx\,dy - \langle \mu, w \rangle \]
over $\{w \in \mathbb{W}^{s,p}(\Omega): w=g \text{ a.e. in } \mathbb{R}^{n}\setminus \Omega \}$, where $\Psi(t) = \int_{0}^{|t|}\Phi(\tau)\,d\tau$. In turn, under assumptions \eqref{mono} and \eqref{c2}, the existence and uniqueness of weak solutions to \eqref{EQg} in the case $\mu\in W^{-s,p'}(\Omega)$ follow from direct methods of the calculus of variations, see for instance \cite[Theorem 2.3]{DicaKuPa2}. For various results concerning such fractional energy functionals, see \cite{BLS18,Co1, DicaKuPa1,DicaKuPa2,IaLiPeSqua1,IaMoSqua1,KoKuPa1,KoKuPa2,KuMiSi2,Va1} and references therein.

On the other hand, in this paper we deal with the situation when $\mu$ is merely a measure, and so $\mu$ does not belong to $W^{-s,p'}(\Omega)$ in general. Thus, we need to consider a weaker notion of solution than the usual weak solution. As mentioned above, there are various notions of solutions to local measure data problems modeled on \eqref{2hvEQ5}, such as entropy solution, renormalized solution, approximable solution, and SOLA. Among these, the notion of SOLA was extended to nonlocal measure data problems of the type \eqref{EQg} by Kuusi, Mingione, and Sire \cite{KuMiSi,KuMiSi3} when $p > 2-s/n$. They proved the existence of SOLA to \eqref{EQg} and obtained pointwise and oscillation estimates via $\mathbf{W}^{T}_{s,p}[\mu]$. Later, similar pointwise estimates were obtained for $-\mathcal{L}_{\Phi}$-superharmonic functions with $1 <p < n/s$ and corresponding nonnegative measures in \cite{KLL23}.  Recently, Gkikas \cite{Gk24} extended the notion of approximable solution to \eqref{EQ} and obtained a sufficient condition for the existence of nonnegative distributional solutions to \eqref{EQp} when $\mu$ is nonnegative, $P(u)=u^{\gamma}$ and $\Omega$ is a bounded domain. Indeed, when $p>2-s/n$, the existence of approximable solutions implies that of SOLA, see \cite[Proposition 2.8]{Gk24}. However, it seems unclear when $1<p\le 2-s/n$ due to the lack of compactness for the fractional Sobolev space $W^{h,q}$ with $h,q\in (0,1)$. 
We also mention that gradient regularity and first-order potential estimates for SOLA to nonlocal measure data problems of the type \eqref{EQg} were recently established in the linear case \cite{KuNoSi} and the nonlinear nondegenerate ($p=2$) case \cite{DKLN}. However, the validity of such gradient regularity results in the case $p\neq2$ is not known, and only fractional estimates are obtained in the case $p > 2$, see \cite{DiNo}.

The aim of this paper is twofold. First, we extend the notion of SOLA to \eqref{EQ} with $1< p< n/s$ and then prove the existence and Wolff potential estimates for SOLA to \eqref{EQ}, thereby completing the low-order regularity theory for nonlocal measure data problems.  
Second, by using such potential estimates, we establish both a necessary condition and a sufficient condition for the existence of nonnegative SOLA to a large class of nonlocal Lane--Emden type equations with nonnegative measure data. Namely, we deal with both power and exponential reaction terms, and both bounded domains and the whole $\mathbb{R}^{n}$. These results are the nonlocal analogs of those obtained for the local Lane--Emden type equation \eqref{EQ:pcase} in \cite{HV1,22PhVe}, and they address a question raised in \cite[Section 8.1]{KuMiSi3}.

\subsection{Wolff potential estimates for nonlocal equations}
We obtain pointwise upper and lower bounds for SOLA to \eqref{EQ} in terms of the Wolff potential of $\mu$ when $1<p<n/s$. 
Note that such estimates were obtained in \cite{KuMiSi} when $p>2-s/n$. In the present paper, we cover the remaining range $1<p \le 2-s/n$. 
To the best of our knowledge, this range has not been treated in the literature on nonlocal equations with general measure data. 

We denote by $\mathcal{M}(\Omega)$ (resp. $\mathcal{M}^{+}(\Omega)$) the set of all Radon measures (resp. nonnegative Radon measures) with finite total mass on $\Omega$. 
Given a measurable function $f:\mathbb{R}^{n}\rightarrow \mathbb{R}$, we denote its nonlocal tail by
\begin{equation*}
\tail(f;x,r) \coloneqq\left(r^{sp}\int_{\mathbb{R}^n\setminus B_r(x)}\frac{|f(y)|^{p-1}}{|y-x|^{n+sp}}\,dy\right)^{1/(p-1)}
\end{equation*} 
whenever $x \in \mathbb{R}^n$ and $r>0$. 
For simplicity, we write $\tail(f;r)=\tail(f;x,r)$ if the point $x$ is not important or clear from the context.

Our first result is a pointwise upper bound for SOLA, which extends \cite[Theorem 1.2]{KuMiSi} to the range $1 < p \le 2-s/n$. 

\begin{theorem}\label{thmupper1}
Let $\mu\in\mathcal{M}(\Omega)$ and let $-\mathcal{L}_{\Phi}$ be defined in \eqref{def.L} under assumptions \eqref{mono}, \eqref{c2} and \eqref{s.p.range}. Let $u$ be a SOLA to \eqref{EQ}. 
Then there exists a constant $c=c(n,s,p,\Lambda)>0$ such that
\begin{equation*}
|u(x)|\leq c{\bf W}^{R}_{s,p}[\mu](x) + c\left(\fint_{B_{R}(x)}|u|^{p-1}\,d\tilde{x}\right)^{1/(p-1)} + c\tail(u;x,R)
\end{equation*}
holds whenever $B_{R}(x)\subset\Omega$ and the right-hand side is finite.
\end{theorem}

We next introduce a pointwise lower bound for nonnegative SOLA, for which we recall  \cite[Theorem 1.3]{KuMiSi}. We remark that the paper \cite{KuMiSi} primarily assumes $p > 2- s/n$ in the definition of SOLA \cite[Definition 2]{KuMiSi}. This assumption is crucial for the proof of \cite[Theorem 1.2]{KuMiSi}, but it is not a prerequisite for the proof of \cite[Theorem 1.3]{KuMiSi}. 
In fact, an inspection of the proof of \cite[Theorem 1.3]{KuMiSi} shows that if $u$ is a nonnegative weak solution to \eqref{EQ}, then estimate \eqref{lower.wolff} holds for any $1<p<n/s$. 
Therefore, once we show the existence of SOLA to \eqref{EQ} (see Theorem \ref{thm.ex.bdd}), the argument in the proof \cite[Theorem 1.3]{KuMiSi} remains valid for the full range $1< p < n/s$ without any modification. 

\begin{theorem}\label{thmlower}
Let $\mu\in\mathcal{M}^{+}(\Omega)$ and let $-\mathcal{L}_{\Phi}$ be defined in \eqref{def.L} under assumptions \eqref{mono}, \eqref{c2} and \eqref{s.p.range}. 
Let $u$ be a nonnegative SOLA to \eqref{EQ} such that the approximating sequence $\{\mu_{j}\}$ for $\mu$ as described in Definition \ref{defsola} or Definition \ref{defsola2} is made of nonnegative functions. 
Then there exists a constant $c=c(n,s,p,\Lambda) \ge 1$ such that  
\begin{equation}\label{lower.wolff}
c^{-1}{\bf W}_{s,p}^{R/8}[\mu](x)\leq u(x)
\end{equation}
holds whenever $B_{R}(x)\subset\Omega$ and ${\bf W}^{R/8}_{s,p}[\mu](x)$ is finite. 
\end{theorem}

We will obtain Theorem \ref{thmupper1} as a consequence of another new result, which is an oscillation estimate for SOLA. Note that, when $\mu \equiv 0$, the following theorem gives back the H\"older estimates for the homogeneous equation $-\mathcal{L}_{\Phi}v=0$ obtained in \cite{Co1,DicaKuPa2}.
\begin{theorem}\label{thmosc}
Let $\mu \in \mathcal{M}(\Omega)$ and let $-\mathcal{L}_{\Phi}$ be defined in \eqref{def.L} under assumptions \eqref{mono}, \eqref{c2} and \eqref{s.p.range}. Let $u$ be a SOLA to \eqref{EQ}, and let $B_{R}(x_{0}) \subset \Omega$ be a ball. Then there exists a constant $c>0$ such that
\begin{equation}\label{univ.wolff}
\begin{aligned}\hspace{-0.5mm}
|u(x)-u(y)| & \le c\left[\mathbf{W}^{R}_{s-\alpha(p-1)/p,p}[\mu](x) + \mathbf{W}^{R}_{s-\alpha(p-1)/p,p}[\mu](y)\right]|x-y|^{\alpha} \\
& \quad + c\left[\left(\fint_{B_{R}(x_{0})}|u|^{p-1}\,d\tilde{x}\right)^{1/(p-1)} + \tail(u;x_{0},R)\right] \left(\frac{|x-y|}{R}\right)^{\alpha}
\end{aligned}
\end{equation}
holds for a.e. $x,y \in B_{R/8}(x_{0})$, provided the right-hand side is finite and $0 \le \alpha < \alpha_{0}$. Here, $\alpha_{0} = \alpha_{0}(n,s,p,\Lambda) \in (0,\min\{sp/(p-1),1\})$ is the H\"older continuity exponent for the homogeneous equation 
\[ -\mathcal{L}_{\Phi}v = 0 \quad \text{in}\;\; \Omega, \]
see Lemma \ref{lem.hol} below. Moreover, whenever $\tilde{\alpha} \in [0,\alpha_{0})$ is fixed, the dependence of the constant $c$ is uniform for $\alpha \in [0,\tilde{\alpha}]$, in the sense that $c$ depends only on $n$, $s$, $p$, $\Lambda$ and $\tilde{\alpha}$.
\end{theorem}

Theorems \ref{thmupper1} and \ref{thmosc}, along with the well-known mapping properties of Wolff potentials, yield the following corollary.

\begin{corollary}
Let $\mu \in \mathcal{M}(\Omega)$ and let $-\mathcal{L}_{\Phi}$ be defined in \eqref{def.L} under assumptions \eqref{mono}, \eqref{c2} and \eqref{s.p.range}. Let $u$ be a SOLA to \eqref{EQ}.
\begin{itemize}
\item[(1)] If $1<q<n/(sp)$, then
\[ \mu \in L^{q}_{\mathrm{loc}}(\Omega) \;\; \Longrightarrow \;\; u \in L^{nq(p-1)/(n-spq)}_{\mathrm{loc}}(\Omega). \]
\item[(2)] If $0 < \alpha < \alpha_{0}$, then
\[ \mu \in L^{n/(sp-\alpha(p-1)),\infty}_{\mathrm{loc}}(\Omega) \;\; \Longrightarrow \;\; u \in C^{0,\alpha}_{\mathrm{loc}}(\Omega). \]
\end{itemize}
\end{corollary}

In the proof of Theorem \ref{thmosc}, we further have a borderline continuity criterion, which extends \cite[Theorem 1.5]{KuMiSi} and \cite[Theorem 8.15]{KuMiSi3} to the case $1<p\le2-s/n$. Its proof is completely similar to that of \cite[Theorem 8.15]{KuMiSi3} once we obtain the excess decay estimate in Lemma \ref{excess.decay} below.

\begin{corollary}
Let $\mu \in \mathcal{M}(\Omega)$ and let $-\mathcal{L}_{\Phi}$ be defined in \eqref{def.L} under assumptions \eqref{mono}, \eqref{c2} and \eqref{s.p.range}. Let $u$ be a SOLA to \eqref{EQ}. If 
\[ \lim_{R\rightarrow0}\mathbf{W}^{R}_{s,p}[\mu](x) = 0 \quad \text{locally uniformly in $\Omega$ with respect to $x$,} \]
then $u$ is continuous in $\Omega$. In particular, if one of the following two conditions holds:
\begin{itemize}
\item[(i)] $\mu \in L^{n/(sp),1/(p-1)}_{\mathrm{loc}}(\Omega)$,
\item[(ii)] $|\mu|(B_{r}) \le h(r)r^{n-sp}$ for any ball $B_{r}\subset \Omega$, with $h(\cdot)$ satisfying 
\[ \int_{0} [h(r)]^{1/(p-1)}\frac{dr}{r} < \infty,\]
\end{itemize}
then $u$ is continuous in $\Omega$.
\end{corollary}

The next two corollaries are concerned with global pointwise estimates for nonnegative SOLA without tail terms. They are based on a comparison principle argument under the additional assumption \eqref{mono2}, and will be used in the existence results for nonlocal Lane--Emden type problems. 
The first one deals with \eqref{EQ} in a bounded domain, while the second one considers the same equation in the whole $\mathbb{R}^n$.

\begin{corollary}\label{2hvTH4}
Let $\Omega \subset \mathbb{R}^{n}$ be a bounded domain, $\mu \in \mathcal{M}^{+}(\Omega)$ and let $-\mathcal{L}_{\Phi}$ be defined in \eqref{def.L} under assumptions \eqref{mono}--\eqref{s.p.range}. 
Let $u$ be a nonnegative SOLA to \eqref{EQ} such that the approximating sequence $\{\mu_{j}\}$ for $\mu$ as described in Definition \ref{defsola} is made of nonnegative functions.  Then there exists a constant $C_{0}\geq1$, depending only on $n$, $s$, $p$ and $\Lambda$, such that 
\begin{equation}\label{global.wolff}
\frac{1}{C_{0}}{\bf W}^{\mathrm{dist}(x,\partial\Omega)/8}_{s,p}[\mu](x) \le u(x) \le C_{0} {\bf W}^{2\mathrm{diam}(\Omega)}_{s,p}[\mu](x) 
\end{equation}
holds for a.e. $x \in \Omega$, whenever ${\bf W}^{2\mathrm{diam}(\Omega)}_{s,p}[\mu](x)$ is finite.
\end{corollary}

\begin{corollary}\label{2hvTH4-}
Let $\mu \in \mathcal{M}^{+}(\mathbb{R}^{n})$ and let $-\mathcal{L}_{\Phi}$ be defined in \eqref{def.L} under assumptions \eqref{mono}--\eqref{s.p.range}. 
Let $u$ be a nonnegative SOLA to \eqref{EQ} with $\Omega=\mathbb{R}^{n}$ such that the approximating sequence $\{\mu_{j}\}$ for $\mu$ as described in Definition \ref{defsola2} is made of nonnegative functions.  Then, with the constant $C_{0}\geq1$ determined in Corollary \ref{2hvTH4},
\begin{equation}\label{belowesRn}
\frac{1}{C_{0}}{\bf W}_{s,p}[\mu](x) \leq u(x) \leq C_{0}{\bf W}_{s,p}[\mu](x) 
\end{equation} 
holds for a.e. $x \in \mathbb{R}^{n}$, whenever ${\bf W}_{s,p}[\mu](x)$ is finite. 
\end{corollary} 

As for the oscillation estimate, by letting $R\rightarrow\infty$ in \eqref{univ.wolff}, we have the following:

\begin{corollary}
Let $\mu \in \mathcal{M}(\mathbb{R}^{n})$ and let $-\mathcal{L}_{\Phi}$ be defined in \eqref{def.L} under assumptions \eqref{mono}, \eqref{c2} and \eqref{s.p.range}. Let $u$ be a SOLA to \eqref{EQ} with $\Omega = \mathbb{R}^{n}$. Then, with the same constants $c$ and $\alpha_0$ as in Theorem \ref{thmosc},
\begin{equation*}
|u(x)-u(y)| \le c\left[\mathbf{W}_{s-\alpha(p-1)/p,p}[\mu](x) + \mathbf{W}_{s-\alpha(p-1)/p,p}[\mu](y)\right]|x-y|^{\alpha} 
\end{equation*}
holds for a.e. $x,y \in \mathbb{R}^{n}$, provided the right-hand side is finite and $0 \leq \alpha < \alpha_{0}$.
\end{corollary}

\subsection{Nonlocal equations of Lane--Emden type}\label{sec:LE}

\noindent We now introduce the existence results for the nonlocal Lane--Emden type problem \eqref{EQp} under the additional assumption \eqref{mono2}, where $\mu$ is a nonnegative measure and $\Omega$ is either a bounded domain or the whole $\mathbb{R}^{n}$.

We start with  existence results for the case $P(u) = u^{\gamma}$, see \eqref{capa1} and \eqref{capa2} for the definitions of $\mathrm{Cap}_{\mathbf{G}_{s p},\gamma/(\gamma-p+1)}$ and $\mathrm{Cap}_{\mathbf{I}_{s p},\gamma/(\gamma-p+1)}$, respectively.

\begin{theorem}\label{2hvMT3}
Let $-\mathcal{L}_{\Phi}$ be defined in \eqref{def.L} under assumptions \eqref{mono}--\eqref{s.p.range},  and let $\gamma>p-1$. Let $\Omega\subset \mathbb{R}^n$ be a bounded domain and $\mu\in \mathcal M^{+}(\Omega)$.
\begin{itemize}
\item[ (1)] 	If the following problem
\begin{equation}\label{2hvMT3-2}
\left\{
\begin{aligned}
- \mathcal{L}_\Phi u& = u^\gamma+\mu &\text{in }&\Omega,  \\ 
u & =  0 &\text{in }&\mathbb{R}^n\setminus\Omega
\end{aligned}
\right. 
\end{equation} admits a nonnegative  SOLA $u$ such that the approximating sequence $\{\mu_{j}\}$ for $\mu$ as described in Definition \ref{defsola} is made of nonnegative functions, then for any compact set $K\subset \Omega$, there exists a constant $C>0$, depending only on $n$, $s$, $p$, $\Lambda$, $\gamma$ and $\mathrm{dist}(K,\partial\Omega)$, such that 
\begin{equation}\label{2hvMT3-4}
\int_{E} u^\gamma \,dx  + \mu (E) \le C \mathrm{Cap}_{\mathbf{G}_{s p},\frac{\gamma}{\gamma-p+1}}(E) \quad \text{for any Borel set $E\subset K$.}
\end{equation}
\item[ (2)] Conversely, there exists a small constant $\delta>0$, depending only on $n$, $s$, $p$, $\Lambda$, $\gamma$ and $\mathrm{diam}(\Omega)$, such that if the inequality 
\begin{equation}\label{2hvMT3-1}
\mu(K)\leq \delta\mathrm{Cap}_{\mathbf{G}_{sp},\frac{\gamma}{\gamma-p+1}}(K)
\end{equation}
holds for any compact set $K\subset\mathbb{R}^n$, then  problem \eqref{2hvMT3-2} admits a nonnegative SOLA $u$ which satisfies  
\begin{equation}\label{2hvMT3-3}
u(x) \le \frac{\gamma \max\left\{2^{\frac{2-p}{p-1}},1\right\}}{\gamma-p+1}C_{0} {\bf W}_{s ,p}^{2\mathrm{diam}(\Omega)}[\mu ](x) \quad \text{for a.e. $x\in \Omega$,}
\end{equation}
where $C_{0} \ge 1$ is the constant determined in Corollary \ref{2hvTH4}.
\end{itemize}
\end{theorem}

\begin{theorem}\label{2hvMT4}
Let $-\mathcal{L}_{\Phi}$ be defined in \eqref{def.L} under assumptions \eqref{mono}--\eqref{s.p.range}, and let $\gamma>p-1$.  Let  $\mu \in \mathcal{M}^{+}(\mathbb{R}^{n})$.	
\begin{itemize}
\item[ (1)] If the following problem
\begin{equation}\label{2hvMT4-2}
\left\{
\begin{aligned}
- \mathcal{L}_\Phi u &=  u^\gamma+\mu & \text{in }& \mathbb{R}^n,\\ 
\inf_{\mathbb{R}^n}u & =0 & & 
\end{aligned} 
\right.
\end{equation} 
admits a nonnegative SOLA $u$ such that the approximating sequence $\{\mu_{j}\}$ for $\mu$ as described in Definition \ref{defsola2} is made of nonnegative functions, then  there exists a constant $C>0$, depending only on $n$, $s$, $p$, $\Lambda$ and $\gamma$, such that 
\begin{equation}\label{2hvMT4-4}
\int_{E}u^{\gamma}\,dx  + \mu (E) \le C \mathrm{Cap}_{\mathbf{I}_{s p},\frac{\gamma}{\gamma-p+1}}(E) \quad \text{for any Borel set $E\subset \mathbb{R}^n$.}
\end{equation}
\item[ (2)] Conversely, there exists a small constant $\delta>0$, depending only on $n$, $s$, $p$, $\Lambda$ and $\gamma$,  such that if the inequality 
\begin{equation}\label{2hvMT4-1}
\mu(K)\leq \delta\mathrm{Cap}_{\mathbf{I}_{s p},\frac{\gamma}{\gamma-p+1}}(K)
\end{equation}
holds for any compact set $K\subset\mathbb{R}^n$, then problem \eqref{2hvMT4-2} admits a nonnegative SOLA $u$ which satisfies  
\begin{equation}\label{2hvMT4-3}
{u(x)} \le \frac{\gamma \max\left\{2^{\frac{2-p}{p-1}},1\right\}}{\gamma-p+1}C_{0}{\bf W}_{s ,p}[\mu ](x) \quad \text{for a.e. $x\in \mathbb R^n$,}
\end{equation}
where $C_{0} \ge 1$ is the constant determined in Corollary \ref{2hvTH4}.
\end{itemize}
\end{theorem}

We next state the results for the case when $P=P_{l,a,\beta}$ is an exponential function given in \eqref{2hvEQ11}, see \eqref{2hvEQ9}, \eqref{2hvEQ14} and \eqref{2hvEQ15} for the definitions of ${\bf M}_{sp,R}$, $\mathrm{Cap}_{{\mathbf{G}_{sp}},Q_p^*}$ and $\mathrm{Cap}_{{\mathbf{I}_{sp}},Q_p^*}$, respectively. In this case, we need an extra assumption that $\mu$ has compact support in $\R^{n}$.

\begin{theorem}\label{2hvMT1}
Let $-\mathcal{L}_{\Phi}$ be defined in \eqref{def.L} under assumptions \eqref{mono}--\eqref{s.p.range}. Let $l\in \mathbb{N}$, $a>0$,  $\beta\geq 1$ with $l\beta>p-1$. Let $\Omega\subset \mathbb{R}^n$ be a bounded domain and $\mu\in \mathcal M^{+}(\Omega)$.
\begin{itemize}
\item[ (1)]
If the following problem
\begin{equation}\label{2hvMT1a}
\left\{
\begin{aligned}
- \mathcal{L}_\Phi u & =  P_{l,a,\beta}(u)+\mu &\text{in } & \Omega,  \\ 
u  & =  0 & \text{in }& \mathbb{R}^n\setminus\Omega
\end{aligned}
\right.
\end{equation} 
admits a nonnegative SOLA  $u$ such that the approximating sequence $\{\mu_{j}\}$ for $\mu$ as described in Definition \ref{defsola} is made of nonnegative functions, then for any compact set $K\subset \Omega$, there exists a constant $C>0$, depending only on $n$, $s$, $p$, $\Lambda$, $l$, $a$, $\beta$ and $\mathrm{dist}(K,\partial\Omega)$, such that 
\begin{equation}\label{2hvMT1c}
\int_{E} P_{l,a,\beta}(u)\,dx  + \mu (E) \le C \mathrm{Cap}_{{\mathbf{G}_{ sp}},Q_p^*}(E) \quad \text{for any Borel set $E\subset K$.}
\end{equation}
\item[ (2)]
Conversely, there exists a small constant $\delta>0$, depending only on $n$, $s$, $p$, $\Lambda$, $l$, $a$, $\beta$ and $\mathrm{diam}(\Omega)$, such that if 
\begin{equation}\label{2hvMT3-1-}
\left\|{\bf M}_{sp,2\mathrm{diam}(\Omega)}^{\frac{{(p - 1)(\beta  - 1)}}{\beta }}[\mu ]\right\|_{L^\infty({\mathbb{R}^n})} \le \delta,
\end{equation}
then, with $c_p=\max\{1, 4^{(2-p)/(p-1)}\}$, $C_{0} \ge 1$ determined in Corollary \ref{2hvTH4} and
 \[ \omega_{1} \coloneqq \delta\left\|{\bf M}_{ sp,2\mathrm{diam}(\Omega)}^{\frac{(p-1)(\beta-1)}{\beta}} [1] \right\|_{{L^\infty }(\mathbb{R}^n)}^{-1}+\mu, \]  
we have 
\begin{equation*}
\int_\Omega P_{l,a,\beta}(4c_pC_{0}{\bf W}_{s ,p}^{2\mathrm{diam}(\Omega)}[\omega_{1}])\, dx\leq C
\end{equation*}
for a constant $C=C(n,s,p,\beta, \mathrm{diam}(\Omega))>0$, 
and problem \eqref{2hvMT1a} admits a nonnegative SOLA $u$ satisfying
\begin{equation}\label{2hvMT1b}
{u(x)} \le 2c_pC_{0}{\bf W}_{s ,p}^{2\mathrm{diam}(\Omega)}[\omega_{1}](x) \quad \text{for a.e. $x\in \Omega$.}
\end{equation}
\end{itemize}
\end{theorem}

\begin{theorem}\label{2hvMT2}
Let $-\mathcal{L}_{\Phi}$ be defined in \eqref{def.L} under assumptions \eqref{mono}--\eqref{s.p.range}. Let $l\in \mathbb{N}$, $a>0$,  $\beta\geq 1$ with $l\beta>\frac{n (p-1)}{n-sp}$.  
Let $\mu \in \mathcal{M}^{+}(\mathbb{R}^n)$ with $\mathrm{supp}\,\mu\subset B_{R}(0)$ for some $R>1$.
\begin{itemize}
\item[(1)] If  the following problem 
\begin{equation}\label{2hvMT2a}
\left\{
\begin{aligned}
- \mathcal{L}_\Phi u & = P_{l,a,\beta}(u)+\mu & \text{in }&  \mathbb{R}^n,\\
\inf_{\mathbb{R}^n}u & = 0 &&
\end{aligned}
\right. 
\end{equation} 
admits a nonnegative SOLA $u$ such that the approximating sequence $\{\mu_{j}\}$ for $\mu$ as described in Definition \ref{defsola2} is made of nonnegative functions, then there exists a constant $C>0$, depending only on $n$, $s$, $p$, $\Lambda$, $l$, $a$ and $\beta$, such that 
\begin{equation}\label{2hvMT2c}
\int_{E} P_{l,a,\beta}(u)\,dx  + \mu (E) \le C \mathrm{Cap}_{{\mathbf{I}_{ s p}},Q_p^*}(E) \quad \text{for any Borel set $E \subset \mathbb{R}^n$.}  
\end{equation}
\item[(2)] Conversely, there exists a small constant $\delta>0$, depending only on $n$, $s$, $p$, $\Lambda$, $l$, $a$, $\beta$ and $R$, such that if 
\begin{equation}\label{2hvMT3-1-+}
\left\|{\bf M}_{ sp}^{\frac{{(p - 1)(\beta  - 1)}}{\beta }}[\mu ]\right\|_{L^\infty (\mathbb{R}^n)} \le \delta,
\end{equation}
then, with $c_p=\max\{1, 4^{(2-p)/(p-1)}\}$, $C_{0} \ge 1$ determined in Corollary \ref{2hvTH4} and
\[ \omega_{2} \coloneqq \delta\left\|{\bf M}_{s p}^{\frac{(p-1)(\beta-1)}{\beta}}[\chi_{B_R}]\right\|_{L^\infty(\mathbb{R}^n)}^{-1}\chi_{B_R}+\mu, \] 
we have 
\begin{equation*}
\int_{\mathbb{R}^n}	P_{l,a,\beta}\left(4c_pC_{0}{\bf W}_{s ,p}[\omega_{2}]\right)\, dx\leq C
\end{equation*}
for a constant $C=C(n,s,p,\beta,R)>0$, 
and problem \eqref{2hvMT2a} admits a nonnegative SOLA $u$ satisfying
\begin{equation}\label{2hvMT2b}
{u(x)} \le 2c_pC_{0}{\bf W}_{s ,p}[\omega_{2}](x) \quad \text{for a.e. $x\in \mathbb{R}^n.$}
\end{equation}
\end{itemize}
\end{theorem}

\subsection{Novelties and techniques}
In this paper, we provide a new approach to obtain Wolff potential estimates for the nonlocal measure data problem \eqref{EQ} in the subquadratic case; the main novelty is especially related to the range $1 < p \le 2-s/n$. Here we briefly explain our approach, with emphasis on considerable differences compared to those available in the literature. 

In the case of local measure data problems like \eqref{2hvEQ5} with $1 < p \le 2-1/n$, solutions (in any sense) do not in general belong to the Sobolev space $W^{1,1}$. Thus, usual excess functionals are not available in this case. Moreover, it is well known that Sobolev--Poincar\'e type inequalities for the Sobolev space $W^{1,q}$ do not hold in general when $q \in (0,1)$, see for instance \cite{BK94}. This is a main difficulty in establishing comparison estimates between \eqref{2hvEQ5} and related homogeneous problems \cite{QH1,QH5,NP0}. 

On the other hand, in this paper we show a somewhat surprising fact that Sobolev--Poincar\'e type inequalities and compact embeddings for the space $W^{h,q}$, $h \in (0,1)$, continue to hold for $q \in (0,1)$, see Lemmas \ref{lemsobopoin} and \ref{lemcompact} below. This is a notable difference between classical and fractional Sobolev spaces. 
However, to our knowledge, such results in a precise form do not appear in the literature. We believe that they are of independent interest and can be applied to other topics in nonlocal problems.

We also emphasize that, in the proof of the Wolff potential estimates presented in Theorems \ref{thmupper1} and \ref{thmosc}, we employ new local and nonlocal excess functionals given in \eqref{loc.exc} and \eqref{excess.functional} below, respectively. In particular, our nonlocal excess functional is not comparable to the one considered in \cite{KuMiSi} and is motivated from the modified (local) excess functionals used to prove potential estimates for singular $p$-Laplace type equations \cite{DZ,NP,NP2}. 

With all these ingredients at hand, we can extend the basic comparison estimates and excess decay estimates in \cite{KuMiSi,KuMiSi3} to the case $1 < p \le 2-s/n$, thereby showing the existence and potential estimates for SOLA to \eqref{EQ} in this case as well. These results completely extend the Wolff potential estimates for local $p$-Laplace type equations, obtained in \cite{DM10,DM11,22KiMa1,22KiMa2,KM12,QH5,NP,NP2}, to corresponding nonlocal equations. Moreover, by applying global pointwise estimates for \eqref{EQ}, modifying the approaches in \cite{HV1,22PhVe} and then performing a delicate two-step approximation procedure, we finally show the existence and potential estimates for SOLA to \eqref{EQp}.

\section{Preliminaries}
\subsection{Fractional Sobolev spaces}
Let $U \subseteq \mathbb{R}^{n}$ be an open set. For $h\in(0,1)$ and $q \in (0,\infty)$, we say that a function $f$ belongs to the fractional Sobolev space $W^{h,q}(U)$ if and only if
\[
\|f\|_{W^{h,q}(U)}^q  \coloneqq \int_{U} |f|^q\,dx +\int_{U}\int_{U} \frac{|f(x)-f(y)|^q}{|x-y|^{n+hq}}\,dx\,dy<\infty.
\]
Note that when $q \in (0,1)$, $\|\cdot\|_{W^{h,q}(U)}$ is not a norm but a quasinorm.  
Nevertheless, we use this notation for any $q \in (0,\infty)$. Also, we say that a sequence $\{f_{j}\} \subset W^{h,q}(U)$ converges to $f$ in $W^{h,q}(U)$ if $f\in W^{h,q}(U)$ and $\|f_{j}-f\|_{W^{h,q}(U)} \rightarrow 0$. 

Embedding theorems and Sobolev--Poincar\'e type inequalities for the space $W^{h,q}$ with $q\geq 1$ can be found in \cite[Section 6]{DiPaVa1}, see also \cite[Section 4]{Co1}. Surprisingly, they continue to hold for the case $q<1$ as well, which is a difference between classical and fractional Sobolev spaces. 
To show this, we first recall an elementary inequality.
\begin{lemma} 
Let $q \in (0,1)$. There exists a constant $c(q)>0$ such that
\begin{equation}\label{elementaryineq}
\left||a|^{q-1}a-|b|^{q-1}b\right|\leq c(q)|a-b|^{q} \quad  \text{for any}\;\; a,b\in \mathbb{R}\setminus\{0\}.
\end{equation}
\end{lemma}

Using this, we obtain the following:
\begin{lemma} \label{lemsobopoin}
Let $h \in (0,1)$, $q \in (0,n/h)$ and define $q^{*}_{h} \coloneqq nq/(n-hq)$. If $f\in W^{h,q}(\mathbb{R}^{n})$, then we have
\begin{equation}\label{sobo}
\left(\int_{\mathbb R^n} |f|^{q^{*}_{h}}\,dx\right)^{1/q^{*}_{h}}  \leq c \left(\int_{\mathbb R^n}\int_{\mathbb R^n} \frac{|f(x)-f(y)|^q}{|x-y|^{n+hq}}\,dx\,dy\right)^{1/q}
\end{equation}
for a constant $c=c(n,h,q)>0$.
In particular, if $f\in W^{h,q}(\mathbb R^n)$ satisfies $f=0$ a.e. in $\mathbb{R}^{n}\setminus B_r$, then we have 
\begin{equation}\label{sobo0}
\left(\fint_{B_r} |f|^{q^{*}_{h}}\,dx\right)^{1/q^{*}_{h}} \leq \frac{c  r^h}{(t-1)^{1/q}}\left(\int_{B_{tr}}\fint_{B_{tr}} \frac{|f(x)-f(y)|^q}{|x-y|^{n+hq}}\,dx\,dy\right)^{1/q}
\end{equation}
for any $t \in (1,2]$, where $c=c(n,h,q)>0$.
\end{lemma}
\begin{proof}
We may consider the case $q \in (0,1)$ only. Then note that $hq \in (0,1)$. Applying \eqref{elementaryineq} and the fractional Sobolev--Poincar\'e inequality \cite[Theorem 6.5]{DiPaVa1} to $|f|^{q-1}f \in W^{hq,1}(\mathbb{R}^{n})$, we have \eqref{sobo} as follows:
\begin{align*}
& \left(\int_{\mathbb{R}^n}|f|^{q^{*}_{h}}\,dx\right)^{1/q^{*}_{h}} = \left(\int_{\mathbb{R}^{n}}\left||f|^{q-1}f\right|^{n/(n-hq)}\,dx\right)^{(n-hq)/nq} \\
& \le c\left(\int_{\mathbb{R}^{n}}\int_{\mathbb{R}^{n}}\frac{\left||f(x)|^{q-1}f(x)-|f(y)|^{q-1}f(y)\right|}{|x-y|^{n+hq}}\,dx\,dy\right)^{1/q} \\
& \le c\left(\int_{\mathbb{R}^{n}}\int_{\mathbb{R}^{n}}\frac{|f(x)-f(y)|^{q}}{|x-y|^{n+hq}}\,dx\,dy\right)^{1/q}.
\end{align*}
In the same way, we can also obtain \eqref{sobo0} (see for instance \cite[Lemma 4.7 and Corollary 4.9]{Co1}).
\end{proof}

We also extend the compact embedding result for $W^{h,q}$ to the case $q<1$.
\begin{lemma}\label{lemcompact}
Let $U$ be a bounded Lipschitz domain, $h\in(0,1)$ and $q \in (0,n/h)$. Then for any $\tilde{q} \in [q,q^{*}_{h})$, the  embedding of $W^{h,q}(U)$ into $L^{\tilde{q}}(U)$ is compact.
\end{lemma}
\begin{proof}
The proof of the lemma in the case $q \ge 1$ can be found in \cite[Section 7]{DiPaVa1}; we thus consider the case $q \in (0,1)$ only. Assume that $\{f_j\}$ is bounded in $W^{h,q}(U)$. Then \eqref{elementaryineq} implies that $\{|f_j|^{q-1}f_j\}$ is bounded in $W^{hq,1}(U)$. Therefore, we can extract from $\{|f_j|^{q-1}f_j\}$ a subsequence, still denoted by $\{|f_j|^{q-1}f_j\}$, which converges to  a function $g$ in $L^{t}(U)$ for any $1 \le t < n/(n-hq)$. 
We now set $f \coloneqq |g|^{(1-q)/q}g$, i.e., $g=|f|^{q-1}f$. Then we see that $\{f_j\}$ converges to $f$ a.e. in $U$ and moreover
\[
\lim_{j\to\infty}\int_{U} |f_j|^{qt}\,dx= \int_{U} |f|^{qt}\,dx. 
\]
Hence, Lebesgue's dominated convergence theorem implies that $\{f_j\}$ converges to $f$ in $L^{\tilde q}(U)$  for any $q \leq \tilde q < nq/(n-hq) = q^{*}_{h}$.  
\end{proof}

The following inequality can be obtained by using H\"older's inequality, see for instance \cite[Lemma 4.6]{Co1}.
\begin{lemma}\label{lemembedding}
Let $\Omega'\subseteq\Omega\subset \mathbb R^n$ be two bounded open sets, $0<h<s<1$ and $0<q<p<\infty$. Then we have
\begin{align*}
& \left(\int_\Omega\int_{\Omega'}\frac{|f(x)-f(y)|^{q}}{|x-y|^{n+hq}}\,dx\,dy\right)^{1/q} \\
& \leq  c|\Omega'|^{\frac{p-q}{pq}}[\mathrm{diam}(\Omega)]^{s-h} \left(\int_\Omega\int_{\Omega'}\frac{|f(x)-f(y)|^{p}}{|x-y|^{n+sp}}\,dx\,dy\right)^{1/p}
\end{align*}
for some $c=c(n,s,h,p,q)>0$. In particular, $W^{s,p}(\Omega)\subset W^{h,q}(\Omega)$.
\end{lemma}

\subsection{Auxiliary results for potentials and capacities}

Let $\mu \in \mathcal{M}(\mathbb{R}^n)$. 
For $s>0$, $1<p<n/s$ and $0<T\leq\infty$, we define the {\it $T$-truncated Wolff potential with order $s$} of $\mu$ by
\begin{equation}\label{2hvEQ8}
{\bf W}^T_{s,p}[\mu](x)=\int_{0}^{T}\left[\frac{|\mu| (B_t(x))}{t^{n-sp}}\right]^{1/(p-1)}\frac{dt}{t},
\end{equation}
the {\it $T$-truncated Riesz potential with order $s$} of $\mu$  by
\begin{equation*}
{\bf I}^T_{s}[\mu](x)=\int_{0}^{T}\frac{|\mu|(B_t(x))}{t^{n-s}}\frac{dt}{t},
\end{equation*}
and, for $\eta\geq0$, the {\it $T$-truncated $\eta$-fractional maximal function (with order $s$)} of $\mu$ by
\begin{equation}\label{2hvEQ9}
{\bf M}^\eta_{s,T}[\mu](x)=\sup_{0 < t \le T} \frac{|\mu|(B_t(x))}{t^{n-s}h_\eta(t)},
\end{equation}
where $h_\eta(t)=(-\ln t)^{-\eta}\chi_{(0,2^{-1}]}(t)+(\ln 2)^{-\eta}\chi_{[2^{-1},\infty)}(t)$. When $\eta=0$, we have $h_\eta=1$ and in this case we denote by ${\bf M}_{s,T}[\mu]$ the corresponding {\it $T$-truncated fractional maximal function} of $\mu$. When $T=\infty$, we denote  by ${\bf W}_{s,p}[\mu]$ (resp. ${\bf I}_{s}[\mu]$, ${\bf M}_{s}^\eta[\mu]$)  the {\it ($\infty$-truncated) Wolff potential (resp. Riesz potential, $\eta$-fractional maximal function)} of $\mu$. 
When $\mu$ is defined only on an open subset $\Omega\subset\mathbb{R}^n$, we naturally extend it to $\R^n$ by letting $|\mu|(\R^n \setminus \Omega) = 0$.

The following lemma shows a relationship between fractional maximal functions and potentials, see \cite[Lemma 4.1]{KM12}.
\begin{lemma}\label{max.ftn.pot}
Let $\mu \in \mathcal{M}(\Omega)$. Let $\sigma \in (0,1)$, $s \in [0,n]$, $p>1$, and $B_{T}(x) \subset \Omega$. Then we have
\begin{equation*}
\left[\mathbf{M}_{s,\sigma T}[\mu](x)\right]^{1/(p-1)} \le \frac{\max\{\sigma^{(s-n)/(p-1)},1\}}{-\log\sigma}\mathbf{W}^{T}_{s/p,p}[\mu](x)
\end{equation*}
and
\begin{equation*}
\mathbf{M}_{s,\sigma T}[\mu](x) \le \frac{\max\{\sigma^{s-n},1\}}{-\log \sigma}\mathbf{I}^{T}_{s}[\mu](x).
\end{equation*}
\end{lemma}

Let $P\in C([0,\infty))$ be a nondecreasing positive function and $s>0$. For each Borel set $E\subset\mathbb R^n$, we define the $(s,P)$-Orlicz--Bessel capacity by  
\begin{equation}\label{capa1}
	\mathrm{Cap}_{\mathbf{G}_s,P}(E) = \inf \left\{ \int_{\mathbb{R}^n} P(f)\,dx : \mathbf{G}_s \ast f \geq \chi _E,\ f \geq 0,\ P(f) \in {L^1}(\mathbb{R}^n) \right\},
\end{equation}
and the $(s,P)$-Orlicz--Riesz capacity by
\begin{equation}\label{capa2}
\mathrm{Cap}_{\mathbf{I}_s,P}(E) = \inf \left\{ \int_{\mathbb{R}^n} P(f)\,dx:\mathbf{I}_s \ast f \geq \chi _E,\ f \geq 0,\ P(f) \in {L^1}(\mathbb{R}^n) \right\},
\end{equation}
where $\mathbf{G}_{s}(x)=\mathcal{F}^{-1}((1+|\cdot|^2)^{-s/2})(x)$ and $\mathbf{I}_{s}(x)=(n-s)^{-1}|x|^{-(n-s)}$. For more details, see \cite[Section 2.6]{22AH}. When $P(t)=t^p$, we simply write $\mathrm{Cap}_{\mathbf{G}_s,p}=\mathrm{Cap}_{\mathbf{G}_s,P}$  
and 
$\mathrm{Cap}_{\mathbf{I}_s,p}=\mathrm{Cap}_{\mathbf{I}_s,P}$. We also define the Bessel potential with order $s$ of $\mu$ by $\mathbf{G}_{s}[\mu] \coloneqq \mathbf{G}_{s}\ast \mu$.

For $l\in\mathbb{N}$,  we consider the {\it $l$-truncated exponential function}
\begin{equation*}
H_{l}(t)=e^{t}-\sum_{j=0}^{l-1}\frac{t^j}{j!}.
\end{equation*}
Then, for $a>0$ and $\beta\geq 1$, we set
\begin{equation}
\label{2hvEQ11}
P_{l,a,\beta}(t)=H_l(a|t|^{\beta-1}t)
\end{equation}
and
\begin{equation*}
Q_{p}(t) = 
\begin{cases}
    \sum\limits_{q = l}^\infty  \frac{1}{q!}\left(\frac{t}{q}\right)^{\beta q/(p-1)} \quad&\text{if }\,p \ne 2, \\ 
    H_{l}(t^{\beta})&\text{if }\,p = 2,  
 \end{cases}
 \qquad t\ge 0.
\end{equation*}
As usual, the complementary function of $Q_p$ is defined by
\[
Q_{p}^{*}(t)=\sup \left\{t\tilde{t}-Q_{p}(\tilde{t}):\tilde{t}\geq 0\right\}, \qquad t \ge 0.
\]
We then define, for $s>0$, corresponding Bessel and Riesz capacities respectively by 
\begin{equation}\label{2hvEQ14}
\mathrm{Cap}_{{\mathbf{G}_{s} },Q^*_p}(E) = \inf \left\{ {\int_{\mathbb{R}^n} {Q^*_p(f)}\,dx :{\mathbf{G}_{s} }\ast f \geq {\chi _E},\ f \geq 0,\ Q^*_p(f) \in {L^1}({\mathbb{R}^n})} \right\}
\end{equation}
 and 
\begin{equation}\label{2hvEQ15}
\mathrm{Cap}_{{\mathbf{I}_{s} },Q^*_p}(E) = \inf \left\{ {\int_{\mathbb{R}^n} {Q^*_p(f)}\,dx :{\mathbf{I}_{s} }\ast f \geq {\chi _E},\ f \geq 0,\ Q^*_p(f) \in {L^1}({\mathbb{R}^n})} \right\}.
\end{equation}

The following two propositions are general versions of \cite[Theorem 2.3]{22PhVe}, whose proofs can be found in \cite[Theorems 2.1 and 2.2]{23VH}.
\begin{proposition} \label{241020147}
Let $0<s<1$, $1<p<n/s$, $\gamma>p-1$, and $\mu\in \mathcal{M}^+(\mathbb{R}^n)$. Then, the following statements are equivalent:
\begin{itemize}
 \item[(1)] The inequality 
\begin{equation}\label{241020141}
\mu(K)\leq C_1\mathrm{Cap}_{\mathbf{I}_{s p},\frac{\gamma}{\gamma-p+1}}(K)
\end{equation}
holds for any compact set $K\subset\mathbb{R}^n$, for some $C_1>0$.

\item[(2)] The inequality 
\begin{equation}\label{241020141*}
\int_{K}\left(\mathbf{W}_{s,p}[\mu](y)\right)^\gamma\,dy\leq C_2\mathrm{Cap}_{\mathbf{I}_{s p},\frac{\gamma}{\gamma-p+1}}(K)
\end{equation}
holds for any compact set $K\subset\mathbb{R}^n$, for some $C_2>0$.

\item[(3)] The inequality 
\begin{equation}\label{241020142}
\int_{\mathbb{R}^n}\left(\mathbf{W}_{s,p}[\chi_{B_t(x)}\mu](y)\right)^\gamma\,dy\leq C_{3}\mu(B_t(x))
\end{equation}
holds for any ball $B_t(x)\subset\mathbb{R}^n$, for some $C_3>0$.
\item[(4)] The inequality 
\begin{equation*}
\mathbf{W}_{s,p}\left[\left(\mathbf{W}_{s,p}[\mu]\right)^\gamma\right]\leq C_4 \mathbf{W}_{s,p}[\mu]<\infty \quad \text{a.e. in $\mathbb{R}^n$}
\end{equation*}
holds for some $C_4>0$.
\end{itemize} 
\end{proposition}

\begin{proposition}\label{241020148} 
Let $0<s<1$, $1<p<n/s$, $\gamma>p-1$, and $\mu\in \mathcal{M}^+(\mathbb R^n)$ with $\mathrm{supp}\,\mu \subset B_{R}(0)$ for some $R>1$. Then, the following statements are equivalent:
\begin{itemize}
\item[(1)] The inequality 
\begin{equation}\label{241020144}
\mu(K)\leq C_1\mathrm{Cap}_{\mathbf{G}_{s p},\frac{\gamma}{\gamma-p+1}}(K)
\end{equation}
holds for any compact set $K\subset\mathbb{R}^n$ and for some $C_1=C_1(R)>0$.
\item[(2)] The inequality 
\begin{equation}\label{241020144*}
\int_{K}\left(\mathbf{W}^{4R}_{s,p}[\mu](y)\right)^\gamma\,dy\leq C_2\mathrm{Cap}_{\mathbf{G}_{s p},\frac{\gamma}{\gamma-p+1}}(K)
\end{equation}
holds for any compact set $K\subset\mathbb{R}^n$ and for some $C_2=C_2(R)>0$.
\item[(3)] The inequality 
\begin{equation}\label{241020145}
\int_{\mathbb{R}^n}\left(\mathbf{W}^{4R}_{s,p}[\chi_{B_t(x)}\mu](y)\right)^\gamma\,dy\leq C_3 \mu(B_t(x))
\end{equation}
holds for any ball $B_t(x)\subset\mathbb{R}^n$ and for some $C_3=C_3(R)>0$.
\item[(4)] The inequality 
\begin{equation}\label{241020146}
\mathbf{W}^{4R}_{s,p}\left[\left(\mathbf{W}^{4R}_{s,p}[\mu]\right)^\gamma\right]\leq C_4 \mathbf{W}^{4R}_{s,p}[\mu] \quad \text{a.e. in  $B_{2R}$}
\end{equation}
holds for some $C_4=C_4(R)>0$.
\end{itemize} 
\end{proposition}

\begin{remark}\label{rmk26} In the two propositions above, an inspection of their proofs reveals that for each $i=1,2,3,4$, one can choose a small constant $C_i>0$ in such a way that all other constants $C_j$, $j=1,2,3,4$ with $j\neq i$, can be made sufficiently small. Specifically, for instance, when $i=2$, for any small $C_1, C_3, C_4>0$, there exists a corresponding $C_2>0$ such that (2) implies the others.
\end{remark}

To prove the existence of SOLA to nonlocal equations of Lane--Emden type, we need the following lemmas.

\begin{lemma}\label{apro-le} 
Let $\alpha\in (0,n)$, $\beta>1$ and let $\mathbf{P}$ be either $\mathbf{I}_{\alpha}$ or $\mathbf{G}_{\alpha}$. 
For $\mu\in \mathcal{M}^+(\mathbb{R}^n)$ and a sequence $\{\rho_j\}_{j=1}^{\infty}$ of standard mollifiers in $\mathbb{R}^n$, set $\mu_j=\mu*\rho_j$. Assume that there exists a constant $c>0$ such that
\begin{equation*}
\mu(K)\leq c \mathrm{Cap}_{\mathbf{P},\beta}(K)
\end{equation*}
holds for any compact set $K\subset\mathbb{R}^n$. Then there exists  a constant $c_0>0$, depending only on $c$, $n$, $\alpha$ and $\beta$, such that 
\begin{equation*}
\mu_j(K)\leq c_0 \mathrm{Cap}_{\mathbf{P},\beta}(K)
\end{equation*}
holds for any compact set $K\subset\mathbb{R}^n$.		
\end{lemma}
\begin{proof}
In the special case that $\mathbf{P}=\mathbf{G}_1$, the proof can be found in \cite[Lemma 5.7]{55Ph2}. In the same way, we can also obtain the lemma for general $\mathbf{P}$.
\end{proof}

\begin{lemma}\label{equi-integrable} 
Let $\mu$ and  $\mu_j$ be as in Lemma \ref{apro-le}, $0<s<1$, $1<p<n/s$ and $\gamma>p-1$. 
If $\left(\mathbf{W}_{s,p}[\mu]\right)^\gamma \in L^1_{\mathrm{loc}}(\mathbb{R}^n)$, then 
$\{(\mathbf{W}_{s,p}[\mu_j])^\gamma\}$ is equi-integrable in  $B_M(0)$ for any $M>1$.
\end{lemma}
\begin{proof} 
Throughout the proof, all the balls are centered at the origin. 
We notice that, since $\{(\mathbf{W}_{s,p}[\chi_{\mathbb{R}^{n}\setminus B_{2M}}\mu_{j}])^\gamma\}$ is bounded in $L^\infty (B_{M})$, 
it suffices to show that $\{(\mathbf{W}_{s,p}^{4M}[\chi_{B_{2M}}\mu_{j}])^\gamma\}$ is equi-integrable in  $B_{M}$. 
When $1<p<2$, H\"older's inequality and Fubini's theorem imply 
\[ \mathbf{W}^{4M}_{s,p}[\chi_{B_{2M}}\mu_{j}] \le \mathbf{W}^{4M}_{s,p}[\chi_{B_{4M}}\mu]\ast \rho_{j}, \] 
from which the conclusion follows. We now consider the case $p \ge 2$. Since  
\[ \mathbf{I}_{sp}^{4M}[\chi_{B_{2M}}\mu_{j}]\leq \mathbf{I}_{sp}^{4M}[\chi_{B_{4M}}\mu]\ast\rho_{j} \le c\left(\mathbf{W}^{8M}_{s,p}[\chi_{B_{4M}}\mu]\right)^{p-1}\ast \rho_{j}, \] 
we have that $\{\mathbf{I}_{sp}^{4M}[\chi_{B_{2M}}\mu]*\rho_{j}\}$ is convergent locally in $L^{\gamma/(p-1)}(\mathbb{R}^n)$ and therefore $\{(\mathbf{I}_{sp}^{4M}[\chi_{B_{2M}}\mu_{j}])^{\gamma/(p-1)}\}$ is equi-integrable in $B_R$ for any $R>0$. 
Thus, in light of \cite[Proposition 1.27]{luigi}, we can find a nondecreasing function $\Phi: [0,\infty)\to [0,\infty)$ such that $\Phi(t)/t\to\infty$ as $t\to \infty$ and 
\begin{align*}
& \int_{0}^{\infty}\phi(t)\left|\left\{x\in B_{8M}:\left(\mathbf{I}_{sp}^{4M}[\chi_{B_{2M}}\mu_{j}]\right)^{\gamma/(p-1)}>t\right\}\right|\,dt \\
& =\int_{B_{8M}}\Phi\left(\left(\mathbf{I}_{sp}^{4M}[\chi_{B_{2M}}\mu_{j}](x)\right)^{\gamma/(p-1)}\right)\, dx
\leq 1,
\end{align*} 
where $\phi(t) = \Phi'(t)$. Moreover, we may assume the $\Delta_2$-condition: $\phi(2t)\leq C\phi(t)$ for all $t \ge 0$ and for some $C \ge 1$, see \cite{Meyer}. 
On the other hand, by \cite[Proposition 2.2]{22VHV} and Lemma \ref{max.ftn.pot}, there exist $c,\varepsilon_0,t_0>0$ such that 
\begin{align*}
&\left|\left\{x\in B_{8M}: \mathbf{W}_{s,p}^{4M}[\chi_{B_{2M}}\mu_{j}](x) >3t, \left(\mathbf{I}_{sp}^{4M}[\chi_{B_{2M}}\mu_{j}](x)\right)^{1/(p-1)}\leq \varepsilon t\right\}\right| \\
&\qquad \leq c\varepsilon 	\left|\left\{x\in B_{8M}: \mathbf{W}_{s,p}^{4M}[\chi_{B_{2M}}\mu_{j}](x)>t\right\}\right|
\end{align*}
for any $\varepsilon\in (0,\varepsilon_0)$ and $t>t_0$. Using the above two inequalities and the $\Delta_2$-condition of $\phi$, and then choosing $\varepsilon>0$ sufficiently small, we have 
\begin{align*}
&\int_{0}^{\infty}\phi(t)\left|\left\{x\in B_{8M}:\left(\mathbf{W}_{s,p}^{4M}[\chi_{B_{2M}}\mu_{j}](x)\right)^\gamma>t\right\}\right|\,dt\\
&\leq c\int_{0}^{\infty}\phi(t)\left|\left\{x\in B_{8M}: \left(\mathbf{W}_{s,p}^{4M}[\chi_{B_{2M}}\mu_{j}](x)\right)^\gamma>3^\gamma t\right\}\right|\,dt\\
&\leq c\int_{t_0^{1/\gamma}}^{\infty}\phi(t)\left|\left\{x\in  B_{8M}:\left(\mathbf{W}_{s,p}^{4M}[\chi_{B_{2M}}\mu_{j}]\right)^\gamma>3^\gamma t\right\}\right|\,dt+c\\
&\leq c\varepsilon \int_{0}^{\infty} \phi(t)\left|\left\{x\in  B_{8M}:\left(\mathbf{W}_{s,p}^{4M}[\chi_{B_{2M}}\mu_{j}]\right)^\gamma>t\right\}\right|\, dt\\
&\quad\; + c\int_{0}^{\infty}\phi(t)\left|\left\{x\in  B_{8M}:\left(\mathbf{I}_{sp}^{4M}[\chi_{B_{2M}}\mu_{j}]\right)^{\gamma/(p-1)}>\varepsilon^\gamma t\right\}\right|\,dt+c\\
&\leq \frac{1}{2}\int_{0}^{\infty} \phi(t)\left|\left\{x\in  B_{8M}:  \left(\mathbf{W}_{s,p}^{4M}[\chi_{B_{2M}}\mu_{j}]\right)^\gamma>t\right\}\right|\, dt+c,
\end{align*}
which implies that 
\begin{align*}
& \int_{B_{8M}}\Phi\left(\left(\mathbf{W}_{s,p}^{4M}[\chi_{B_{2M}}\mu_{j}]\right)^\gamma\right)\,dx \\
& = \int_{0}^{\infty}\phi(t)\left|\left\{x\in  B_{8M}:\left(\mathbf{W}_{s,p}^{4M}[\chi_{B_{2M}}\mu_{j}](x)\right)^\gamma>t\right\}\right|\,dt\leq c.
\end{align*}
Hence, $\{(\mathbf{W}^{4M}_{s,p}[\chi_{B_{2M}} \mu_{j}])^\gamma\}$ is equi-integrable in  $B_M$ for any $M>1$, as desired.
\end{proof}

\begin{lemma}\label{equi-integrable2}
Let $\Omega \subset \mathbb{R}^{n}$ be either a bounded domain or the whole $\mathbb{R}^{n}$. 
Let $\{\mu_{j}\} \subset \mathcal{M}(\Omega)$, $0 < R \le \infty$, $0<s<1$, $1<p<n/s$, $l \in \mathbb{N}$, $a>0$, $\beta\ge1$ and $l\beta > p-1$. 
If $\{P_{l,a,\beta}(c\mathbf{W}^{R}_{s,p}[\mu_{j}])\}$ is bounded in $L^{1}(\Omega)$ for some $c>0$, then 
$\left\{P_{l,a,\beta}(c'\mathbf{W}^{R}_{s,p}[\mu_{j}])\right\}$ is equi-integrable in $\Omega$ for any $c' \in (0,c)$.
\end{lemma}
\begin{proof} 
By the definition of $P_{l,a,\beta}$ given in \eqref{2hvEQ11}, we have for $\kappa=(c-c')/4$
\begin{align*}
P_{l,a,\beta}(c't) & \leq \exp\left(-\frac{1}{10^{2+\beta}}\kappa^\beta at_0^{\beta}\right)\exp\left((\kappa+c')^\beta a|t|^{\beta-1}t\right) \\
& \leq \exp\left(-\frac{1}{10^{2+\beta}}\kappa^\beta at_0^{\beta}\right)P_{l,a,\beta}(ct)
\end{align*}
whenever $t\geq t_0$, where $t_0$ is large enough. This implies 
\begin{equation*}
P_{l,a,\beta}(c't)\leq \varepsilon	P_{l,a,\beta}(ct)+C(\varepsilon,c,c')
\end{equation*}
for any $\varepsilon>0$, $c' \in (0,c)$ and $t>0$. Let
\begin{equation*}
\sup_j	\int_{\Omega}P_{l,a,\beta}\left(c\mathbf{W}^{R}_{s,p}[\mu_j]\right) dx = M.
\end{equation*}
Then for any $E\subset \Omega$, we have 
\begin{align*}
\int_{E}P_{l,a,\beta}\left(c'\mathbf{W}^{R}_{s,p}[\mu_j]\right) dx & \leq \varepsilon\int_{E}P_{l,a,\beta}\left(c\mathbf{W}^{R}_{s,p}[\mu_j]\right) dx+|E|C(\varepsilon,c,c') \\
& \leq M\varepsilon +|E|C(\varepsilon,c,c').
\end{align*}
This implies the equi-integrability of $P_{l,a,\beta}(c'\mathbf{W}^{R}_{s,p}[\mu_j])$ for any $c'<c$.
\end{proof}

\section{SOLA (Solutions Obtained as Limits of Approximations)}
Here we introduce the definition of SOLA. Let $P(u)$ be either the zero function, a power function $|u|^{\gamma-1}u$ or an exponential function $P_{l,a,\beta}(u)$ defined in \eqref{2hvEQ11}.

\begin{definition}
\label{defsola}
Let $\Omega \subset \mathbb{R}^{n}$ be a bounded domain, $\mu\in \mathcal{M}(\Omega)$ and let $-\mathcal{L}_{\Phi}$ be defined in \eqref{def.L} under assumptions \eqref{mono}, \eqref{c2} and \eqref{s.p.range}. We say that a function $u\in W^{h,q}(\Omega)$ for any $h$ and $q$ satisfying
\begin{equation}\label{SOLA1}
h\in (0,s) \quad \text{and} \quad p-1 \le q<\bar{q} \coloneqq \frac{n(p-1)}{n-s}
\end{equation}
 is a SOLA to \eqref{EQp}
if $u$ is a distributional solution to $-\mathcal{L}_\Phi u = P(u) + \mu$ in $\Omega$, that is, $P(u)\in L^1(\Omega)$ and 
\begin{equation}\label{SOLA2}
\int_{\mathbb{R}^n}\int_{\mathbb{R}^n}\Phi(u(x)-u(y))(\varphi(x)-\varphi(y))K(x,y)\,dx\,dy=\int_{\Omega}  P(u) \varphi\, dx + \int_{\Omega}\varphi\, d\mu
\end{equation}
holds whenever $\varphi\in C^{\infty}_{c}(\Omega)$, and $u=0$ in $\mathbb{R}^n\setminus\Omega$. Moreover, it has to satisfy the following approximation property: there exists a sequence of weak solutions $\{u_j\}_{j=1}^\infty$ to the approximate Dirichlet problems 
\begin{equation*}
\left\{
\begin{aligned}
-\mathcal{L}_\Phi u_j&= P(u_j)+ \mu_j &\text{in }&\Omega,\\
u_j&=0&\text{in }&\mathbb{R}^n\setminus\Omega
\end{aligned}
\right.
\end{equation*}
such that $u_j \to u$ a.e in $\Omega$, $u_j \to u$ locally in $L^{q}(\mathbb{R}^n)$ and $P(u_{j}) \rightarrow P(u)$ locally in $L^{1}(\mathbb{R}^{n})$. Here the sequence $\{\mu_j\}\subset L^{\infty}(\Omega)$ converges to $\mu$ weakly in the sense of measures in $\Omega$ and moreover satisfies 
\begin{equation}\label{limsupmuj}
\limsup_{j\to \infty}|\mu_j|(B)\leq |\mu|(\overline{B})
\end{equation}
whenever $B \subset \mathbb{R}^{n}$ is a ball. 
\end{definition}

\begin{remark}
We note that \cite{KuMiSi} considered SOLA under the assumption
\[
p > 2-\frac{s}{n} \quad \text{and} \quad \max\{p-1,1\}\leq q<\bar q.
\]
However, in contrast with the local case, we use the terminology SOLA for any $p>1$. This is because the strong convergence of $u_{j}$ in $W^{h,q}$ is not required, already in the case $p>2-s/n$. 
Thus, we can extend the definition of SOLA in \cite{KuMiSi} to the case $1<p\le 2-s/n$, by considering fractional Sobolev spaces $W^{h,q}$ with $q \in (0,1)$. 
\end{remark}

In this paper, we extend the existence result in \cite[Theorem 1.1]{KuMiSi} to the ranges of $p$ and $q$ stated in Definition \ref{defsola}.   
Furthermore, we also discuss SOLA in the whole  $\mathbb{R}^n$. Here we introduce the definition of SOLA in $\mathbb{R}^n$. 
\begin{definition}
\label{defsola2}
Let $\mu\in \mathcal{M}(\mathbb{R}^n)$ and let $-\mathcal{L}_{\Phi}$ be defined in \eqref{def.L} under assumptions \eqref{mono}, \eqref{c2} and \eqref{s.p.range}. We say that a function $u\in W^{h,q}_{\mathrm{loc}}(\mathbb{R}^n)$ for any $h$ and $q$ satisfying \eqref{SOLA1} is a SOLA to
\begin{equation}\label{EQpRn}
-\mathcal{L}_\Phi u=P(u)+ \mu  \quad \text{in}\;\; \mathbb{R}^n
\end{equation} 
if 
\begin{equation}\label{SOLA.tail}
\tail(u;0,r)<\infty \quad \text{for any }\; r \ge 1 
\end{equation} 
and  $u$ is a distributional solution to \eqref{EQpRn}, that is, $P(u)\in L^1_{\mathrm{loc}}(\mathbb{R}^n)$ and
\begin{equation}\label{SOLA2*}
\int_{\mathbb{R}^n}\int_{\mathbb{R}^n}\Phi(u(x)-u(y))(\varphi(x)-\varphi(y))K(x,y)\,dx\,dy=\int_{\mathbb{R}^n}  P(u) \varphi\, dx + \int_{\mathbb{R}^n}\varphi\,d\mu
\end{equation}
holds whenever $\varphi\in C_c^\infty(\mathbb{R}^n)$. Moreover, it has to satisfy the following approximation property: there exist a sequence of open sets $\{\Omega_j\}_{j=1}^\infty$ with $\Omega_1\Subset\Omega_2\Subset\cdots\Subset\Omega_j \Subset \cdots$ and $\cup_{j=1}^\infty\Omega_j=\mathbb{R}^n$ and a sequence of weak solutions $\{u_j\}_{j=1}^{\infty}$ to the approximate Dirichlet problems 
\begin{equation*}
\left\{
\begin{aligned}
-\mathcal{L}_\Phi u_j&= P(u_j) + \mu_{j} &\text{in }&\Omega_j,\\
u_j&= 0&\text{in }&\mathbb{R}^n\setminus\Omega_j
\end{aligned}
\right.
\end{equation*}
such that  $u_j \to u$ a.e. in $\mathbb{R}^n$, $u_j \to u$ locally in $L^{q}(\mathbb{R}^n)$ and $P(u_{j}) \rightarrow P(u)$ locally in $L^{1}(\mathbb{R}^{n})$. Here the sequence $\{\mu_j\}\subset L^{\infty}(\mathbb{R}^n)\cap L^1(\mathbb{R}^n)$ converges to $\mu$ weakly in the sense of measures in $B$ and satisfies \eqref{limsupmuj} whenever $B \subset \mathbb{R}^{n}$ is a ball. 
\end{definition}

\subsection{Existence of SOLA}

In this subsection, we prove the existence of SOLA to \eqref{EQ}. We start with the following a priori estimates. 
\begin{lemma} \label{es-L1}Let $\Omega \subset \mathbb{R}^{n}$ be a bounded domain, $\mu \in L^{\infty}(\Omega)$ and let $-\mathcal{L}_{\Phi}$ be defined in \eqref{def.L} under assumptions \eqref{mono}, \eqref{c2} and \eqref{s.p.range}. Let $u\in \mathbb{W}^{s,p}(\Omega)$ be the weak solution to \eqref{EQ}. Then we have the following:
\begin{itemize} 
\item[(1)]
There exists a constant $c=c(n,s,p,\Lambda)>0$ such that
\begin{equation}\label{z1}
\|u\|_{L^{q_0,\infty}(\mathbb{R}^{n})}\leq c[|\mu|(\Omega)]^{1/(p-1)}, \quad \text{where} \quad q_0=\frac{n(p-1)}{n-sp}.
\end{equation}
\item[(2)] There exists a constant $c=c(n,s,p,\Lambda)>0$ such that, for any $\xi>1$ and $d>0$,
\begin{equation}\label{z1'}
\int_{\mathbb{R}^n}\int_{\mathbb{R}^{n}} \frac{|u(x)-u(y)|^{p}}{(d+|u(x)|+|u(y)|)^\xi} \frac{dx\,dy}{|x-y|^{n+sp}}\leq  c\frac{d^{1-\xi}}{\xi-1}|\mu|(\Omega).
\end{equation}
\item[(3)] For any $q\in(0,\bar{q})$ and $h\in(0,s)$, where $\bar{q}$ is given in \eqref{SOLA1}, there exists a constant $c=c(n,s,p,\Lambda,h,q)>0$ such that
\begin{equation}\label{z2}
\left(\int_{B_{r}}\fint_{B_{r}}\frac{|u(x)-u(y)|^{q}}{|x-y|^{n+hq}}\,dx\,dy\right)^{1/q} \le cr^{-h}\left[\frac{|\mu|(\Omega)}{r^{n-sp}}\right]^{1/(p-1)}
\end{equation}
for any ball $B_{r} \subset \mathbb{R}^{n}$. Moreover, 
\begin{equation}\label{global}
\left(\int_{\mathbb{R}^n}\int_{\mathbb{R}^{n}}\frac{|u(x)-u(y)|^q}{|x-y|^{n+hq}}\,dx\,dy\right)^{1/q}\leq  c[\mathrm{diam}(\Omega)]^{\frac{n}{q}-h}\left(\frac{|\mu|(\Omega)}{[\mathrm{diam}(\Omega)]^{n-sp}}\right)^{1/(p-1)}.
\end{equation}
\end{itemize}
\end{lemma}
\begin{proof} 
Estimates \eqref{z1}, \eqref{z1'} and \eqref{global} are obtained in \cite[Propositions 2.5, 2.6 and 2.7]{Gk24}, respectively. Moreover, an inspection of the proof of \cite[Proposition 2.7]{Gk24} also gives \eqref{z2}. 
\end{proof}

The above lemma and the compactness result in Lemma \ref{lemcompact} imply the existence of SOLA. 

\begin{theorem}\label{thm.ex.bdd}
Let $\Omega \subset \mathbb{R}^{n}$ be a bounded domain, $\mu \in \mathcal{M}(\Omega)$ and let $-\mathcal{L}_{\Phi}$ be defined in \eqref{def.L} under assumptions \eqref{mono}, \eqref{c2} and \eqref{s.p.range}. Then there exists a SOLA $u$ to \eqref{EQ} satisfying \eqref{global}.
\end{theorem}
\begin{proof} 
Consider $\mu_j= \mu\ast \rho_{j}$, where $\{\rho_{j}\}_{j=1}^{\infty}$ is a sequence of standard mollifiers, and let $\{u_j\}_{j=1}^{\infty}\subset \mathbb{W}^{s,p}(\Omega)$ be the sequence of weak solutions to
\begin{equation*}
\left\{
\begin{aligned}
-\mathcal{L}_{\Phi}u_j&=\mu_j&\text{in }& \Omega,\\
u_j&=0&\text{in }&\mathbb{R}^n\setminus \Omega.
\end{aligned}
\right.
\end{equation*}
Then $\{\mu_{j}\}$ converges to $\mu$ weakly in the sense of measures in $\mathbb{R}^n$ and satisfies \eqref{limsupmuj} for any ball $B \subset \mathbb R^n$. Now, for any $h \in (0,s)$ and $q \in (0,\bar{q})$, we use Lemma \ref{es-L1}, together with the fact that
\begin{equation}\label{hol.uj} 
\|u_j\|_{L^q(\Omega)} \le c(q,q_0)|\Omega|^{1/q-1/q_0}\|u_j\|_{L^{q_0,\infty}(\Omega)} 
\end{equation}
 and the construction of $\{\mu_{j}\}$, in order to have
\begin{align*}
\left(\int_{\Omega}|u_{j}|^{q}\,dx\right)^{1/q} + \left(\int_{\Omega}\int_{\Omega}\frac{|u_j(x)-u_j(y)|^q}{|x-y|^{n+hq}}\,dx\,dy\right)^{1/q}  & \leq c[|\mu_j|(\Omega)]^{1/(p-1)} \\
& \leq c[|\mu|(\Omega)]^{1/(p-1)}	
\end{align*}
whenever $j \in \mathbb{N}$, where $c=c(n,s,p,\Lambda, \Omega,h,q)>0$. Namely, $\{u_j\}$ is bounded in $W^{h,q}(\Omega)$.  Therefore, by Lemma \ref{lemcompact}, there exists $u\in W^{h,q}(\Omega)$ with $u = 0$ a.e. in $\mathbb R^n\setminus \Omega$ such that
\begin{equation*}
\left\{
\begin{aligned}
	&u_j \ \longrightarrow\ \ u\quad \textrm{in } L^{q}(\Omega),\\
	&u_j \ \longrightarrow\ \ u\quad \mathrm{a.e.\  in }\  \Omega
\end{aligned}
\right. 
\quad \text{as} \quad j\to \infty \ \ \text{(up to subsequences).}
\end{equation*}
Accordingly, Fatou's lemma implies
\[
\left(\int_{\Omega}|u|^{q}\,dx\right)^{1/q} +\left(\int_{\Omega}\int_{\Omega}\frac{|u(x)-u(y)|^q}{|x-y|^{n+hq}}\,dx\,dy\right)^{1/q}\leq c[|\mu|(\Omega)]^{1/(p-1)}.
\]
It remains to show that $u$ satisfies \eqref{SOLA2}. The argument is essentially identical to the one in \cite[Theorem 1.1]{KuMiSi}, see \cite[pp. 1352--1354]{KuMiSi} with $g_j=g=0$ 
and also the proof of Theorem \ref{thm.ex.rn} below. We therefore only highlight the point relevant to the setting of the present paper. 
Note that, in the proof of \cite[Theorem 1.1]{KuMiSi}, the convergence of the nonlocal term is achieved via an equi-integrability argument. More precisely, with the choices of $h \in (0,s)$ and $\varepsilon>0$ satisfying
\begin{equation}\label{existence.hq}
\varepsilon n + [p(s-h)+h-1](1+\varepsilon) \le 0 \quad \text{and} \quad q = (p-1)(1+\varepsilon) < \bar{q}, 
\end{equation}
where $\bar{q}$ is given in \eqref{SOLA1}, the argument solely relies on uniform bounds for 
\[ \frac{|u_j(x)-u_j(y)|}{|x-y|^h} \] 
in the corresponding weighted $L^q$-space over $\Omega\times\Omega$. Thus, by invoking the compactness result in Lemma \ref{lemcompact}, we can adapt the argument in \cite{KuMiSi} to the present range $q<1$. 
Hence we omit the details here. 
\end{proof}

\begin{theorem}\label{thm.ex.rn}
Let $\mu \in \mathcal{M}(\R^n)$ and let $-\mathcal{L}_{\Phi}$ be defined in \eqref{def.L} under assumptions \eqref{mono}, \eqref{c2} and \eqref{s.p.range}. Then there exists a SOLA $u$ to \eqref{EQ} with $\Omega=\mathbb{R}^{n}$.
\end{theorem}
\begin{proof} 
Consider $\mu_{j} = \mu \ast \rho_{j}$, where $\{\rho_{j}\}_{j=1}^{\infty}$ is a sequence of standard mollifiers, and let $\{u_j\}_{j=1}^{\infty} \subset \mathbb{W}^{s,p}(B_{j}(0))$ be the sequence of weak solutions to 
\begin{equation*}
\left\{
\begin{aligned}
-\mathcal{L}_\Phi u_j&=\mu_{j} &\text{in }&B_{j}(0),\\
u_j&=0&\text{in }&\mathbb{R}^n\setminus B_{j}(0).
\end{aligned}
\right.
\end{equation*}
Then $\{\mu_{j}\} \subset C^{\infty}_{c}(\mathbb{R}^{n})$ converges to $\mu$ weakly in the sense of measures in $B$ and satisfies \eqref{limsupmuj}
for every ball $B \subset \mathbb{R}^n$. 

We now fix any $r\geq 1$. Then for any $h \in (0,s)$ and $q \in (0,\bar{q})$, we have
\begin{equation}\label{eq:j_energy+}
\begin{aligned}
\left(\int_{B_{r}(0)}\fint_{B_{r}(0)}\frac{|u_{j}(x)-u_{j}(y)|^{q}}{|x-y|^{n+hq}}\,dx\,dy\right)^{1/q} \overset{\eqref{z2}}&{\le} \frac{c}{r^{h+n/q_{0}}}[|\mu_{j}|(B_{j}(0))]^{1/(p-1)} \\
& \le \frac{c}{r^{h+n/q_{0}}}[|\mu|(\mathbb{R}^{n})]^{1/(p-1)}
\end{aligned}
\end{equation}
and
\begin{equation}\label{eq:j_energy}
\left(\fint_{B_{r}(0)}|u_{j}|^{q}\,dx\right)^{1/q} \overset{\eqref{z1},\eqref{hol.uj}} \le \frac{c}{r^{n/q_{0}}}[|\mu_{j}|(B_{j}(0))]^{1/(p-1)} \le \frac{c}{r^{n/q_{0}}}[|\mu|(\mathbb{R}^{n})]^{1/(p-1)}
\end{equation}
whenever $j \ge r$, where $q_0$ is given in \eqref{z1} and $c=c(n,s,p,\Lambda,h,q)>0$. That is, $\{u_j\}_{j\ge r}$ is bounded in $W^{h,q}(B_r(0))$. 
Since $r \ge 1$ was arbitrary, by Lemma \ref{lemcompact}, there exists $u\in W^{h,q}_{\mathrm{loc}}(\mathbb R^n)$ such that
\begin{equation}\label{strong_converge}
\left\{
\begin{aligned}
	&u_j \ \longrightarrow\ \ u\quad \textrm{locally in } L^{q}(\mathbb R^n),\\
	&u_j \ \longrightarrow\ \ u\quad \mathrm{a.e.\  in }\  \mathbb R^n
\end{aligned}
\right.
\quad \text{as} \quad j\to \infty \ \ \text{(up to subsequences).}
\end{equation}
In particular, applying Fatou's lemma to \eqref{eq:j_energy+} and \eqref{eq:j_energy}, we also have
\begin{equation}\label{eq:Rn_energy}
 \left(\fint_{B_{r}(0)}|u|^{q}\,dx\right)^{1/q} + r^{h}\left(\int_{B_{r}(0)}\fint_{B_{r}(0)}\frac{|u(x)-u(y)|^{q}}{|x-y|^{n+hq}}\,dx\,dy\right)^{1/q} 
 \le \frac{c}{r^{n/q_{0}}}[|\mu|(\mathbb{R}^{n})]^{1/(p-1)}.
\end{equation}

Next, in order to show \eqref{SOLA.tail}, we consider the weighted Lebesgue and Marcinkiewicz spaces with the weight
\[ \omega(x) \coloneqq \frac{1}{(1+|x|)^{n+sp}}, \qquad x \in \mathbb{R}^{n}. \]
Observe that 
\begin{align*} 
\|f\|_{L^{p-1}(E,d\omega)} & \le c(p,q_{0})[\omega(E)]^{1/(p-1)-1/q_{0}}\|f\|_{L^{q_{0},\infty}(E,d\omega)} \\
& \le c(p,q_{0})[\omega(E)]^{1/(p-1)-1/q_{0}}\|f\|_{L^{q_{0},\infty}(E)}
\end{align*}
holds for any Borel set $E \subset \mathbb{R}^{n}$ and $f \in L^{q_{0},\infty}(E)$, where for the last inequality we have used the fact that $d\omega \le dx$. Using this inequality with $E= \mathbb{R}^{n}\setminus B_{r}(0)$ and $f=u_{j}$, we have
\begin{align*}
 \left(\int_{\mathbb{R}^{n}\setminus B_{r}(0)}\frac{|u_{j}(x)|^{p-1}}{(1+|x|)^{n+sp}}\,dx\right)^{1/(p-1)}
& = \|u_{j}\|_{L^{p-1}(\mathbb{R}^{n}\setminus B_{r}(0),d\omega)} \\
&  \le cr^{-sp\left(\frac{1}{p-1}-\frac{1}{q_{0}}\right)}\|u_{j}\|_{L^{q_{0},\infty}(\mathbb{R}^{n})}.
\end{align*}
Then \eqref{z1} along with an elementary manipulation gives 
\begin{equation}\label{uj.tail}
\int_{\mathbb{R}^{n}\setminus B_{r}(0)}\frac{|u_{j}(x)|^{p-1}}{|x|^{n+sp}}\,dx \le cr^{-(sp)^{2}/n}[|\mu_{j}|(B_{j})] \le cr^{-(sp)^{2}/n}[|\mu|(\mathbb{R}^{n})]
\end{equation}
for some $c=c(n,s,p,\Lambda)>0$, whenever $r \ge 1$ and $j \in \mathbb{N}$. Hence, Fatou's lemma implies
\begin{equation}\label{u.tail}
\int_{\mathbb{R}^{n}\setminus B_{r}(0)}\frac{|u(x)|^{p-1}}{|x|^{n+sp}}\,dx \le cr^{-(sp)^{2}/n}[|\mu|(\mathbb{R}^{n})],
\end{equation}
which in particular yields \eqref{SOLA.tail}.

It remains to show that $u$ satisfies \eqref{SOLA2*} with $P(\cdot)\equiv0$; we follow the argument in \cite[Section 4.2]{KuMiSi}. Fix  $\varphi\in C_c^\infty(\mathbb{R}^n)$ and $T_0>0$ such that $\mathrm{supp}\, \varphi\subset B_{T_{0}} = B_{T_0}(0)$. Then
\begin{equation*}
\int_{\mathbb{R}^n}\int_{\mathbb{R}^n}\Phi(u_j(x)-u_j(y))(\varphi(x)-\varphi(y))K(x,y)\, dx\,dy =  \int_{\mathbb{R}^n}\varphi\,d\mu_{j}
\end{equation*} 
for all $j>2T_0$. 
By the weak convergence of $\mu_{j}$, we have 
\begin{equation*}
\lim_{j\to \infty}\int_{\mathbb{R}^n}\varphi \,d\mu_j = \int_{\mathbb{R}^n}\varphi \,d\mu.
\end{equation*}
Now, denoting
\[ \psi_{j}(x,y) = [\Phi(u_{j}(x)-u_{j}(y))-\Phi(u(x)-u(y))](\varphi(x)-\varphi(y))K(x,y), \]
it suffices to show that
\begin{equation}\label{es16}
\lim_{j\to \infty}\int_{\mathbb{R}^n}\int_{\mathbb{R}^n}\psi_{j}(x,y)\,dxdy = 0.
\end{equation}
Let $T>2T_0$. Then, with $B_{T} = B_{T}(0)$ and $B_{T_{0}} = B_{T_{0}}(0)$, we split
\begin{align*}
\int_{\mathbb{R}^n}\int_{\mathbb{R}^n} \psi_{j}(x,y)\,dx\,dy 
& = \int_{B_{T_0}}\int_{B_{T_0}}(\cdots) + \int_{\mathbb{R}^n\setminus B_{T}}\int_{B_{T_0}}(\cdots) + \int_{ B_{T_0}}\int_{\mathbb{R}^n\setminus B_{T}}(\cdots) \\
& \eqqcolon I_{1,j,T}+I_{2,j,T}+I_{3,j,T}.
\end{align*}
For $I_{1,j,T}$, 
we choose $h \in (0,s)$ and $\varepsilon>0$ such that \eqref{existence.hq} holds. Then, since
\begin{align*}
\left|\psi_{j}(x,y)\right|^{1+\varepsilon} &\leq c\frac{\left[(|u(x)-u(y)|^{p-1}+|u_j(x)-u_j(y)|^{p-1})|\varphi(x)-\varphi(y)|\right]^{1+\varepsilon}}{|x-y|^{(n+sp)(1+\varepsilon)}} \\
&\leq c\|D\varphi\|_{L^{\infty}}^{1+\varepsilon}\frac{\left[|u(x)-u(y)|^{p-1}+|u_j(x)-u_j(y)|^{p-1}\right]^{1+\varepsilon}}{|x-y|^{n+h(p-1)(1+\varepsilon)}|x-y|^{\varepsilon n + [p(s-h)+h-1](1+\varepsilon)}} \\
&\leq c(T_{0},\varphi) \left[\frac{|u(x)-u(y)|^{(p-1)(1+\varepsilon)}}{|x-y|^{n+h(p-1)(1+\varepsilon)}} + \frac{|u_j(x)-u_j(y)|^{(p-1)(1+\varepsilon)}}{|x-y|^{n+h(p-1)(1+\varepsilon)}} \right]
\end{align*}
for any $x,y\in B_{T_{0}}$, it follows from \eqref{eq:j_energy+} and \eqref{eq:Rn_energy}  that $\{\psi_{j}\}$ is equi-bounded in $L^{1+\varepsilon}(B_{T_{0}}\times B_{T_{0}})$ and hence $\{\psi_{j}\}$ is equi-integrable in $B_{T_{0}}\times B_{T_{0}}$. In turn, Vitali's convergence theorem with \eqref{strong_converge} implies that $|I_{1,j,T}|\to 0$ as $j\to \infty$.
As for $I_{2,j,T}$ and $I_{3,j,T}$, we have from \eqref{eq:j_energy}, \eqref{eq:Rn_energy}, \eqref{uj.tail} and \eqref{u.tail} that for $j>T$,
\begin{align*}
& |I_{2,j,T}|+|I_{3,j,T}| \\
&\leq c(T_0,\varphi) \int_{\mathbb{R}^n\setminus B_{T}}\int_{B_{T_0}} \left(|u_j(x)|^{p-1}+|u(x)|^{p-1}+ |u_j(y)|^{p-1}+|u(y)|^{p-1}\right)\frac{dx\,dy}{|y|^{n+sp}}\\
&\leq c(T_0,\varphi)\bigg( T^{-sp} \int_{ B_{T_0}} \left[|u_j(x)|^{p-1}+|u(x)|^{p-1}\right]\, dx \\
&\qquad \qquad \qquad +  \int_{ \R^n\setminus B_{T}} \left[|u_j(y)|^{p-1}+|u(y)|^{p-1}\right]\, \frac{dy}{|y|^{n+sp}} \bigg)\\
&\leq c(T_0,\varphi,|\mu|(\mathbb{R}^{n}))\left(T^{-sp} + T^{-(sp)^{2}/n}\right)  \ \longrightarrow\ 0 \quad  \text{as}\ \ T\ \to\ \infty.
\end{align*}
Consequently, we obtain \eqref{es16} and the proof is complete.
\end{proof}

\subsection{A comparison principle for SOLA}
We end this section with a comparison principle. This is the point where the monotonicity property \eqref{mono2} is used.

\begin{proposition}\label{pro1} Let $\Omega \subset \mathbb{R}^{n}$ be a bounded domain, $-\mathcal{L}_{\Phi}$ be defined in \eqref{def.L} under assumptions \eqref{mono}--\eqref{s.p.range}, and $u_1$ be a SOLA to \eqref{EQ} with $\mu=\mu_1\in \mathcal{M}(\Omega)$. For any $\mu_2\in \mathcal{M}(\Omega)$ satisfying $\mu_{2}-\mu_{1} \in \mathcal{M}^{+}(\Omega)$, there exists a SOLA $u_2$ to \eqref{EQ} with $\mu=\mu_2$ satisfying 
\[
u_1\leq u_2 \quad \text{a.e. in}\;\; \Omega.
\]
\end{proposition}
\begin{proof} Let $\{u_{1,j}\} \subset \mathbb{W}^{s,p}(\Omega)$ be a sequence of weak solutions to \eqref{EQ} with datum $\mu=\mu_{1,j}\in C_c^\infty(\Omega)$ such that $u_{1,j}\to u_1$ in $L^{q}(\Omega)$  for any $q \in (0,\bar{q})$, where $\bar{q}$ is given in \eqref{SOLA1},  
and $\mu_{1,j}\rightharpoonup\mu_1$ in $\mathcal{M}(\Omega)$. We consider $\mu_{2,j} \coloneqq \phi_j*(\eta_{1/j}(\mu_2-\mu_1))+\mu_{1,j}\in  C_c^\infty(\Omega)$ for $j \in \mathbb{N}$, where $\phi_j$ is a standard mollifier with $\mathrm{supp}\,\phi_j \subset B_{1/(4j)}(0)$ and $\eta_{j}$ is a smooth function satisfying $\eta_{j} \equiv 1$ in $\{x\in \Omega:\mathrm{dist}(x,\partial\Omega)>2/j\}$ and $\eta_{j}\equiv0$ in $\{x\in \Omega:\mathrm{dist}(x,\partial\Omega)<1/j\}$. Clearly, $\mu_{2,j}\geq \mu_{1,j}$ for any $j$ and $\mu_{2,j} \rightharpoonup \mu_2$ in $\mathcal{M}(\Omega)$. Let $u_{2,j}\in \mathbb{W}^{s,p}(\Omega)$ be the weak solution to \eqref{EQ} with datum $\mu=\mu_{2,j}$. Then, it suffices to show $u_{1,j}\leq u_{2,j}$ a.e. in $\Omega$. Indeed, by choosing 
$\varphi=(u_{1,j}-u_{2,j})_+ \coloneqq \max\{u_{1,j}-u_{2,j},0\}$  
as a test function in \eqref{EQ}, we have 
\begin{align*}
& 0  \leq  \int_{\Omega}\varphi\, d(\mu_{2,j}-\mu_{1,j})\\
&=
\int_{\mathbb{R}^n}\int_{\mathbb{R}^n}[\Phi(u_{2,j}(x)-u_{2,j}(y))-\Phi(u_{1,j}(x)-u_{1,j}(y))](\varphi(x)-\varphi(y))K(x,y)\,dx\,dy\\
& = -
\int_{\{u_{1,j}\ge u_{2,j}\}}\int_{\{u_{1,j} \ge u_{2,j}\}}[\Phi(u_{2,j}(x)-u_{2,j}(y))-\Phi(u_{1,j}(x)-u_{1,j}(y))] \\
& \hspace{4.75cm} \cdot [u_{2,j}(x)-u_{2,j}(y) - (u_{1,j}(x)-u_{1,j}(y))] K(x,y)\,dx\,dy\\
&\quad\; +
\int_{\{u_{1,j} < u_{2,j}\}}\int_{\{u_{1,j} \ge  u_{2,j}\}}[\Phi (u_{2,j}(x)-u_{2,j}(y))-\Phi (u_{1,j}(x)-u_{1,j}(y))]\\
& \hspace{7cm}\cdot (u_{1,j}(x)-u_{2,j}(x))K(x,y)\,dx\,dy\\
&\quad\; -
\int_{\{u_{1,j} \ge  u_{2,j}\}}\int_{\{u_{1,j} < u_{2,j}\}}[\Phi (u_{2,j}(x)-u_{2,j}(y))-\Phi (u_{1,j}(x)-u_{1,j}(y))]\\
& \hspace{7cm}\cdot (u_{1,j}(y)-u_{2,j}(y))K(x,y)\,dx\,dy\\
& \eqqcolon - I_1+ I_2 - I_3.
\end{align*}
It follows from \eqref{mono2} that
\begin{equation*}
\Phi (u_{2,j}(x)-u_{2,j}(y))-\Phi (u_{1,j}(x)-u_{1,j}(y))< 0
\end{equation*}
provided $(u_{2,j}(x)-u_{2,j}(y))-(u_{1,j}(x)-u_{1,j}(y)) < 0$, which implies $I_2\leq 0$ and $I_3 \geq 0$. Hence $I_1 \le 0$. Then, again by \eqref{mono2},  $u_{2,j}(x)-u_{2,j}(y) = u_{1,j}(x)-u_{1,j}(y)$ for a.e. $(x,y)\in\{u_{1,j}\ge u_{2,j}\}\times \{u_{1,j}\ge u_{2,j}\}$. That is, $u_{1,j}-u_{2,j}$ is constant a.e. in $\{u_{1,j}\ge u_{2,j}\}$. Since $u_{1,j} = u_{2,j}$ a.e. in $\R^n\setminus \Omega$, we deduce that $(u_{1,j}-u_{2,j})_+=0$ a.e. in $\R^n$.
\end{proof}

\section{Wolff potential estimates for nonlocal equations}

\subsection{Regularity for homogeneous equations}
Here we recall various regularity results for the homogeneous equation 
\begin{equation}\label{homoeq}
-\mathcal{L}_{\Phi}v = 0 \quad \text{in}\;\; \Omega.
\end{equation}
Such results were first obtained in \cite{DicaKuPa1,DicaKuPa2} (see also \cite{Co1}) and later modified in \cite{KuMiSi} to match the setting of measure data problems.

We start with a local sup-estimate and a Caccioppoli type estimate. 
The following two lemmas are proved in \cite[Corollary~2.1 and Lemma~2.3]{KuMiSi} for $q=1$, and the proof for general $q$ is essentially the same. 
\begin{lemma}\label{lem.bdd.L1}
Let $p \in (1,\infty)$, $s \in (0,1)$ and let $v \in \mathbb{W}^{s,p}(\Omega)$ be a weak solution to \eqref{homoeq} under assumptions \eqref{mono} and \eqref{c2}. Then for any $B_{r} \subset \Omega$ and $k\in\mathbb{R}$, we have
\begin{equation*}
\sup_{B_{\sigma r}}|v-k| \le \frac{c}{(1-\sigma)^{np'/q}}\left[\left(\fint_{B_{r}}|v-k|^{q}\,dx\right)^{1/q} + \tail(v-k;r/2)\right]
\end{equation*}
whenever $q \in (0,p]$ and $\sigma \in (0,1)$, where $c= c(n,s,p,\Lambda,q)>0$. 
\end{lemma}

\begin{lemma}\label{lem.Caccio.Lq}
Let $p \in (1,\infty)$, $s \in (0,1)$ and let $v \in \mathbb{W}^{s,p}(\Omega)$ be a weak solution to \eqref{homoeq} under assumptions \eqref{mono} and \eqref{c2}. Then for any $B_{r} \subset \Omega$ and $k \in \mathbb{R}$, we have 
\begin{align*}
& \left(\int_{B_{\sigma r}}\fint_{B_{\sigma r}}\frac{|v(x)-v(y)|^{q}}{|x-y|^{n+hq}}\,dx\,dy\right)^{1/q} \\
& \le \frac{c}{(1-\sigma)^{\theta}r^{h}}\left[ \left(\fint_{B_{r}}|v-k|^{q}\,dx\right)^{1/q} + \tail(v-k;r/2) \right]
\end{align*}
whenever $h \in (0,s)$, $q \in (0,p]$ and $\sigma \in [1/2,1)$, where $c= c(n,s,p,\Lambda,h,q)$ and  $\theta = \theta(n,p)$ are positive constants.
\end{lemma}

We further note an oscillation estimate for $v$, which in turn implies local H\"older regularity of $v$, see for instance \cite[Theorem 1.2]{DicaKuPa2}. In particular, using Lemma \ref{lem.bdd.L1}, we can obtain the following:
\begin{lemma}\label{lem.hol}
Let $p \in (1,\infty)$, $s \in (0,1)$ and let $v \in \mathbb{W}^{s,p}(\Omega)$ be a weak solution to \eqref{homoeq} under assumptions \eqref{mono} and \eqref{c2}. 
Then $v$ is locally H\"older continuous in $\Omega$. 
In particular, there exists $\alpha_{0} = \alpha_{0}(n,s,p,\Lambda) \in (0,\min\{sp/(p-1),1\})$ such that, for any $B_{2r} \subset \Omega$ and $k\in\mathbb{R}$, we have
\begin{equation*}
\osc_{B_{\rho}}v \le c\left(\frac{\rho}{r}\right)^{\alpha_{0}}\left[ \left(\fint_{B_{2r}}|v-k|^{q}\,dx\right)^{1/q} + \tail(v-k;r/2) \right]
\end{equation*}
whenever $q \in (0,p]$ and $\rho \in (0, r]$, where $c = c(n,s,p,\Lambda,q)>0$.
\end{lemma}

One of the main features in \cite{KuMiSi,KuMiSi3} is to consider the nonlocal excess functional
\begin{equation}\label{excess1}
\left(\fint_{B_{r}(x_{0})}|f-(f)_{B_r(x_0)}|^{\max\{p-1,1\}}\,dx\right)^{1/\max\{p-1,1\}} + \tail(f-(f)_{B_r(x_0)};x_{0},r),
\end{equation}
where the first term is the traditional (local) excess functional and the second one is concerned with long-range interactions. 
However, 
the above excess functional cannot be directly employed in our setting,
since SOLA to \eqref{EQ} may not be locally integrable when $1<p \le 2n/(n+s)$. 
In the local case, modified excess functionals of the form
\begin{equation}\label{loc.exc}
A(f;x_{0},r) \coloneqq \inf_{k \in \mathbb{R}}\left(\fint_{B_{r}(x_{0})}|f-k|^{p-1}\,dx\right)^{1/(p-1)}
\end{equation}
have been considered in \cite{DZ,NP,NP2} to obtain potential estimates for \eqref{2hvEQ5}. 
In this point of view, we consider the following modified nonlocal excess functional:
\begin{equation}\label{excess.functional}
E(f;x_{0},r) \coloneqq \inf_{k \in \mathbb{R}}\left[\left(\fint_{B_{r}(x_{0})}|f-k|^{p-1}\,dx\right)^{1/(p-1)} + \tail(f-k;x_{0},r)\right].
\end{equation}
We simply write $A(f;x_{0},r)=A(f;r)$ and $E(f;x_{0},r)=E(f;r)$ when the point $x_{0}$ is not important or clear from the context. 
It is obvious that
\[ A(f;x_{0},r) \le E(f;x_{0},r). \]
We also note that when $p \ge 2$, the new excess functional  \eqref{excess.functional} is comparable to the previous one given in \eqref{excess1}, i.e.,
\begin{equation}\label{exc.equiv}
E(f;x_{0},r) \approx \left(\fint_{B_{r}(x_{0})}|f-(f)_{B_{r}(x_{0})}|^{p-1}\,dx\right)^{1/(p-1)} + \tail(f-(f)_{B_{r}(x_{0})};x_{0},r).
\end{equation} 
However, when $p<2$, our excess functional \eqref{excess.functional} is not comparable to \eqref{excess1}.

If $f \in L^{p-1}(B_{r})$, then there exists a number $\mathcal{P}_{B_{r}}(f) \in \mathbb{R}$ satisfying
\begin{equation}\label{def.P}
A(f;r) = \left(\fint_{B_{r}}|f-\mathcal{P}_{B_{r}}(f)|^{p-1}\,dx\right)^{1/(p-1)}.
\end{equation}
Such a number is not uniquely determined in general, but we use any possible value of it. Here we note that
\begin{equation}\label{av.min}
\begin{aligned}
& |\mathcal{P}_{B_{r}}(f) - k_{0}| = \left(\fint_{B_{r}}|\mathcal{P}_{B_{r}}(f) - k_{0}|^{p-1}\,dx\right)^{1/(p-1)}  \\
& \le c\left(\fint_{B_{r}}|\mathcal{P}_{B_{r}}(f)-f|^{p-1}\,dx\right)^{1/(p-1)} + c\left(\fint_{B_{r}}|f-k_{0}|^{p-1}\,dx\right)^{1/(p-1)}  \\
& \le c\left(\fint_{B_{r}}|f-k_{0}|^{p-1}\,dx\right)^{1/(p-1)}
\end{aligned}
\end{equation}
holds for any $k_{0} \in \mathbb{R}$, where $c=c(p)>0$. 
Similarly, if $f \in L^{p-1}(B_{r})$ satisfies $\tail(f;r) < \infty$, then there exists a number $\mathcal{Q}_{B_{r}}(f) \in \mathbb{R}$ satisfying
\begin{equation}\label{def:Q}
 E(f;r) = \left(\fint_{B_{r}}|f-\mathcal{Q}_{B_{r}}(f)|^{p-1}\,dx\right)^{1/(p-1)} + \tail(f-\mathcal{Q}_{B_{r}}(f);r). 
 \end{equation}

Here we collect a few elementary properties of $E(\cdot)$.
\begin{lemma}\label{exc.property}
Let $p \in (1,\infty)$, $s \in (0,1)$ and $B_{r} = B_{r}(x_0) \subset \mathbb{R}^n$. Assume that $f,g \in L^{p-1}(B_r)$ satisfy $\tail(f;r), \tail(g;r) < \infty$. Then there exists a constant $c=c(n,s,p)>0$ such that the following properties hold:
\begin{itemize}
\item [(1)] 
$E(f+g;r) \le cE(f;r) + cE(g;r)$.
\item [(2)]
$E(f;\sigma r) \le c\sigma^{-n/(p-1)}E(f;r)$ for any $\sigma \in (0,1)$.
\end{itemize}
\end{lemma}
\begin{proof}
(1) Let $\mathcal{Q}_{B_r}(f)$ and $\mathcal{Q}_{B_r}(g)$ be two numbers satisfying \eqref{def:Q} for $f$ and $g$, respectively. Then we estimate
\begin{align*}
 E(f+g;r) 
& \le \left(\fint_{B_{r}}|f+g-\mathcal{Q}_{B_{r}}(f)-\mathcal{Q}_{B_{r}}(g)|^{p-1}\,dx\right)^{1/(p-1)} \\
&\quad\; + \tail(f+g-\mathcal{Q}_{B_{r}}(f)-\mathcal{Q}_{B_{r}}(g);r) \\
& \le c\left(\fint_{B_{r}}|f-\mathcal{Q}_{B_{r}}(f)|^{p-1}\,dx\right)^{1/(p-1)} + c\left(\fint_{B_{r}}|g-\mathcal{Q}_{B_{r}}(g)|^{p-1}\,dx\right)^{1/(p-1)} \\
& \quad\; + c\tail(f-\mathcal{Q}_{B_{r}}(f);r) + c\tail(g-\mathcal{Q}_{B_{r}}(g);r) \\
& = cE(f;r) + cE(g;r).
\end{align*}
(2) Let $k \in \mathbb{R}$. We obviously have
\[ \left(\fint_{B_{\sigma r}}|f-k|^{p-1}\,dx\right)^{1/(p-1)} \le \sigma^{-n/(p-1)}\left(\fint_{B_{r}}|f-k|^{p-1}\,dx\right)^{1/(p-1)}. \]
As for the tail term, we get
\begin{align*}
\tail(f-k;\sigma r) & \le c\sigma^{sp/(p-1)}\left(r^{sp}\int_{B_{r}\setminus B_{\sigma r}}\frac{|f(y)-k|^{p-1}}{|y-x_0|^{n+sp}}\,dy \right)^{1/(p-1)} \\
& \quad + c\sigma^{sp/(p-1)}\left(r^{sp}\int_{\mathbb{R}^n \setminus B_{r}}\frac{|f(y)-k|^{p-1}}{|y-x_0|^{n+sp}}\,dy\right)^{1/(p-1)} \\
& \le c\sigma^{-n/(p-1)}\left[\left(\fint_{B_r}|f(y)-k|^{p-1}\,dy\right)^{1/(p-1)} + \tail(f-k;r)\right].
\end{align*}
Combining these two inequalities and then taking infimum with respect to $k \in \mathbb{R}$, we obtain the conclusion. 
\end{proof}

We now obtain an excess decay estimate for \eqref{homoeq}, which will play a key role in the proof of potential estimates.  The following lemma is analogous to \cite[Theorem 8.11]{KuMiSi3}; in particular, when $p \ge 2$, it follows from \cite[Theorem 8.11]{KuMiSi3} via \eqref{exc.equiv}. 
However, when $1<p<2$, it cannot be directly obtained from \cite[Theorem 8.11]{KuMiSi3}.
\begin{lemma}\label{thm.ed.v}
Let $p \in (1,\infty)$, $s \in (0,1)$ and let $v \in \mathbb{W}^{s,p}(\Omega)$ be a weak solution to \eqref{homoeq} under assumptions \eqref{mono} and \eqref{c2}. Then we have 
\begin{equation*}
E(v;\rho) \le c\left(\frac{\rho}{r}\right)^{\alpha_{0}}E(v;r)
\end{equation*}
whenever $B_r = B_{r}(x_0) \subset \Omega$ and $0< \rho \le r$, where $\alpha_{0}$ is the exponent given in Lemma \ref{lem.hol} and $c= c(n,s,p,\Lambda)>0$.
\end{lemma}
\begin{proof}
We may assume $\rho \le r/4$ without loss of generality. 
Recalling the definition of $E(\cdot)$ given in \eqref{excess.functional}, we have
\begin{equation}\label{ed.v.start}
E(v;\rho) \le \left(\fint_{B_{\rho}}|v-(v)_{B_{\rho}}|^{p-1}\,dx\right)^{1/(p-1)} + \tail(v-(v)_{B_{\rho}};\rho),
\end{equation}
and we estimate each term in the right-hand side. 
Observe that Lemma \ref{lem.hol} implies
\begin{equation}\label{ed.osc.est}
\osc_{B_{t}}v \le c\left(\frac{t}{r}\right)^{\alpha_{0}}E(v;r)
\end{equation}
for any $t \in [\rho,r/4]$.  
In particular,
\begin{equation}\label{ed.av.est} 
 \left(\fint_{B_{\rho}}|v-(v)_{B_{\rho}}|^{p-1}\,dx\right)^{1/(p-1)} 
 \le \osc_{B_{\rho}}v \overset{\eqref{ed.osc.est}}{\le} c\left(\frac{\rho}{r}\right)^{\alpha_{0}}E(v;r). 
\end{equation}
In order to estimate the tail term, we split the integral as follows:
\begin{align*}
& [\tail(v-(v)_{B_{\rho}};\rho)]^{p-1} \\
&= \rho^{sp}\int_{\mathbb{R}^{n}\setminus B_{\rho}}\frac{|v(x)-(v)_{B_{\rho}}|^{p-1}}{|x-x_{0}|^{n+sp}}\,dx \\
& = \rho^{sp}\int_{\mathbb{R}^{n}\setminus B_{r/4}}\frac{|v(x)-(v)_{B_{\rho}}|^{p-1}}{|x-x_{0}|^{n+sp}}\,dx + \rho^{sp}\int_{B_{r/4}\setminus B_{\rho}}\frac{|v(x)-(v)_{B_{\rho}}|^{p-1}}{|x-x_{0}|^{n+sp}}\,dx \\
& \eqqcolon I_{1} + I_{2}.
\end{align*}
Let $\mathcal Q_{B_r}(v)\in \R$ be a number satisfying \eqref{def:Q} with $f=v$. Then Lemma \ref{lem.bdd.L1} implies
\begin{equation*}
\left|(v)_{B_{\rho}}-\mathcal{Q}_{B_{r}}(v)\right| \le \fint_{B_{\rho}}|v-\mathcal{Q}_{B_{r}}(v)|\,dx  \le \sup_{B_{r/2}}|v-\mathcal{Q}_{B_{r}}(v)| \le cE(v;r).
\end{equation*}
Using this, we estimate $I_{1}$ as
\begin{align*}
I_{1} & \le c \rho^{sp}\int_{\mathbb{R}^{n}\setminus B_{r/4}}\frac{|v(x)-\mathcal{Q}_{B_{r}}(v)|^{p-1}}{|x-x_{0}|^{n+sp}}\,dx + c\left(\frac{\rho}{r}\right)^{sp}\left|(v)_{B_{\rho}}-\mathcal{Q}_{B_{r}}(v)\right|^{p-1}  \\
& \le c\rho^{sp}\int_{\mathbb{R}^{n}\setminus B_{r}}\frac{|v(x)-\mathcal{Q}_{B_{r}}(v)|^{p-1}}{|x-x_{0}|^{n+sp}}\,dx + c\left(\frac{\rho}{r}\right)^{sp}\fint_{B_{r}}|v-\mathcal{Q}_{B_{r}}(v)|^{p-1}\,dx \\
& \quad\, + c\left(\frac{\rho}{r}\right)^{sp}[E(v;r)]^{p-1} \\
& \le c\left(\frac{\rho}{r}\right)^{sp}[E(v;r)]^{p-1}.
\end{align*}
As for $I_{2}$, we get
\begin{align*}
I_{2}  \le c\int_{\rho}^{r/4}\left(\frac{\rho}{t}\right)^{sp}\left(\osc_{B_{t}}v\right)^{p-1}\frac{dt}{t} 
\overset{\eqref{ed.osc.est}} & {\le} c[E(v;r)]^{p-1}\int_{\rho}^{r/4}\left(\frac{\rho}{t}\right)^{sp}\left(\frac{t}{r}\right)^{\alpha_{0}(p-1)}\frac{dt}{t} \\
& \le \frac{c}{sp-\alpha_{0}(p-1)}\left(\frac{\rho}{r}\right)^{\alpha_{0}(p-1)}[E(v;r)]^{p-1},
\end{align*}
where we have also used the fact that $\alpha_{0} < sp/(p-1)$. 
The above two displays imply
\begin{equation*}
\tail(v-(v)_{B_{\rho}};\rho) \le c\left(\frac{\rho}{r}\right)^{\alpha_0}E(v;r).
\end{equation*}
Connecting this last estimate and \eqref{ed.av.est} to \eqref{ed.v.start}, we obtain the desired estimate.
\end{proof}

\subsection{Comparison estimates} 
Here we derive some comparison estimates. We fix any ball $B_{2r}=B_{2r}(x_0)\subset \mathbb{R}^{n}$, and consider the weak solution $u\in \mathbb{W}^{s,p}(B_{2r})$ to the Dirichlet problem 
\begin{equation*}
\left\{
\begin{aligned}
 -\mathcal{L}_{\Phi}u&=\mu\in L^\infty(\mathbb{R}^n)&\text{in }& B_{2r},\\
 u&=0 &\text{in }&\mathbb{R}^n\setminus B_{2r}.
\end{aligned}
\right.
\end{equation*}
In order to take advantage of symmetry, we consider the operator $-\tilde{\mathcal{L}}_{u}$ defined by
\begin{equation*}
\langle-\tilde{\mathcal{L}}_{u}w,\varphi\rangle \coloneqq \int_{\mathbb{R}^n}\int_{\mathbb{R}^n}|w(x)-w(y)|^{p-2}(w(x)-w(y))(\varphi(x)-\varphi(y))\tilde{K}_{u,\Phi}(x,y) \,dx\,dy,
\end{equation*} 
where
\begin{equation*}
\tilde{K}_{u,\Phi}(x,y) \coloneqq \tfrac{1}{2}\left(\bar{K}_{u,\Phi}(x,y)+\bar{K}_{u,\Phi}(y,x)\right)
\end{equation*}
and 
\begin{equation*}
\bar{K}_{u,\Phi}(x,y) \coloneqq 
\begin{cases}
 \dfrac{\Phi(u(x)-u(y))K(x,y)}{|u(x)-u(y)|^{p-2}(u(x)-u(y))}&\text{ if }  x\not =y ~\text{and}~u(x)\not= u(y), \\ 
 |x-y|^{-n-sp} & \text{ if }  x\not= y~\text{and}~ u(x)=u(y). \\ 
\end{cases}
\end{equation*} 
We then see that $\Lambda^{-2}|x-y|^{-(n+sp)} \leq \tilde{K}_{u,\Phi}(x,y) = \tilde{K}_{u,\Phi}(y,x) \leq \Lambda^{2}|x-y|^{-(n+sp)}$ holds for all $(x,y) \in \mathbb{R}^n \times \mathbb{R}^n$ with $x \neq y$, and that $u$ solves 
\begin{equation*}
\left\{
\begin{aligned}
-\tilde{\mathcal{L}}_{u}u&=\mu\in L^\infty(\mathbb{R}^n)&\text{in }& B_{2r},\\
u&=0 &\text{in }&\mathbb{R}^n\setminus B_{2r},
\end{aligned}
\right.
\end{equation*}
We next consider the weak solution $v\in \mathbb{W}^{s,p}(B_{r})$ to the following Dirichlet problem: 
\begin{equation*}
\left\{
\begin{aligned}
-\tilde{\mathcal{L}}_{u}v&=0&\text{in }& B_{r},\\
v&=u&\text{in }&\mathbb{R}^n\setminus B_{r},
\end{aligned}
\right.
\end{equation*}
and set 
\[ w \coloneqq u-v. \]
Under the above setting, we have the following lemma, see \cite[Lemma 8.4.1]{KuMiSi3}. 
\begin{lemma} There exists a constant $c=c(n,s,p,\Lambda)>0$ such that, for any $\xi>1$ and $d>0$,
\begin{equation*}
 \int_{\mathbb{R}^n}\int_{\mathbb{R}^{n}} \frac{(|u(x)-u(y)|+|v(x)-v(y)|)^{p-2}|w(x)-w(y)|^2}{(d+|w(x)|+|w(y)|)^\xi} \frac{dx\,dy}{|x-y|^{n+sp}}\leq  c\frac{d^{1-\xi}}{\xi-1}|\mu|(B_r).
\end{equation*}
\end{lemma}

Using this lemma, we can obtain a preliminary version of comparison estimate in Lemma \ref{basic.comp.2} below. The estimate is essentially the same as the one given in \cite[Lemma 8.4.4]{KuMiSi3} (see also \cite[Lemma 3.4]{KuMiSi}) except the ranges of $q$ and $p$. Note that, in  both \cite[Lemma 8.4.4]{KuMiSi3} and \cite[Lemma 3.4]{KuMiSi}, $q$ is assumed to satisfy
\begin{equation*}
q\in [1,\bar q),  \qquad \bar{q} \coloneqq \min\left\{\frac{n(p-1)}{n-s},p\right\}.
\end{equation*}
This range is sensible since the papers \cite{KuMiSi,KuMiSi3} primarily assume that $p>2-s/n$, and the restriction $q\ge 1$ is used in order to apply the well-known fractional Sobolev--Poincar\'e inequality with $q\ge1$. However, by using the fractional Sobolev--Poincar\'e inequality in Lemma \ref{lemsobopoin} for $q\in(0,1)$, the proof of \cite[Lemma 8.4.4]{KuMiSi3} carries over to the case $q \in (0,1)$. Consequently, we can extend the admissible ranges of $q$ and $p$.

\begin{lemma}\label{basic.comp.2}
Assume that $1 < p < 2$. Then for any $q \in (0,\bar{q})$, where $\bar{q}$ is given in \eqref{SOLA1}, there exists $h_{0} = h_{0}(n,s,p,q) \in (0,s)$ such that if $h \in (h_{0},s)$, then 
\begin{equation}\label{u-v}
\begin{aligned}
& \left(\int_{B_{tr}}\fint_{B_{tr}}\frac{|w(x)-w(y)|^{q}}{|x-y|^{n+hq}}\,dx\,dy\right)^{1/q}  \\
& \le \frac{c}{(t-1)^{1/[(p-1)q]}r^{h}}\left[\frac{|\mu|(B_{r})}{r^{n-sp}}\right]^{1/(p-1)} \\
& \quad\; + \frac{c}{(t-1)^{1/q}r^{h(p-1)}}\left(\int_{B_{tr}}\fint_{B_{tr}}\frac{|u(x)-u(y)|^{q}}{|x-y|^{n+hq}}\,dx\,dy\right)^{(2-p)/q}\left[\frac{|\mu|(B_{r})}{r^{n-sp}}\right]
\end{aligned}
\end{equation}
holds whenever $t \in (1, 2]$, where $c=c(n,s,p,\Lambda,h,q)>0$.
\end{lemma}
\begin{proof}
The proof proceeds along the lines of \cite[Lemma 3.4]{KuMiSi} and \cite[Lemma 8.4.4]{KuMiSi3}, with the help of the fractional Sobolev--Poincar\'e inequality \eqref{sobo0}.
\end{proof}

In order to handle the second term appearing in the right-hand side of \eqref{u-v}, we establish a Caccioppoli type estimate for $u$. 
The case $2-s/n < p < 2$ was treated in \cite[Lemma 3.5]{KuMiSi}, but we need to modify the approach for lower values of $p$. 
\begin{lemma}
Assume that $1 < p < 2$. Then for any $q \in (0,\bar{q})$, where $\bar{q}$ is given in \eqref{SOLA1}, there exists $h_{0} = h_{0}(n,s,p,q) \in (0,s)$ such that if $h \in (h_{0},s)$, then
\begin{equation*}
\left(\int_{B_{2r}}\fint_{B_{2r}}\frac{|u(x)-u(y)|^{q}}{|x-y|^{n+hq}}\,dx\,dy\right)^{1/q} \le \frac{c}{r^{h}}\left\{ E(u;4r) + \left[\frac{|\mu|(B_{4r})}{r^{n-sp}}\right]^{1/(p-1)} \right\}.
\end{equation*}
holds for a constant $c = c(n,s,p,\Lambda,h,q)>0$. 
\end{lemma}

\begin{proof} 
We may consider the case $q \ge p-1$ only, since the estimate for lower values of $q$ follows from  Lemma \ref{lemembedding}. 
For $\varphi \in W^{h,q}(B_{\rho})$ with $\rho \in (0,4r)$, we denote
\begin{equation*}
F(\varphi;\rho) \coloneqq \left(\int_{B_{\rho}}\fint_{B_{\rho}}\frac{|\varphi(x)-\varphi(y)|^{q}}{|x-y|^{n+hq}}\,dx\,dy\right)^{1/q}.
\end{equation*}
For $2 \le \sigma' < \sigma \le 4$, we consider the weak solution $\tilde{v} = \tilde{v}_{\sigma,\sigma'} \in \mathbb{W}^{s,p}(B_{(\sigma+\sigma')r/2})$ to
\begin{equation*}
\left\{
\begin{aligned}
-\tilde{\mathcal{L}}_{u}\tilde{v} & = 0 & \text{ in } & B_{(\sigma + \sigma')r/2}, \\
\tilde{v} & = u & \text{ in } & \mathbb{R}^{n}\setminus B_{(\sigma+\sigma')r/2}.
\end{aligned}
\right.
\end{equation*}
We start with the obvious inequality 
\begin{equation}\label{F.first}
F(u;\sigma'r) \le cF(\tilde{v};\sigma'r) + cF(u-\tilde{v};\sigma'r).
\end{equation}
Lemma \ref{basic.comp.2}, with the choice $t = 2\sigma/(\sigma+\sigma') \in (1,2]$, directly implies
\begin{equation}\label{F.u-v}
\begin{aligned}
 F(u-\tilde{v};\sigma' r) & \le F(u-\tilde{v};\sigma r) \\
& \le \frac{c}{(\sigma-\sigma')^{1/[(p-1)q]}r^{h}}\left[\frac{|\mu|(B_{4r})}{r^{n-sp}}\right]^{1/(p-1)} \\
& \quad\; + \frac{c}{(\sigma-\sigma')^{1/q}r^{h(p-1)}}[F(u;\sigma r)]^{2-p}\left[\frac{|\mu|(B_{4r})}{r^{n-sp}}\right].
\end{aligned}
\end{equation}
We next apply Lemma \ref{lem.Caccio.Lq} to $\tilde{v}$ in the concentric balls $B_{\sigma'r} \subset B_{(\sigma+\sigma')r/2}$, which gives
\begin{align*}
F(\tilde{v};\sigma'r) 
 \le \frac{cr^{-h}}{(\sigma-\sigma')^{\theta}}\left[\left(\fint_{B_{(\sigma+\sigma')r/2}}|\tilde{v}-k|^{p-1}\,dx\right)^{1/(p-1)} + \tail(\tilde{v}-k;(\sigma+\sigma') r/4)\right]
\end{align*}
for any $k \in \mathbb{R}$. To proceed, we estimate the first term in the right-hand side as
\begin{align*}
& \left(\fint_{B_{(\sigma+\sigma')r/2}}|\tilde{v}-k|^{p-1}\,dx\right)^{1/(p-1)} \\
& \le c\left(\fint_{B_{(\sigma+\sigma')r/2}}|u-\tilde{v}|^{p-1}\,dx\right)^{1/(p-1)} + c\left(\fint_{B_{(\sigma+\sigma')r/2}}|u-k|^{p-1}\,dx\right)^{1/(p-1)} \\
& \le \frac{cr^{h}}{(\sigma-\sigma')^{1/q}}F(u-\tilde{v};\sigma r) + c\left(\fint_{B_{4r}}|u-k|^{p-1}\,dx\right)^{1/(p-1)},
\end{align*}
where we have used \eqref{sobo0} (with the choice $t=2\sigma/(\sigma+\sigma')$) for the last inequality.
As for the second term, we use the facts that $\tilde{v}=u$ in $\mathbb{R}^n \setminus B_{(\sigma+\sigma')r/2}$ and that $\sigma+\sigma \ge 4$, and once again \eqref{sobo0} in order to get
\begin{align*}
& \tail(\tilde{v}-k;(\sigma+\sigma')r/4) \\
& \le c\tail(\tilde{v}-u;(\sigma+\sigma') r/4) + c\tail(u-k;(\sigma+\sigma') r/4) \\
& \le c\left(\fint_{B_{(\sigma+\sigma')r/2}}|\tilde{v}-u|^{p-1}\,dx\right)^{1/(p-1)} + c\tail(u-k;r) \\
& \le \frac{cr^{h}}{(\sigma-\sigma')^{1/q}}F(u-\tilde{v};\sigma r) + c\left(\fint_{B_{4r}}|u-k|^{p-1}\,dx\right)^{1/(p-1)} + c\tail(u-k;4r).
\end{align*}
Note that, in the last inequality, the term $\tail(u-k;r)$ has been estimated by an elementary calculation similar to that in the proof of Lemma \ref{exc.property}. 
Combining the above three displays and then taking the infimum with respect to $k$, we obtain
\begin{equation}\label{Fv}
F(\tilde{v};\sigma'r) \le \frac{c}{(\sigma-\sigma')^{\theta+1/q}}\left[r^{-h}E(u;4r)+F(u-\tilde{v};\sigma r)\right].
\end{equation}
In turn, connecting \eqref{F.u-v} and \eqref{Fv} to \eqref{F.first}, we arrive at
\begin{align*}
F(u;\sigma'r) & \le \frac{cr^{-h}}{(\sigma-\sigma')^{\theta +p'/q}}\left\{ E(u;4r) + \left[\frac{|\mu|(B_{4r})}{r^{n-sp}}\right]^{1/(p-1)} \right\} \\
& \quad\; + \frac{cr^{-h(p-1)}}{(\sigma-\sigma')^{\theta+2/q}}[F(u;\sigma r)]^{2-p}\left[\frac{|\mu|(B_{4r})}{r^{n-sp}}\right].
\end{align*}
We apply Young's inequality to the last term in the right-hand side, with conjugate exponents $1/(2-p)$ and $1/(p-1)$, in order to see that
\begin{equation*}
F(u;\sigma'r) \le \frac{1}{2}F(u;\sigma r) + \frac{cr^{-h}}{(\sigma-\sigma')^{(\theta+2/q)/(p-1)}}\left\{ E(u;4r) + \left[\frac{|\mu|(B_{4r})}{r^{n-sp}}\right]^{1/(p-1)} \right\}
\end{equation*}
holds for some $c= c(n,s,p,\Lambda,q)$, whenever $2 \le \sigma' \le \sigma \le 4$. Then, applying a standard iteration lemma (see for instance \cite[Lemma~6.1]{Giu}) to the above display, we can drop the first term in the right-hand side. Recalling the definition of $F(\cdot)$, we conclude with the desired estimate.
\end{proof}

The above two lemmas, along with Lemma \ref{lemsobopoin}, imply the following:

\begin{lemma}\label{final.comp}
Assume that $1<p<2$. Then for any $q \in (0,\bar{q})$, where $\bar{q}$ is given in \eqref{SOLA1}, there exists $h_{0} = h_{0}(n,s,p,q) \in (0,s)$ such that if $h \in (h_{0},s)$, then 
\begin{align*}
& \left(\int_{B_{2r}}\fint_{B_{2r}}\frac{|w(x)-w(y)|^q}{|x-y|^{n+hq}}\,dx\,dy\right)^{1/q} \\
& \le \frac{c}{r^h}\left[\frac{|\mu|(B_{4r})}{r^{n-sp}}\right]^{1/(p-1)} + \frac{c}{r^{h}}[E(u;4r)]^{2-p}\left[\frac{|\mu|(B_{4r})}{r^{n-sp}}\right]
\end{align*}
holds for a constant $c=c(n,s,p,\Lambda,h,q)>0$. 
Moreover, for any $\tilde{q}\in(0, q_{0})$, where $q_{0}$ is given in \eqref{z1}, there exists a constant $c = c(n,s,p,\tilde{q},\Lambda)>0$ such that 
\begin{equation*}
\left(\fint_{B_{r}}|w|^{\tilde{q}}\,dx\right)^{1/\tilde{q}} \le c\left[\frac{|\mu|(B_{4r})}{r^{n-sp}}\right]^{1/(p-1)} + c[E(u;4r)]^{2-p}\left[\frac{|\mu|(B_{4r})}{r^{n-sp}}\right].
\end{equation*}
\end{lemma}

\subsection{Pointwise and oscillation estimates for SOLA} 
We first obtain an excess decay estimate for $u$.
\begin{lemma}\label{excess.decay}
Let $u$ be a SOLA to \eqref{EQ}. Then there exist positive constants $c$ and $\tau$, both depending only on $n$, $s$, $p$ and $\Lambda$, such that
\begin{equation}\label{ed.u}
E(u;\sigma\rho) \le c\sigma^{\alpha_{0}}E(u;\rho) + c\sigma^{-\tau}\left[ \frac{|\mu|(B_{\rho})}{\rho^{n-sp}} \right]^{1/(p-1)}
\end{equation}
holds whenever $B_{\sigma\rho} \subset B_{\rho} \subset \Omega$ are concentric balls with $\sigma \in (0,1)$, where $\alpha_{0}$ is the exponent given in Lemma \ref{lem.hol}.
\end{lemma}
\begin{proof}
Note that \eqref{ed.u} for $p \ge 2$ follows from \cite[Lemma~8.5.2]{KuMiSi3} via \eqref{exc.equiv}; we thus consider the case $1<p<2$ only. Moreover, without loss of generality, we may assume that $\sigma \in (0,1/4]$.

\textit{Step 1.}  
Let $\{u_{j}\}$ be an approximating sequence for $u$ with corresponding source terms $\mu_{j}$, as described in Definition \ref{defsola}. We consider the weak solution $v_{j}$ to
\begin{equation*}
\left\{
\begin{aligned}
-\tilde{\mathcal{L}}_{u_j}v_{j} & = 0 & \text{in } & B_{\rho/4}, \\
v_{j} & = u_{j} & \text{in } & \mathbb{R}^{n}\setminus B_{\rho/4}.
\end{aligned}
\right.
\end{equation*}
Since $v_{j} = u_{j}$ in $\mathbb{R}^{n}\setminus B_{\rho/4}$, 
we have for any $t \le \rho/4$
\begin{equation*}
E(u_{j}-v_{j};t) \le c\left(\frac{\rho}{t}\right)^{n/(p-1)}\left(\fint_{B_{\rho/4}}|u_{j}-v_{j}|^{p-1}\,dx\right)^{1/(p-1)}.
\end{equation*}
Using this inequality, along with Lemmas \ref{exc.property} and \ref{thm.ed.v}, we get
\begin{align*}
E(u_{j};\sigma \rho) & \le cE(v_{j};\sigma \rho) + cE(u_{j}-v_{j}; \sigma\rho) \\
& \le cE(v_{j};\sigma\rho) + c\sigma^{-n/(p-1)}\left(\fint_{B_{\rho/4}}|u_{j}-v_{j}|^{p-1}\,dx\right)^{1/(p-1)} \\
& \le c\sigma^{\alpha_{0}}E(v_{j};\rho/4) + c\sigma^{-n/(p-1)}\left(\fint_{B_{\rho/4}}|u_{j}-v_{j}|^{p-1}\,dx\right)^{1/(p-1)} \\
& \le c\sigma^{\alpha_{0}}E(u_{j};\rho/4) + c\sigma^{-n/(p-1)}\left(\fint_{B_{\rho/4}}|u_{j}-v_{j}|^{p-1}\,dx\right)^{1/(p-1)}.
\end{align*}
In order to estimate the last term in the right-hand side, 
we use Lemma \ref{final.comp} and Young's inequality with conjugate exponents $1/(2-p)$ and $1/(p-1)$, which give
\begin{align*}
\left(\fint_{B_{\rho/4}}|u_{j}-v_{j}|^{p-1}\,dx\right)^{1/(p-1)} & \le c\left[\frac{|\mu_{j}|(B_{\rho})}{\rho^{n-sp}}\right]^{1/(p-1)} + c[E(u_{j};\rho)]^{2-p}\left[\frac{|\mu_{j}|(B_{\rho})}{\rho^{n-sp}}\right] \\
& \le \varepsilon E(u_{j};\rho) + c_{\varepsilon}\left[\frac{|\mu_{j}|(B_{\rho})}{\rho^{n-sp}}\right]^{1/(p-1)}
\end{align*}
for any $\varepsilon \in (0,1)$. Choosing $\varepsilon = \sigma^{\alpha_{0} + n/(p-1)}$, we obtain \eqref{ed.u} for $u_{j}$.

\textit{Step 2.} We now show that 
\begin{equation}\label{exc.conv}
\lim_{j\to\infty} E(u_{j};t) = E(u;t)
\end{equation}
for any ball $B_{t} \subset \Omega$. 
Since
\begin{equation*}
E(u_{j};t) \le \left(\fint_{B_{t}}|u_{j}-\mathcal{Q}_{B_{t}}(u)|^{p-1}\,dx\right)^{1/(p-1)} + \tail(u_{j}-\mathcal{Q}_{B_{t}}(u);t),
\end{equation*}
we have
\[
\limsup_{j\to\infty}E(u_{j};t) \le E(u;t). 
\]
On the other hand, using the inequalities $(a+b)^\beta \le a^\beta +b^\beta$ and 
$(a+b)^{\tilde{\beta}} - a^{\tilde{\beta}} = \int_{a}^{a+b}\tilde{\beta} t^{\tilde{\beta}-1}\,dx \le \tilde{\beta}(a+b)^{\tilde{\beta}-1}b$
for any $a,b>0$, $\beta\in(0,1)$ and $\tilde{\beta}>1$, we get
\begin{align*}
& E(u;t) \le \left(\fint_{B_{t}}|u-\mathcal{Q}_{B_{t}}(u_{j})|^{p-1}\,dx\right)^{1/(p-1)} + \tail(u-\mathcal{Q}_{B_{t}}(u_{j});t) \\
& \le \left(\fint_{B_{t}}|u_{j}-\mathcal{Q}_{B_{t}}(u_{j})|^{p-1}\,dx + \fint_{B_{t}}|u-u_{j}|^{p-1}\,dx \right)^{1/(p-1)} \\
& \quad\; + \left([\tail(u_{j}-\mathcal{Q}_{B_{t}}(u_{j});t)]^{p-1} + [\tail(u-u_{j};t)]^{p-1}  \right)^{1/(p-1)}\\
 & \le \left(\fint_{B_{t}}|u_{j}-\mathcal{Q}_{B_{t}}(u_{j})|^{p-1}\,dx\right)^{1/(p-1)} + [\tail(u_{j}-\mathcal{Q}_{B_{t}}(u_{j});t)]\\
& \quad\;  + \frac{1}{p-1}\left(\fint_{B_{t}}|u_{j}-\mathcal{Q}_{B_{t}}(u_{j})|^{p-1}\,dx + \fint_{B_{t}}|u-u_{j}|^{p-1}\,dx\right)^{\frac{2-p}{p-1}}\fint_{B_{t}}|u-u_{j}|^{p-1}\,dx \\
& \quad\; + \frac{1}{p-1}\left([\tail(u_{j}-\mathcal{Q}_{B_{t}}(u_{j});t)]^{p-1} + [\tail(u-u_{j};t)]^{p-1}  \right)^{\frac{2-p}{p-1}}[\tail(u-u_{j};t)].
\end{align*}
Letting $j\to\infty$ in the last display, we obtain
\[
 \liminf_{j\to\infty}E(u_{j};t) \ge E(u;t). 
\]
We thus have \eqref{exc.conv}, and the proof of \eqref{ed.u} is complete.
\end{proof}

Now, we are ready to prove Theorem \ref{thmosc}. 
The theorem is in fact related to pointwise estimates for certain fractional sharp maximal functions of $u$. Here we define the {\it $R$-truncated nonlocal fractional sharp maximal function (with order $\alpha$)} of $u$ by
\begin{equation*}
\mathbf{N}^{\sharp}_{\alpha,R}[u](x) \coloneqq \sup_{0<r\le R}r^{-\alpha}E(u;x,r).
\end{equation*}

\begin{proof}[Proof of Theorem \ref{thmosc}] 
\textit{Step 1.} 
Let us first show that
\begin{equation}\label{nonlocal.max.est}
\mathbf{N}^{\sharp}_{\alpha,R}[u](x) \le c\left[\mathbf{M}_{sp-\alpha(p-1),R}[\mu](x)\right]^{1/(p-1)} + cR^{-\alpha}E(u;x,R)
\end{equation}
holds for any $0 \le \alpha < \alpha_{0}$ and any ball $B_{R}(x) \subset \Omega$, where $c=c(n,s,p,\Lambda,\alpha)>0$.

We fix any radius $\rho \in (0,R]$. Then, with $\sigma \in (0,1/2)$ being a free parameter to be chosen in a few lines, we apply Lemma \ref{excess.decay} to the concentric balls $B_{\sigma \rho}(x) \subset B_{\rho}(x)$.
Multiplying both sides of the resulting inequality by $(\sigma\rho)^{-\alpha}$, we have 
\begin{align*}
(\sigma\rho)^{-\alpha}E(u;x,\sigma\rho) & \le c_{*}\sigma^{\alpha_{0}-\alpha}\rho^{-\alpha}E(u;x,\rho) + c\sigma^{-\tau-\alpha}\rho^{-\alpha}\left[\frac{|\mu|(B_{\rho}(x))}{\rho^{n-sp}}\right]^{1/(p-1)} \\
& = c_{*}\sigma^{\alpha_{0}-\alpha}\rho^{-\alpha}E(u;x,\rho) + c\sigma^{-\tau-\alpha}\left[\frac{|\mu|(B_{\rho}(x))}{\rho^{n-sp+\alpha(p-1)}}\right]^{1/(p-1)},
\end{align*}
where $c$, $c_{*}$ and $\tau$ are positive constants depending only on $n$, $s$, $p$ and $\Lambda$. 
We now choose $\sigma = \sigma(n,s,p,\Lambda,\alpha)>0$ so small that 
$c_{*}\sigma^{\alpha_{0}-\alpha} = 1/2$, 
which yields
\begin{align*}
(\sigma\rho)^{-\alpha}E(u;x,\sigma\rho) & \le \frac{1}{2}\rho^{-\alpha}E(u;x,\rho) + c\left[\frac{|\mu|(B_{\rho}(x))}{\rho^{n-sp+\alpha(p-1)}}\right]^{1/(p-1)} \\
& \le \frac{1}{2}\mathbf{N}^{\sharp}_{\alpha,R}[u](x) + c\left[\mathbf{M}_{sp-\alpha(p-1),R}[\mu](x)\right]^{1/(p-1)}.
\end{align*}
Since $\rho \in (0, R]$ was arbitrary, we have
\begin{equation*}
\sup_{0 < r \le \sigma R}r^{-\alpha}E(u;x,r) \le \frac{1}{2}\mathbf{N}^{\sharp}_{\alpha,R}[u](x) + c\left[\mathbf{M}_{sp-\alpha(p-1),R}[\mu](x)\right]^{1/(p-1)}.
\end{equation*}
On the other hand, Lemma \ref{exc.property} (2) implies
\[ \sup_{\sigma R < r \le R}r^{-\alpha}E(u;x,r) \le c\sigma^{-\alpha-n/(p-1)}R^{-\alpha}E(u;x,R). \]
From the last two displays, we conclude with
\[ \mathbf{N}^{\sharp}_{\alpha,R}[u](x) \le \frac{1}{2}\mathbf{N}^{\sharp}_{\alpha,R}[u](x) + cR^{-\alpha}E(u;x,R) + c\left[\mathbf{M}_{sp-\alpha(p-1),R}[\mu](x)\right]^{1/(p-1)}, \]
which leads to \eqref{nonlocal.max.est}.

\textit{Step 2.} Let $B_{R} = B_{R}(x_{0}) \subset \Omega$ and $x,y \in B_{R/8}$ be as in the statement of the theorem; let $r$ be a radius such that $r \le R/8$. 
We start by integrating \eqref{ed.u} with respect to Haar measure and then making an elementary manipulation, to get
\[ \int_{\rho}^{r}E(u;x,\sigma t)\frac{dt}{t} \le c_{**}\sigma^{\alpha_{0}}\int_{\rho}^{r}E(u;x,t)\frac{dt}{t} + c\sigma^{-\tau}\int_{\rho}^{r}\left[\frac{|\mu|(B_{t}(x))}{t^{n-sp}}\right]^{1/(p-1)}\frac{dt}{t} \]
for any $\rho \in (0,r]$, where $c$ and $c_{**}$ are positive constants depending only on $n$, $s$, $p$ and $\Lambda$. 
Taking the constant $\sigma = \sigma(n,s,p,\Lambda)>0$ so small that
$c_{**}\sigma^{\alpha_{0}} = 1/2$,
and then changing variables, the above inequality becomes
\begin{equation*} 
\int_{\sigma\rho}^{\sigma r}E(u;x,t)\frac{dt}{t} \le \frac{1}{2}\int_{\rho}^{r}E(u;x,t)\frac{dt}{t} + c\int_{\rho}^{r}\left[\frac{|\mu|(B_{t}(x))}{t^{n-sp}}\right]^{1/(p-1)}\frac{dt}{t}.
\end{equation*}
We thus have
\begin{align*}
\int_{\sigma\rho}^{r}E(u;x,t)\frac{dt}{t} 
\le \frac{1}{2}\int_{\sigma\rho}^{r}E(u;x,t)\frac{dt}{t} 
+ \int_{\sigma r}^{r}E(u;x,t)\frac{dt}{t} + c\int_{\rho}^{r}\left[\frac{|\mu|(B_{t}(x))}{t^{n-sp}}\right]^{1/(p-1)}\frac{dt}{t}
\end{align*}
and, after reabsorbing terms, 
\begin{equation*}
\int_{\sigma\rho}^{r}E(u;x,t)\frac{dt}{t} \le 2\int_{\sigma r}^{r}E(u;x,t)\frac{dt}{t} + 2c\int_{\rho}^{r}\left[\frac{|\mu|(B_{t}(x))}{t^{n-sp}}\right]^{1/(p-1)}\frac{dt}{t}.
\end{equation*}
Using Lemma \ref{exc.property} (2) and a few elementary manipulations, we arrive at
\begin{equation}\label{int.result}
\int_{\rho}^{r}E(u;x,t)\frac{dt}{t} \le cE(u;x,r) + c\int_{\rho}^{r}\left[\frac{|\mu|(B_{t}(x))}{t^{n-sp}}\right]^{1/(p-1)}\frac{dt}{t}
\end{equation}
for some $c=c(n,s,p,\Lambda)>0$.

We now fix any two radii $\rho,\tilde{\rho}$ satisfying $0 < \tilde{\rho} \le \rho/2 \le r/8$, and then choose $k\in\mathbb{N}$ and $\theta \in (1/4,1/2]$ such that $\tilde{\rho} = \theta^{k}\rho$. 
We recall \eqref{def.P}. Then, since
\begin{align*} 
\big| \mathcal{P}_{B_{\theta^{j}\rho}(x)}(u) - \mathcal{P}_{B_{\theta^{j+1}\rho}(x)}(u) \big| \overset{\eqref{av.min}}&{\le} c\left(\fint_{B_{\theta^{j+1}\rho}(x)}\big|u-\mathcal{P}_{B_{\theta^{j}\rho}(x)}(u)\big|^{p-1}\,dx\right)^{1/(p-1)} \\
 & \le c\theta^{-n/(p-1)}A(u;x,\theta^{j}\rho),
\end{align*}
we have
\begin{align*}
\big| \mathcal{P}_{B_{\rho}(x)}(u) - \mathcal{P}_{B_{\tilde{\rho}}(x)}(u) \big| & \le \sum_{j=0}^{k-1}\big| \mathcal{P}_{B_{\theta^{j}\rho}(x)}(u) - \mathcal{P}_{B_{\theta^{j+1}\rho}(x)}(u) \big| \\
& \le c\theta^{-n/(p-1)}\sum_{j=0}^{k-1}A(u;x,\theta^{j}\rho)
\le c \theta^{-n/(p-1)}\sum_{j=0}^{k-1}E(u;x,\theta^{j}\rho).
\end{align*}
We also recall the elementary inequalities
\begin{align*}
\sum_{j=0}^{k-1}E(u;x,\theta^{j}\rho) & = \frac{1}{\log(1/\theta)}\sum_{j=0}^{k-1}\int_{\theta^{j}\rho}^{\theta^{j-1}}E(u;x,\theta^{j}\rho)\frac{dt}{t} \\
& \le c\sum_{j=0}^{k-1}\int_{\theta^{j}\rho}^{\theta^{j-1}\rho}E(u;x,t)\frac{dt}{t} \le c\int_{\tilde{\rho}}^{\rho/\theta}E(u;x,t)\frac{dt}{t}
\end{align*}
and
\begin{equation*}
\int_{\rho}^{r}\left[\frac{|\mu|(B_{t}(x))}{t^{n-sp}}\right]^{1/(p-1)}\frac{dt}{t}  \le r^{\alpha}\int_{\rho}^{r}\left[\frac{|\mu|(B_{t}(x))}{t^{n-sp+\alpha(p-1)}}\right]^{1/(p-1)}\frac{dt}{t} 
 \le r^{\alpha}\mathbf{W}^{R}_{s-\alpha(p-1)/p,p}[\mu](x).
\end{equation*}
Then, using \eqref{int.result}, we have
\begin{equation}\label{ave.cauchy}
\big| \mathcal{P}_{B_{\rho}(x)}(u) - \mathcal{P}_{B_{\tilde{\rho}}(x)}(u) \big| \le c\int_{\tilde{\rho}}^{\rho/\theta}E(u;x,t)\frac{dt}{t}
\end{equation}
and so
\begin{equation}\label{ave.pot}
    \big|\mathcal{P}_{B_{\rho}(x)}(u)-\mathcal{P}_{B_{\tilde{\rho}}(x)}(u)\big| \le cE(u;x,r) + cr^{\alpha}\mathbf{W}^{R}_{s-\alpha(p-1)/p,p}[\mu](x).
\end{equation}
Here we note that, 
by the absolute continuity of the integral, \eqref{ave.cauchy} implies that $\{\mathcal{P}_{B_{\rho}(x)}(u)\}$ is a Cauchy net. In turn, the limit 
\[ u(x) \coloneqq \lim_{\rho\rightarrow0}\mathcal{P}_{B_{\rho}(x)}(u) \] 
exists and therefore defines the precise representative of $u$ at $x$. Now we let $\tilde{\rho}\rightarrow 0$ in \eqref{ave.pot} and take $\rho = r/4$ in order to have
\begin{equation*}
\big|\mathcal{P}_{B_{r/4}(x)}(u) - u(x) \big| \le cE(u;x,r) + cr^{\alpha}\mathbf{W}^{R}_{s-\alpha(p-1)/p,p}[\mu](x).
\end{equation*}
Combining this with the elementary estimate
\begin{equation*}
\big| \mathcal{P}_{B_{r}(x)}(u) - \mathcal{P}_{B_{r/4}(x)}(u) \big| \overset{\eqref{av.min}}{\le} cA(u;x,r) \le cE(u;x,r),
\end{equation*}
we arrive at
\[ \left|u(x) - \mathcal{P}_{B_{r}(x)}(u) \right| \le cE(u;x,r) + cr^{\alpha}\mathbf{W}^{R}_{s-\alpha(p-1)/p,p}[\mu](x). \]
Writing this estimate with $y$ instead of $x$, i.e.,
\[ \left|u(y) - \mathcal{P}_{B_{r}(y)}(u) \right| \le cE(u;y,r) + cr^{\alpha}\mathbf{W}^{R}_{s-\alpha(p-1)/p,p}[\mu](y) \]
and merging the last two displays, we obtain
\begin{align*}
|u(x)-u(y)| & \le c\left|\mathcal{P}_{B_{r}(x)}(u)-\mathcal{P}_{B_{r}(y)}(u)\right| + cE(u;x,r) + cE(u;y,r) \\
& \quad + cr^{\alpha}\left[\mathbf{W}^{R}_{s-\alpha(p-1)/p,p}[\mu](x) + \mathbf{W}^{R}_{s-\alpha(p-1)/p,p}[\mu](y)\right].
\end{align*}
We now take $r = |x-y|/2$. Then, since $B_{r}(y) \subset B_{3r}(x)$, we have
\begin{align*}
& \left|\mathcal{P}_{B_{r}(x)}(u)-\mathcal{P}_{B_{r}(y)}(u)\right|  \le \left|\mathcal{P}_{B_{r}(x)}(u)-\mathcal{P}_{B_{3r}(x)}(u)\right| + \left|\mathcal{P}_{B_{3r}(x)}(u)-\mathcal{P}_{B_{r}(y)}(u)\right|  \\
\overset{\eqref{av.min}}&{\le} c\left(\fint_{B_{r}(x)}\left|u-\mathcal{P}_{B_{3r}(x)}(u)\right|^{p-1}\,d\tilde{x}\right)^{1/(p-1)} 
+ c\left(\fint_{B_{r}(y)}\left|u-\mathcal{P}_{B_{3r}(x)}\right|^{p-1}\,d\tilde{x}\right)^{1/(p-1)}  \\
& \le cA(u;x,3r) \le cE(u;x,3r)
\end{align*}
and
\begin{equation*}
E(u;x,r) + E(u;y,r) \le cE(u;x,3r).
\end{equation*}
Therefore, we obtain
\begin{equation*}
|u(x)-u(y)|  \le cr^{\alpha}\left[\mathbf{W}^{R}_{s-\alpha(p-1)/p,p}[\mu](x) + \mathbf{W}^{R}_{s-\alpha(p-1)/p,p}[\mu](y)\right] + cE(u;x,3r).
\end{equation*}
To estimate the second term in the right-hand side, we proceed as
\begin{equation*} 
\begin{aligned}
E(u;x,3r) & \le cr^{\alpha}\mathbf{N}^{\sharp}_{\alpha,R/2}[u](x) \\
 \overset{\eqref{nonlocal.max.est}}&{\le} cr^{\alpha}\left\{\left[\mathbf{M}_{sp-\alpha(p-1),R/2}[\mu](x)\right]^{1/(p-1)} + R^{-\alpha}E(u;x,R/2) \right\}.
\end{aligned}
\end{equation*}
Here we have
\begin{equation*}
E(u;x,R/2)  \le cE(u;x_{0},R) 
 \le c\left(\fint_{B_{R}(x_{0})}|u|^{p-1}\,d\tilde{x}\right)^{1/(p-1)} + c\tail(u;x_{0},R)
\end{equation*}
and, by Lemma \ref{max.ftn.pot},
\begin{equation*}
\left[\mathbf{M}_{sp-\alpha(p-1),R/2}[\mu](x)\right]^{1/(p-1)} \le c\mathbf{W}^{R}_{s-\alpha(p-1)/p,p}[\mu](x).
\end{equation*}
Combining the above four displays, we conclude with \eqref{univ.wolff}. 
\end{proof}

\begin{proof}[Proof of Theorem \ref{thmupper1}]
We revisit the proof of Theorem \ref{thmosc}, with $r$ being replaced by $R$.
Letting $\rho = R/4$ and $\tilde{\rho} = \theta^{k}\rho$, we obtain
\begin{align*}
\hspace{-1mm} \left(\fint_{B_{\tilde{\rho}}(x)}|u|^{p-1}\,d\tilde{x}\right)^{1/(p-1)} 
& \le c\left(\fint_{B_{\tilde{\rho}}(x)}\big|u-\mathcal{P}_{B_{\tilde{\rho}}(x)}(u)\big|^{p-1}\,d\tilde{x}\right)^{1/(p-1)} + c\big|\mathcal{P}_{B_{\tilde{\rho}}(x)}(u)\big| \\
& \le cE(u;x,\tilde{\rho}) + c\big|\mathcal{P}_{B_{\tilde{\rho}}(x)}(u) - \mathcal{P}_{B_{R/4}(x)}(u)\big| + c\big|\mathcal{P}_{B_{R/4}(x)}(u)\big| \\
\overset{\eqref{ave.cauchy}}&{\le} cE(u;x,R) + c\int_{\tilde{\rho}}^{R}E(u;x,t)\,\frac{dt}{t} + c\big|\mathcal{P}_{B_{R/4}(x)}(u)\big| \\
& \le c\mathbf{W}^{R}_{s,p}[\mu](x) + c\left(\fint_{B_{R}(x)}|u|^{p-1}\,d\tilde{x}\right)^{1/(p-1)} + c\tail(u;x,R),
\end{align*}
where for the last inequality, we have used \eqref{int.result} as well as the facts that
\begin{equation*}
\int_{\tilde{\rho}}^{R}\left[\frac{|\mu|(B_{t}(x))}{t^{n-sp}}\right]^{1/(p-1)}\,\frac{dt}{t} \le \mathbf{W}^{R}_{s,p}[\mu](x)
\end{equation*}
and that
\begin{equation*}
\big|\mathcal{P}_{B_{R/4}(x)}(u)\big| \overset{\eqref{av.min}} {\le} c\left(\fint_{B_{R/4}(x)}|u|^{p-1}\,d\tilde{x}\right)^{1/(p-1)} \le c\left(\fint_{B_{R}(x)}|u|^{p-1}\,d\tilde{x}\right)^{1/(p-1)}.
\end{equation*}
Since $k$ was arbitrary, we deduce
\begin{align*} 
\left(\fint_{B_{r}(x)}|u|^{p-1}\,d\tilde{x}\right)^{1/(p-1)} 
 \le c\mathbf{W}^{R}_{s,p}[\mu](x) + c\left(\fint_{B_{R}(x)}|u|^{p-1}\,d\tilde{x}\right)^{1/(p-1)} + c\tail(u;x,R)
\end{align*}
for any $r \in (0, R]$, and the proof of Theorem \ref{thmupper1} is complete.
\end{proof}

\begin{remark}
In the above proof, we actually proved that
\begin{equation*}
\left(\mathbf{M}_{R}[|u|^{p-1}](x)\right)^{1/(p-1)} \le c\mathbf{W}^{R}_{s,p}[\mu](x) + c\left(\fint_{B_{R}(x)}|u|^{p-1}\,d\tilde{x}\right)^{1/(p-1)} + c\tail(u;x,R)
\end{equation*}
holds whenever $B_{R}(x) \subset \Omega$ and the right-hand side is finite, where $\mathbf{M}_{R}$ is the $R$-truncated Hardy--Littlewood maximal operator and $c=c(n,s,p,\Lambda)>0$.
\end{remark}

\subsection{Global pointwise estimates for nonnegative SOLA} 
\begin{proof}[Proof of Corollary \ref{2hvTH4}] 
The lower bound in \eqref{global.wolff} directly follows from Theorem \ref{thmlower}. We now prove the upper bound. Fix any $x\in\Omega$, and set  $B=B_{R_0}(x)$ with $R_0=\mathrm{diam}(\Omega)$. 
Let $\{u_{j}\}$ be an approximating sequence for $u$ with corresponding source terms $\mu_{j}$, as described in Definition \ref{defsola}. We consider the weak solution $\tilde{u}_{j}$ to
\begin{equation*}
\left\{
\begin{aligned}
-\mathcal{L}_\Phi \tilde{u}_{j} &=\mu_j &\text{in } & B,\\
\tilde{u}_{j}&=0&\text{in }&\mathbb{R}^n\setminus B.
\end{aligned}
\right.
\end{equation*}
Since $0 = u_j \le \tilde{u}_j$ a.e. in $\mathbb{R}^n \setminus \Omega$, we can take $\varphi = (u_{j} - \tilde{u}_{j})_{+}$ as a test function in the equations satisfied by $u_j$ and $\tilde{u}_j$. Then, arguing exactly as in the proof of Proposition \ref{pro1}, with $\mu_{1,j} = \mu_{2,j} = \mu_j$, we see that $\{\tilde{u}_j\}$ converges to a SOLA $\tilde{u}$ to
\begin{equation*}
\left\{
\begin{aligned}
-\mathcal{L}_\Phi \tilde{u} &=\mu &\text{in } & B,\\
\tilde{u} &=0&\text{in }&\mathbb{R}^n\setminus B.
\end{aligned}
\right.
\end{equation*}
satisfying $u \leq \tilde{u}$ a.e. in $\Omega$. 
Also, Theorem \ref{thmupper1} implies
\begin{equation}\label{es1}
\tilde{u}(x)\leq c{\bf W}_{s,p}^{R_0}[\mu](x) + c\left(\fint_{B}\tilde{u}^{p-1}\,dy\right)^{1/(p-1)}
\end{equation}
for some $c=c(n,s,p,\Lambda)>0$. 
On the other hand, by \eqref{global} and \eqref{sobo0},
\begin{equation*}
\left(\fint_{B}\tilde{u}^{p-1}\,dy\right)^{1/(p-1)}\leq c\left[\frac{\mu(B)}{R_0^{n-sp}}\right]^{1/(p-1)}.
\end{equation*}
Combining this and \eqref{es1}, we get 
\begin{equation*}
u(x)\leq \tilde{u}(x) \leq c{\bf W}_{s,p}^{R_0}[\mu](x) + c\left[\frac{\mu(B)}{R_0^{n-sp}}\right]^{1/(p-1)}\leq c{\bf W}_{s,p}^{2R_0}[\mu](x)
\end{equation*}
for a.e. $x\in \Omega$, and the proof is complete. 
\end{proof}

\begin{proof}[Proof of Corollary \ref{2hvTH4-}] 
As described in Definition \ref{defsola2}, there exist a sequence of open subsets $\{\Omega_j\}_{j=1} ^\infty$ in $\mathbb{R}^n$ and a sequence of weak solutions $\{u_j\}_{j=1}^{\infty} \subset \mathbb{W}^{s,p}(\Omega_{j})$ to the Dirichlet problems 
\begin{equation*}
\left\{
\begin{aligned}
-\mathcal{L}_\Phi u_j&=\mu_{j} &\text{in }&\Omega_j,\\
u_j&=0&\text{in }&\mathbb{R}^n\setminus\Omega_j,
\end{aligned}
\right.
\end{equation*}
such that $\{u_j\}$ converges to $u$ a.e in $\mathbb{R}^n$ and locally in $L^q(\mathbb{R}^n)$ for any $q \in (0,\bar{q})$. Here the sequence $\{\mu_j\}_{j=1}^\infty\subset L^{\infty}(\mathbb{R}^n) \cap L^{1}(\mathbb{R}^{n})$ converges to $\mu$ in  $\mathbb{R}^n$ weakly in the sense of measure with \eqref{limsupmuj}. 
By Corollary \ref{2hvTH4}, we have 
\begin{equation*}
\frac{1}{C_{0}}{\bf W}_{s,p}^{\mathrm{dist}(x,\partial\Omega_j)/8}[\mu_j](x) \leq u_j(x) \leq C_{0}{\bf W}_{s,p}[\mu_{j}](x) \quad \text{for a.e. } x\in \Omega_j,
\end{equation*}  
which implies  for any $j>j_1>j_0$
\begin{equation*}
\frac{1}{C_{0}}{\bf W}_{s,p}^{\mathrm{dist}(\Omega_{j_0},\partial\Omega_{j_1})/8}[\mu_j](x) \leq u_j(x) \leq C_{0}{\bf W}_{s,p}[\mu_{j}](x) \quad \text{for a.e. } x\in \Omega_{j_0}. 
\end{equation*}  
Letting $j\to\infty$ and recalling the convergence properties of $\{u_j\}$ and $\{\mu_j\}$, we get 
\begin{equation*}
\frac{1}{C_{0}}{\bf W}_{s,p}^{\mathrm{dist}(\Omega_{j_0},\partial\Omega_{j_1})/8}[\mu](x) \leq u(x) \leq C_{0}{\bf W}_{s,p}[\mu](x) \quad \text{for a.e. } x\in \Omega_{j_0}. 
\end{equation*}  
Finally, letting $j_1\to \infty$ and then $j_0\to \infty$, we conclude with \eqref{belowesRn}. 
\end{proof}

\section{Nonlocal equations of Lane--Emden type}

\subsection{Auxiliary lemmas}
We obtain pointwise estimates for a sequence of functions satisfying recurrence inequalities involving Wolff potentials. We start with the case of power function $P(u)=u^\gamma$.

\begin{lemma}\label{th1a} Let $\gamma>p-1$, $C_{*}>0$ and $\mu\in \mathcal{M}^+(\mathbb R^n)$ with $\mathrm{supp}\,\mu \subset B_{R}(0)$ for some $R>1$. 
There exists a small constant $\delta>0$, depending only on $n$, $s$, $p$, $\gamma$, $C_{*}$ and $R$, such that if the inequality 
\begin{equation}\label{th1a-2}
\mu(K)\leq \delta\mathrm{Cap}_{\mathbf{G}_{s p},\frac{\gamma}{\gamma-p+1}}(K)
\end{equation}
holds for any compact set $K\subset\mathbb{R}^n$, then the following holds:
\begin{itemize}
\item[(1)] There exists a constant $C=C(n,s,p,\gamma,R,C_{*})>0$ such that
\begin{equation}\label{th1a-3}
\int_{K}\left({\bf W}_{s ,p}^{R}[\mu ](x)\right)^{\gamma}\,dx \le C\mathrm{Cap}_{\mathbf{G}_{sp},\frac{\gamma}{\gamma-p+1}}(K)
\end{equation}
for any compact set $K \subset \mathbb{R}^{n}$.
\item[(2)] For any $k \in \mathbb{N}, i$f $\{u_m\}_{m=0}^{k}$ is a sequence of nonnegative measurable functions in $\mathbb{R}^n$ that satisfies 
\begin{equation}\label{th1a-1}
\begin{aligned}
u_{m} \in L^{\gamma}_{\mathrm{loc}}(\mathbb{R}^{n}) \ \ \text{and} \ \ u_{m + 1}& \leq C_{*} {\bf W}_{s ,p}^{R}[u_{m}^{\gamma} + \mu]\ \ \ \text{for all}\ \  0 \le m \le k-1,\\
u_0 & \leq C_{*} {\bf W}_{s ,p}^{R}[\mu ],
\end{aligned}
\end{equation}
then 
\begin{equation}\label{th1a-4}
u_{k} \le \frac{\gamma \max\left\{2^{\frac{2-p}{p-1}},1\right\}}{\gamma-p+1} C_{*} {\bf W}_{s,p}^{R}[\mu ] \ \ \  \text{a.e. in } B_{2R},
\end{equation}
hence $u_{k} \in L^{\gamma}_{\mathrm{loc}}(\mathbb{R}^{n})$. 
\end{itemize}
\end{lemma}
\begin{proof} Assertion (1) follows from Proposition \ref{241020148}; note that \eqref{241020144} is satisfied by \eqref{th1a-2}. Hence, \eqref{241020144*} implies \eqref{th1a-3}. 

We now prove (2). Set $c_0=C_{*}$, then we have from \eqref{241020146} and \eqref{th1a-1} that 
\begin{align*}
u_{1} & \leq C_{*}{\bf W}^{R}_{s ,p}[u_0^\gamma+ \mu ]\\
& \leq  C_{*} \max\left\{2^{\frac{2-p}{p-1}},1\right\} \left({\bf W}^{R}_{s ,p}[u_0^\gamma]+ {\bf W}^{R}_{s ,p}[\mu ] \right)\\
& \leq  C_{*} \max\left\{2^{\frac{2-p}{p-1}},1\right\} \left(c_0^{\gamma/(p-1)}{\bf W}^{R}_{s ,p}\left[\left({\bf W}^{R}_{s ,p}[\mu]\right)^\gamma\right]+ {\bf W}^{R}_{s ,p}[\mu ] \right)\\
&\leq c_1{\bf W}^{R}_{s ,p}[\mu], \quad \text{where}\quad c_1=C_{*}\max\left\{2^{\frac{2-p}{p-1}},1\right\} \left(c_0^{\gamma/(p-1)}C_4+1\right)
\end{align*}
and that, for $m=1,2,\dots, k-1$,
\[
u_{m+1}\leq c_{m+1}{\bf W}^{R}_{s ,p}[\mu],\quad \text{where}\quad c_{m+1}=C_{*}\max\left\{2^{\frac{2-p}{p-1}},1\right\}   \left(c_m^{\gamma/(p-1)}C_4+1\right).
\]
Therefore, if $C_4>0$ satisfies 
\begin{equation}\label{pf22}
C_4\leq \left(\frac{\gamma-p+1}{\gamma C_{*}\max\left\{2^{\frac{2-p}{p-1}},1\right\} }\right)^{\gamma/(p-1)}\left(\frac{p-1}{\gamma-p+1}\right).
\end{equation}
then 
\[
c_m\leq \frac{\gamma \max\left\{2^{\frac{2-p}{p-1}},1\right\}}{\gamma-p+1}C_{*} \quad \text{for all } m \in \{0\}\cup\mathbb{N}.
\]
In view of Remark \ref{rmk26}, we can assure that \eqref{pf22} holds by taking $\delta>0$ sufficiently small in \eqref{th1a-2}, which eventually leads to \eqref{th1a-4}.
\end{proof}

\begin{lemma}\label{th1b} 
Let $\mu\in \mathcal{M}^+(\mathbb{R}^n)$ and $\gamma>p-1$, and assume that there exists a constant $\delta>0$ such that the inequality 
\begin{equation*}
\mu(K)\leq \delta\mathrm{Cap}_{\mathbf{I}_{s p},\frac{\gamma}{\gamma-p+1}}(K)
\end{equation*}
holds for any compact set $K\subset\mathbb{R}^n$. Then there exists a constant $C=C(\delta)>0$ such that
\begin{equation*}
 \int_{\mathbb{R}^{n}}\left(\mathbf{W}_{s,p}[\mu](x)\right)^{\gamma}\,dx \le C\mu(\mathbb{R}^{n}).
\end{equation*}
\end{lemma}
\begin{proof}
By (1) and (3) of Proposition \ref{241020147}, we have
\[ \int_{\mathbb{R}^{n}}\left(\mathbf{W}_{s,p}[\chi_{B_{t}(y)}\mu](x)\right)^{\gamma}\,dx \le C\mu(B_{t}(y)) \le C\mu(\mathbb{R}^{n}) \]
for any ball $B_{t}(y) \subset \mathbb{R}^{n}$, where $C=C(\delta)>0$. Then an application of the monotone convergence theorem leads to the conclusion.
\end{proof}

We next consider the case of exponential function $P=P_{l,a,\beta}$.

\begin{lemma}\label{2hvTH3-B}
Let $a,R,C_{*}>0$, $l\in \mathbb{N}$,  $\beta\geq 1$ with $l\beta >p-1$, $\mu\in \mathcal{M}^+(\mathbb{R}^n)$.  
There exists a small constant $\delta>0$, depending only on $n$, $s$, $p$, $l$, $a$, $\beta$, $C_{*}$ and $R$, such that if 
\begin{equation}\label{2hv13062}
\left\|{\bf M}_{sp,R}^{\frac{(p-1)(\beta-1)}{\beta}}[\mu]\right\|_{L^\infty(\mathbb{R}^n)} \le \delta,
\end{equation}
then the following holds:
\begin{itemize}
\item[(1)] For any $M>0$, there exists a constant $C=C(n,s,p,\beta,M)>0$ such that
\begin{equation*} 
\int_{B_{M}(0)}P_{l,a,\beta}\left( 4c_p C_{*} {\bf W}^R_{s,p}[\omega_{1}] \right)\,dx \le C, \ \ \text{where} \ \
\omega_{1} = \delta\left\|{\bf M}_{sp,R}^{\frac{(p-1)(\beta-1)}{\beta}} [1] \right\|^{-1}_{L^\infty(\mathbb{R}^n)} + \mu.
 \end{equation*}
\item[(2)] For any $k \in \mathbb{N}$, if $\{u_{m}\}_{m=0}^{k}$ is a sequence of nonnegative measurable functions in $\mathbb{R}^{n}$ that satisfies
\begin{equation*} 
\begin{aligned}
P_{l,a,\beta}(u_{m}) \in L^{1}_{\mathrm{loc}}(\mathbb{R}^{n}) \ \ \text{and} \ \ u_{m + 1}& \leq C_{*} {\bf W}^R_{s,p}[P_{l,a,\beta}(u_{m}) + \mu ] \ \ \text{for all}\ \  0 \le m \le k-1,\\
u_0 & \leq C_{*} {\bf W}^R_{s,p}[\mu ],
\end{aligned}
\end{equation*}
then
\begin{equation*}
u_{k} \le 2c_p C_{*} {\bf W}^R_{s,p}[\omega_{1}], \quad \text{hence} \quad P_{l,a,\beta}(u_{k}) \in L^{1}_{\mathrm{loc}}(\mathbb{R}^{n}). 
\end{equation*}
\end{itemize}
\end{lemma}
\begin{proof}
The proof of (2) can be found in \cite[Theorem~2.5]{HV1}, so we give the proof of (1) only. 
It suffices to show the inequality for $M>R$. Note that \eqref{2hv13062} implies
\begin{equation*}
\left\|\mathbf{M}^{\frac{(p-1)(\beta-1)}{\beta}}_{sp,R}[\omega_{1}]\right\|_{L^{\infty}(\mathbb{R}^{n})} \le 2\delta.
\end{equation*}
We then estimate
\begin{align*} 
P_{l,a,\beta}(4c_{p}C_{*}\mathbf{W}^{R}_{s,p}[\omega_{1}]) & = H_{l}\left(a(4c_{p}C_{*})^{\beta}(\mathbf{W}^{R}_{s,p}[\omega_{1}])^{\beta}\right)  \le \exp\left(a(4c_{p}C_{*})^{\beta}(\mathbf{W}^{R}_{s,p}[\omega_{1}])^{\beta}\right) \\
& \le \exp\Bigg(a(4c_{p}C_{*})^{\beta}(2\delta)^{\beta/(p-1)}\frac{(\mathbf{W}^{R}_{s,p}[\omega_{1}])^{\beta}}{\big\|\mathbf{M}^{(p-1)(\beta-1)/\beta}_{sp,R}[\omega_{1}]\big\|_{L^{\infty}(\mathbb{R}^{n})}^{\beta/(p-1)}}\Bigg).
\end{align*}
Observe that $\mathbf{W}^{R}_{s,p}[\omega_{1}](x) = \mathbf{W}^{R}_{s,p}[\chi_{B_{2M}(0)}\omega_{1}](x)$ for $x \in B_{M}(0)$. Then, in light of \cite[Theorem 2.2]{HV1}, we can choose $\delta>0$ so small that
\begin{equation*}
\fint_{B_{M}(0)}\exp\Bigg(a(4c_{p}C_{*})^{\beta}(2\delta)^{\beta/(p-1)}\frac{(\mathbf{W}^{R}_{s,p}[\chi_{B_{2M}(0)}\omega_{1}])^{\beta}}{\big\|\mathbf{M}^{(p-1)(\beta-1)/\beta}_{sp,R}[\chi_{B_{2M}(0)}\omega_{1}]\big\|_{L^{\infty}(\mathbb{R}^{n})}^{\beta/(p-1)}}\Bigg)\,dx \le c(M).
\end{equation*}
The last two displays imply the desired conclusion. 
\end{proof}

\begin{lemma}\label{2hvTH3} 
Let  $l\in \mathbb{N}$, $a,C_{*}>0$, $\beta\geq 1$ with $l\beta\geq \frac{n(p-1)}{n-sp}$. Let $\mu\in \mathcal M^+(\mathbb R^n)$ satisfy $\mathrm{supp}\,\mu \subset B_{R}(0)$ for some $R>1$ and 
\[ \left\|\mathbf{M}_{sp}^{\frac{(p-1)(\beta-1)}{\beta}}[\chi_{B_{R}}]\right\|_{L^{\infty}(\mathbb{R}^{n})}^{-1}|B_{R}| + \mu(\mathbb{R}^{n}) \le T. \]
There exists a small constant $\delta>0$, depending only on $n$, $s$, $p$, $l$, $a$, $\beta$, $C_{*}$ and $R$, such that if 
\begin{equation}\label{2hv13061}
\left\|{\bf M}_{sp}^{\frac{(p-1)(\beta-1)}{\beta}}[\mu]\right\|_{L^\infty(\mathbb{R}^n)} \le \delta,
\end{equation}
then there exists a constant $C=C(n,s,p,\beta,R,C_{*}, T)>0$ such that
\begin{equation*}
\int_{\mathbb{R}^{n}}P_{l,a,\beta}\left(4c_{p}C_{*}{\bf W}_{s,p}[\omega_{2}] \right)\,dx \le C, \ \ \text{where} \ \
\omega_{2}=\delta\left\|{\bf M}_{sp}^{\frac{(p-1)(\beta-1)}{\beta}} [\chi_{B_{R}}] \right\|^{-1}_{{L^\infty }(\mathbb{R}^{n})}\chi_{B_{R}} + \mu.
\end{equation*}
\end{lemma}
\begin{proof}
We first split 
\[ \int_{\mathbb{R}^{n}}P_{l,a,\beta}(4c_{p}C_{*}\mathbf{W}_{s,p}[\omega_{2}])\,dx  = \int_{B_{2R}(0)}(\cdots)\,dx + \int_{\mathbb{R}^{n}\setminus B_{2R}(0)}(\cdots)\,dx \eqqcolon I_{1} + I_{2}. \]
Similarly to the proof of Lemma \ref{2hvTH3-B}, we can obtain $I_{1} \le C(n,s,p,\beta)$ by choosing $\delta$ sufficiently small. As for $I_{2}$, observe that
\begin{equation*}
\mathbf{W}_{s,p}[\omega_{2}](y) = \int_{|y|/2}^{\infty}\left[\frac{\omega_{2}(B_{t}(y))}{t^{n-sp}}\right]^{1/(p-1)}\frac{dt}{t} \le cT^{\frac{1}{p-1}}\int_{|y|/2}^{\infty}t^{-\frac{n-sp}{p-1}}\frac{dt}{t} \le C|y|^{-\frac{n-sp}{p-1}}
\end{equation*}
for any $y \in \mathbb{R}^{n}\setminus B_{2R}(0)$. We thus have
\begin{equation*}
I_{2} \le \int_{\mathbb{R}^{n}\setminus B_{2R}(0)}H_{l}\left(C|y|^{-\frac{\beta(n-sp)}{p-1}}\right)\,dy \le C\int_{\mathbb{R}^{n}\setminus B_{2R}(0)}|y|^{-l\beta\frac{n-sp}{p-1}}\,dy \le CR^{n-l\beta\frac{n-sp}{p-1}},
\end{equation*}
which leads to the conclusion.
\end{proof}

Finally, we obtain relations between  pointwise estimates and capacities conditions.

\begin{lemma}\label{2hvTH1a} Let  $\gamma>p-1$  and $\mu\in \mathcal{M}^+(\mathbb{R}^n)$.
\begin{itemize}	
\item[(1)]  Let $0<R\le \infty$. Let $u$ be a nonnegative measurable function in $\mathbb{R}^n$ such that $u^\gamma$ is locally integrable in $\mathbb{R}^n$ and
\begin{equation}\label{242}
u(x)\geq c {\bf W}^R_{s,p}[u^\gamma+\mu](x) \quad \text{for a.e. } x\in \mathbb{R}^n
\end{equation}
for some $c>0$. Then the following holds: 
\begin{itemize}
\item[(i)] If $R<\infty$, then there exists a constant $C_1>0$, depending only on $n$, $s$, $p$, $\gamma$, $R$ and $c$, such that
	\begin{equation}
	\label{2hv06064b}\int_E u^\gamma\, dx+\mu(E)  \le C_1\mathrm{Cap}_{{\mathbf{G}_{s p}},\frac{\gamma}{\gamma-p+1}}(E) \quad \text{for any Borel set }E\subset \mathbb{R}^n.
	\end{equation} 
\item[(ii)] If $R=\infty$, then there exists a constant $C_2>0$, depending only on $n$, $s$, $p$, $\gamma$ and $c$, such that
	\begin{equation}
	\label{2hv06065b}\int_E u^\gamma\,dx+\mu(E)  \le C_2\mathrm{Cap}_{{\mathbf{I}_{sp}},\frac{\gamma}{\gamma-p+1}}(E) \quad \text{for any Borel set } E\subset \mathbb{R}^n.
	\end{equation} 
\end{itemize}
	
\item[(2)] Let $\Omega \subset \mathbb{R}^{n}$ be a bounded domain, $\mu\in \mathcal{M}^+(\Omega)$ and $\varepsilon_{0}\in (0,1)$. Let $u$ be a nonnegative measurable function in $\Omega$ such that  $u^\gamma$ is locally integrable in  $\Omega$ and 
\begin{equation}\label{2hv25032b}
u(x)\geq c{\bf W}^{\varepsilon_{0} \mathrm{dist}(x,\partial\Omega)}_{s,p}[u^\gamma+\mu](x) \quad \text{for a.e. } x\in \Omega
\end{equation}
for some $c>0$. Then for any compact set $K\subset \Omega$, there exists a constant $C_3>0$, depending only on $n$, $s$, $p$, $\gamma$, $\varepsilon_{0}$, $\mathrm{dist}(K,\partial\Omega)$ and $c$, such that 
\begin{equation}\label{2hv08063b}
\int_E {u^\gamma\,dx}  + \mu (E) \leq C_3 \mathrm{Cap}_{{\mathbf{G}_{s p}},\frac{\gamma}{\gamma-p+1}}(E) \quad \text{for any Borel set } E\subset K.
\end{equation}	
\end{itemize}
\end{lemma}
\begin{proof} (1) Let us set 
\[ d\omega \coloneqq  u^\gamma\,dx+d\mu. \]
Then \eqref{242} directly implies
\[ \left(c {\bf W}^R_{s,p}[\omega]\right)^\gamma\,dx = \left(c {\bf W}^R_{s,p}[u^\gamma+\mu]\right)^\gamma\,dx\leq u^{\gamma}\, dx \leq d\omega. \]
In addition, considering the centered maximal function with the weight $\omega$ denoted by
\[ {\bf M}_\omega f(x) \coloneqq \sup_{t>0} \frac{1}{\omega(B_t(x))}\int_{B_t(x)}|f|\, d\omega, \] 
we have
\[ \int_{\mathbb{R}^n}\left({\bf M}_\omega\chi_E\right)^{\gamma/(p-1)}\left(c {\bf W}^R_{s,p}[\omega]\right)^\gamma\,dx  \leq  \int_{\mathbb{R}^n}\left({\bf M}_\omega\chi_E\right)^{\gamma/(p-1)} d\omega \leq  c_1 \omega(E) \]
for any Borel set $E\subset \mathbb{R}^n$, 
where $c_1 = c_1(n,p,\gamma)>0$. Note that, in the second inequality, we have used the fact that the maximal operator ${\bf M}_\omega$ is bounded on $L^{t}(\mathbb{R}^n,d\omega)$ for any $t \in (1,\infty)$. Moreover, since 
\[ \left({\bf M}_\omega\chi_E\right)^{\gamma/(p-1)}\left({\bf W}^R_{s,p}[\omega]\right)^\gamma\geq \left({\bf W}^R_{s,p}[\chi_E\omega]\right)^\gamma, \]
we obtain
\[ \int_{\mathbb{R}^n}\left(c {\bf W}^R_{s,p}[\chi_E\omega]\right)^\gamma\, dx \leq c_1\omega(E) \quad \text{for any Borel set } E\subset \mathbb{R}^n. \]
From this last inequality and \cite[Theorem 2.1]{HV1}, we see that
\begin{equation}\label{247}
\left\{
\begin{aligned}
c^\gamma \int_{\mathbb{R}^n} \left({\bf G}_{sp}[\chi_E\omega]\right)^{\gamma/(p-1)}\, dx & \leq c_2 \omega(E) \quad \text{if}\;\; 0<R<\infty, \\
c^\gamma \int_{\mathbb{R}^n} \left({\bf I}_{sp}[\chi_E\omega]\right)^{\gamma/(p-1)}\, dx & \leq c_3 \omega(E) \quad \text{if}\;\; R=\infty
\end{aligned}
\right.
\end{equation}
holds for some positive constants $c$, $c_{2}$ and $c_{3}$, all depending only on $n$, $s$, $p$ and $\gamma$. 
We now consider any $f\in L^{\gamma/(\gamma-p+1)}(\mathbb{R}^{n})$ with $f\geq0$ and ${\bf G}_{sp}*f\geq \chi_E$. In the case $0<R<\infty$, by Fubini's theorem and Young's inequality, we get  
\begin{align*}
\omega(E) &\leq \int_{\mathbb{R}^n}({\bf G}_{sp} *f)\chi_E \, d\omega= \int_{\mathbb{R}^n}({\bf G}_{sp} [\chi_E\omega])f\, dx\\
&\leq  \frac{c^\gamma}{2c_2} \int_{\mathbb{R}^n}({\bf G}_{sp} [\chi_E\omega])^{\gamma/(p-1)}\, dx+c_4 \int_{\mathbb{R}^n} f^{\gamma/(\gamma-p+1)}\, dx.
\end{align*}
Consequently, this together with \eqref{247} and the definition of Orlicz--Bessel capacity gives \eqref{2hv06064b}. In the case $R=\infty$, we can also show \eqref{2hv06065b} by a similar argument.

\medskip

\noindent (2) Let us set $r_K \coloneqq \mathrm{dist}(K,\partial\Omega)$ and $\Omega_K \coloneqq \{x\in\Omega : \mathrm{dist}(x,K)<r_K/2\}$. We follow the same argument as in (1). Since 
\[ \left(c {\bf W}^{\varepsilon_{0} \mathrm{dist}(x,\partial\Omega)}_{s,p}[\omega]\right)^\gamma\,dx \leq  d\omega \quad \text{in } \Omega \]
by  \eqref{2hv25032b}, we have  
\[ \int_{\Omega}\left(c {\bf W}^{\varepsilon_{0} \mathrm{dist}(x,\partial\Omega)}_{s,p}[\chi_E\omega]\right)^\gamma\, dx   \leq  \int_{\Omega}\left({\bf M}_\omega\chi_E\right)^{\gamma/(p-1)} d\omega \leq  c_4 \omega(E) \]  
for any Borel set $E \subset K$. 
Note that if $x\in\Omega$ satisfies $\mathrm{dist}(x,\partial\Omega)\leq r_K/8$, then $B_t(x)\subset \Omega\setminus \Omega_K$ for all $t\in (0,r_K/8)$, and so ${\bf W}^{\varepsilon_{0} r_{K}/8}_{s,p}[\chi_E\omega](x)=0$. From this observation, we have 
\[ {\bf W}^{\varepsilon_{0} r_K/8}_{s,p}[\chi_E\omega](x) \leq {\bf W}^{\varepsilon_{0} \mathrm{dist}(x,\partial\Omega)}_{s,p}[\chi_E\omega](x) \quad \text{for a.e. }  x\in\Omega, \]
hence
\[ \int_{\mathbb{R}^n} \left(c{\bf W}^{\varepsilon_{0} r_K/8}_{s,p}[\chi_E\omega]\right)^\gamma\, dx \leq c_4\omega(E) \quad \text{for any Borel set } E\subset K.  \]
Therefore, in the same way as in (1), we obtain \eqref{2hv08063b}.
\end{proof}

The next lemma can be found in \cite[Theorem~2.7]{HV1}. Recall that $\mathrm{Cap}_{\mathbf{G}_{sp},Q^{\ast}_{p}}$ and $\mathrm{Cap}_{\mathbf{I}_{sp},Q^{\ast}_{p}}$ have been defined in \eqref{2hvEQ14} and \eqref{2hvEQ15}, respectively.

\begin{lemma}\label{2hvTH1} 
Let  $a>0$, $l\in \mathbb{N}$, $\beta\geq 1$ with $l\beta>p-1$, and $\mu\in \mathcal{M}^+(\mathbb{R}^n)$.
\begin{itemize}
\item[(1)] Let $0<R\le \infty$. Let $u$ be a nonnegative measurable function in $\mathbb{R}^n$ such that $P_{l,a,\beta}(u)$ is locally integrable in $\mathbb{R}^n$ and
\begin{equation*}
u(x)\geq c {\bf W}^R_{s,p}[P_{l,a,\beta}(u)+\mu](x) \quad \text{for a.e. }  x\in \mathbb{R}^n
\end{equation*}
for some $c>0$. Then the following holds:
\begin{itemize}
\item[(i)]
If $R<\infty$, then there exists a constant $C_1>0$, depending only on $n$, $s$, $p$, $l$, $a$, $\beta$, $R$ and $c$, such that
\begin{equation*}
\int_E P_{l,a,\beta}(u)\,dx+\mu(E)  \le C_1\mathrm{Cap}_{{\mathbf{G}_{sp}},Q_p^*}(E) \quad \text{for any Borel set }E\subset \mathbb{R}^n.
\end{equation*} 
\item[(ii)]
If $R=\infty$, then there exists a constant $C_2>0$, depending only on $n$, $s$, $p$, $l$, $a$, $\beta$ and $c$, such that
\begin{equation*}
\int_E P_{l,a,\beta}(u)\,dx+\mu(E)  \le C_2\mathrm{Cap}_{{\mathbf{I}_{sp}},Q_p^*}(E) \quad \text{for any Borel set } E\subset \mathbb{R}^n.
\end{equation*} 
\end{itemize}
\item[(2)] Let $\Omega \subset \mathbb{R}^{n}$ be a bounded domain, $\mu\in \mathcal{M}^+(\Omega)$ and $\varepsilon_{0}\in (0,1)$. Let $u$ be a nonnegative measurable function in $\Omega$ such that  $P_{l,a,\beta}(u)$ is locally integrable in  $\Omega$ and 
\begin{equation*}
u(x)\geq c{\bf W}^{\varepsilon_{0} \mathrm{dist}(x,\partial\Omega)}_{s,p}[P_{l,a,\beta}(u)+\mu](x) \quad \text{for a.e. }  x\in \Omega,
\end{equation*}
for some $c>0$. Then for any compact set $K\subset \Omega$, there exists a constant $C_3>0$, depending only on $n$, $s$, $p$, $l$, $a$, $\beta$, $ \varepsilon_{0}$, $\mathrm{dist}(K,\partial\Omega)$ and $c$, such that 
\begin{equation*}
\int_E P_{l,a,\beta}(u)\,dx  + \mu (E) \leq C_3 \mathrm{Cap}_{{\mathbf{G}_{sp}},Q_p^*}(E) \quad \text{for any Borel set } E\subset K.
\end{equation*}
\end{itemize}
\end{lemma}

\subsection{Proof of the existence results}

Now, we shall prove our main results.

\begin{proof}[Proof of Theorem \ref{2hvMT3}] 
(1) Assume that \eqref{2hvMT3-2} admits a nonnegative SOLA $u$. Then, by Corollary \ref{2hvTH4}, there holds 
\[ u(x)\geq \frac{1}{C_{0}}{\bf W}^{\mathrm{dist}(x,\partial\Omega)/8}_{s,p}[u^{\gamma}+\mu](x) \quad \text{for a.e. } x\in  \Omega. \]
Hence,  we achieve \eqref{2hvMT3-4} from Lemma \ref{2hvTH1a} (2).

\medskip
\noindent (2) Assume that \eqref{2hvMT3-1} holds with $\delta=\tilde{\delta}>0$ which will be determined in a few lines.

\textit{Step 1.} We first consider the case $0 \le \mu \in L^{\infty}(\Omega)$. Let $u_0$ be the weak solution to 
\begin{equation*}
\left\{
\begin{aligned}
- \mathcal{L}_\Phi u_0& =  \mu&\text{in }&\Omega,  \\ 
u_0& =  0 & \text{in }& \mathbb{R}^n\setminus\Omega.
\end{aligned}
\right.
\end{equation*}
Then Proposition \ref{pro1} and Corollary \ref{2hvTH4} imply
\[ 0 \le u_{0} \leq  C_{0} {\bf W}^{R}_{s,p}[\mu] \quad \text{a.e. in }\Omega, \quad \text{where} \quad R=2\mathrm{diam}(\Omega). \] 
In particular, since $\mu \in L^{\infty}(\Omega)$, this inequality implies $u_{0} \in L^{\infty}(\mathbb{R}^{n})$. 
We inductively find a sequence $\{u_{m}\}_{m=1}^{\infty} \subset \mathbb{W}^{s,p}(\Omega) \cap L^{\infty}(\mathbb{R}^{n})$ of weak solutions to
\begin{equation*}
\left\{
\begin{aligned}
- \mathcal{L}_\Phi u_{m}&=\mu_m &\text{in }&\Omega,\\
u_{m}&=0&\text{in }&\mathbb R^n\setminus\Omega,
\end{aligned}
\right. 
\quad\text{where}\quad \mu_m \coloneqq u_{m-1}^\gamma+ \mu.
\end{equation*}
Note that $\{u_{m}\}$ is nondecreasing by Proposition \ref{pro1}. Also, again by Corollary \ref{2hvTH4}, we have for any $m \in \mathbb{N}$
\[ u_{m} \leq C_{0} {\bf W}^{R}_{s,p}[u_{m-1}^\gamma+\mu] \quad \text{a.e. in } \Omega. \]
We now fix $\tilde{\delta}>0$, depending only on $n$, $s$, $p$, $\gamma$, $C_{0}$ and $R$, in a way such that \eqref{th1a-2} holds with $\delta=\tilde{\delta}$ in the case $C_{*}=C_{0}$. Then Lemma \ref{th1a} implies that, for any $m \in \mathbb{N}$,
\begin{equation}\label{2hv1335}
u_{m} \le \frac{\gamma \max\left\{2^{\frac{2-p}{p-1}},1\right\}}{\gamma-p+1}C_{0}{\bf W}_{s,p}^{R}[\mu] \quad \text{a.e. in } \Omega. 
\end{equation}
Thus, by the monotone convergence theorem, there exists a measurable function $u$ such that $u_m\to u$ a.e. in $\mathbb R^n$; in particular, $u$ satisfies \eqref{2hvMT3-3} and so $0 \le u \in L^{\infty}(\mathbb{R}^{n})$.  
From \eqref{2hv1335} and Lebesgue's dominated convergence theorem, it follows that $u_m^\gamma\to u^\gamma$ in $L^{1}(\Omega)$. 
Accordingly, using \eqref{global} for any $h$ and $q$ satisfying \eqref{SOLA1}, we have that
\begin{align*}
\left(\int_{\mathbb{R}^{n}}\int_{\Omega}\frac{|u_m(x)-u_m(y)|^{q}}{|x-y|^{n+hq}}\,dx\,dy\right)^{1/q} & \leq c\left(\int_\Omega u_{m-1}^\gamma\,dx + \mu(\Omega)\right)^{1/(p-1)} \\
& \le c\left(\int_{\Omega}u^{\gamma}\,dx + \mu(\Omega)\right)^{1/(p-1)}
\end{align*}
holds whenever $m \in \mathbb{N}$, where $c>0$ is independent of $m$. 
In turn, Fatou's lemma and Lemma \ref{lemcompact} imply that $u \in W^{h,q}(\Omega)$ and $u_{m} \to u$ locally in $L^{q}(\mathbb{R}^{n})$ (up to subsequences). 
Then a similar argument as in the proof of Theorem \ref{thm.ex.bdd} shows that $u$ satisfies \eqref{SOLA2}. 
Moreover, by using the standard energy estimate
\begin{align*} 
\int_{\mathbb{R}^{n}}\int_{\Omega}\frac{|u_{m}(x)-u_{m}(y)|^{p}}{|x-y|^{n+sp}}\,dx\,dy & \le c\int_{\Omega}(u_{m-1}^{\gamma} + \mu)u_{m}\,dx \\
\overset{\eqref{2hv1335}} & {\le} c\int_{\Omega}\left[(\mathbf{W}^{R}_{s,p}[\mu])^{\gamma+1} +  (\mathbf{W}^{R}_{s,p}[\mu])\mu\right]\,dx
\end{align*}
and Fatou's lemma, we deduce that $u \in \mathbb{W}^{s,p}(\Omega) \cap L^{\infty}(\mathbb{R}^{n})$ is a nonnegative weak solution to \eqref{2hvMT3-2}.

\textit{Step 2.} We now consider any $\mu\in \mathcal{M}^{+}(\Omega)$. Assume that \eqref{2hvMT3-1} holds with $\delta=\tilde{\delta}/c_0$, where $\tilde{\delta}$ and $c_{0}$ are the constants determined in \textit{Step 1} and Lemma \ref{apro-le}, respectively. Then, with $\{\rho_{j}\}_{j=1}^{\infty}$ being a sequence of standard mollifiers, $\mu_{j}=\mu*\rho_{j}\in L^\infty(\Omega)$ satisfies  
\begin{equation}\label{muj_Bessel}
\mu_{j}(K)\leq \tilde{\delta}\mathrm{Cap}_{\mathbf{G}_{sp},\frac{\gamma}{\gamma-p+1}}(K)
\end{equation} 
for any compact set $K\subset\mathbb{R}^n$.
Thus, applying the conclusion of \textit{Step 1} to $\mu=\mu_{j}$, we find a sequence of nonnegative weak solutions $\{u_{j}\} \subset \mathbb{W}^{s,p}(\Omega)\cap L^\infty(\mathbb{R}^n)$ to 
\begin{equation*}
\left\{
\begin{aligned}
- \mathcal{L}_\Phi u_{j}&=u_{j}^\gamma+\mu_{j} &\text{in }&\Omega,\\
u_{j}&=0&\text{in }&\mathbb{R}^{n}\setminus\Omega
\end{aligned}
\right. 
\end{equation*}
satisfying, for any $j \in \mathbb{N}$, 
\begin{equation}\label{PF30}
u_{j}\leq \frac{\gamma \max\left\{2^{\frac{2-p}{p-1}},1\right\}}{\gamma-p+1}C_{0}\mathbf{W}_{s,p}^{R}[\mu_{j}] \quad \text{a.e. in }\Omega. 
\end{equation}
Since \eqref{muj_Bessel} holds for any $j \in \mathbb{N}$, it follows from Lemma \ref{th1a} (1) and \eqref{PF30} that $\{(\mathbf{W}_{s,p}^{R}[\mu_{j}])^{\gamma}\}$ is bounded in $L^{1}(\Omega)$ and so is $\{u_{j}^{\gamma}\}$. 
Using this fact and \eqref{global}, we get
\[ \left(\int_{\mathbb{R}^{n}}\int_{\Omega}\frac{|u_{j}(x)-u_{j}(y)|^{q}}{|x-y|^{n+hq}}\,dx\,dy\right)^{1/q} \leq c\left(\int_{\Omega}u_{j}^{\gamma}\,dx + \mu_{j}(\Omega)\right)^{1/(p-1)} \le c \]
for any $j\in \mathbb{N}$, where $c>0$ is independent of $j$. 
Therefore, by Lemma \ref{lemcompact}, there exists $u \in W^{h,q}(\Omega)$ such that $u_{j}\to u$  a.e. in $\Omega$ and locally in $L^{q}(\mathbb{R}^{n})$ (up to subsequences). 
Moreover, since $\mathbf{W}^{R}_{s,p}[\mu] \in L^{\gamma}(\Omega)$ by Proposition \ref{241020148}, it follows from Lemma \ref{equi-integrable} and \eqref{PF30} that $\{(\mathbf{W}_{s,p}^R[\mu_{j}])^\gamma\}$ is equi-integrable in $\Omega$ and so is $\{u_{j}^{\gamma}\}$. Hence, Vitali's convergence theorem implies $u_{j}^{\gamma} \to u^{\gamma}$ in $L^{1}(\Omega)$ (up to subsequences). 
Then a similar argument as in the proof of Theorem \ref{thm.ex.bdd} shows that $u$ satisfies \eqref{SOLA2}, hence $u$ is a nonnegative SOLA to \eqref{2hvMT3-2}. 
Also, \eqref{2hvMT3-3} follows by letting $j\rightarrow\infty$ in \eqref{PF30}.
\end{proof}

\begin{proof}[Proof of Theorem \ref{2hvMT4}]
(1) Assume that \eqref{2hvMT4-2} admits a nonnegative SOLA $u$. Then, by Corollary \ref{2hvTH4}, there holds 
\[
u(x)\geq \frac{1}{C_{0}}{\bf W}_{s,p}[u^\gamma+\mu](x) \quad \text{for a.e. } x\in  \mathbb{R}^n.
\]
Hence,  we achieve \eqref{2hvMT4-4} from Lemma \ref{2hvTH1a} (1). 

\medskip

\noindent (2) Assume that \eqref{2hvMT4-1} holds for some $\delta>0$, and consider $\mu_{j} = \mu \ast \rho_{j}$ for a sequence $\{\rho_{j}\}_{j=1}^{\infty}$ of standard mollifiers. 
In view of Propositions \ref{241020147} and \ref{241020148}, 
in particular \eqref{241020141}, \eqref{241020142}, \eqref{241020145} and \eqref{241020144},
we can choose $\delta>0$ so small that \eqref{2hvMT3-1} with $\mu$ replaced by $\mu_{j}$ holds for any $j \in \mathbb{N}$ (recall Remark \ref{rmk26} and Lemma \ref{apro-le}). 
Thus, applying the conclusion of \textit{Step 1} in the proof of Theorem \ref{2hvMT3} (2), 
along with \eqref{241020141} and \eqref{241020141*} in Proposition \ref{241020147}, we find a sequence $\{u_{j}\}_{j=1}^{\infty}$ of nonnegative weak solutions to
\begin{equation*}
\left\{
\begin{aligned}
- \mathcal{L}_\Phi u_{j} & =  u_{j}^{\gamma} + \mu_{j} &\text{in }&B_{j}(0),  \\ 
u_{j} & =  0  &\text{in }& \mathbb{R}^{n}\setminus B_{j}(0)
\end{aligned}
\right.
\end{equation*}
satisfying, for any $j \in \mathbb{N}$,
\begin{equation}\label{PF40}
u_{j} \leq \frac{\gamma \max\left\{2^{\frac{2-p}{p-1}},1\right\}}{\gamma-p+1}C_{0}{\bf W}^{4j}_{s,p}[\mu_{j}] \le \frac{\gamma \max\left\{2^{\frac{2-p}{p-1}},1\right\}}{\gamma-p+1}C_{0}\mathbf{W}_{s,p}[\mu_{j}] \;\; \text{a.e. in } \mathbb R^{n}. 
\end{equation}
Let us fix any $r \ge 1$. Since \eqref{2hvMT3-1} with $\mu$ replaced by $\mu_{j}$ holds for any $j \in \mathbb{N}$, it follows from Lemma \ref{th1b} and \eqref{PF40} that $\{(\mathbf{W}_{s,p}[\mu_{j}])^{\gamma}\}$ is bounded in $L^{1}(\mathbb{R}^{n})$ and so is $\{u_{j}^{\gamma}\}$. 
Using this fact along with \eqref{z2} and \eqref{uj.tail}, we have
\begin{align*}
&\left(\int_{B_{r}(0)}\fint_{B_{r}(0)}\frac{|u_{j}(x)-u_{j}(y)|^{q}}{|x-y|^{n+hq}}\,dx\,dy\right)^{1/q} \\
& \le \frac{c}{r^{h+n/\bar{q}}}\left(\int_{B_{j}(0)}u_{j}^{\gamma}\,dx + \mu_{j}(B_{j}(0))\right)^{1/(p-1)} \le c
\end{align*}
for any $h$ and $q$ satisfying \eqref{SOLA1} and
\begin{equation}\label{tail.uniform}
\int_{\R^n \setminus B_r(0)}\frac{|u_j(x)|^{p-1}}{|x|^{n+sp}}\,dx \le cr^{-(sp)^2/n}\left(\int_{B_j(0)}u_j^{\gamma}\,dx + \mu_j(B_j(0)) \right) \le c
\end{equation}
whenever $j \ge r$, where $c>0$ is independent of $j$. 
Namely, for any $r \ge 1$, $\{u_{j}\}_{j \ge r}$ is bounded in $W^{h,q}(B_{r}(0))$ and $\{\tail(u_j;0,r)\}_{j\ge r}$ is bounded. Therefore, in light of Lemma \ref{lemcompact}, there exists $u \in W^{h,q}_{\mathrm{loc}}(\mathbb{R}^{n})$ such that $u_{j} \rightarrow u$ a.e. in $\mathbb{R}^n$ and locally in $L^{q}(\mathbb{R}^n)$ (up to subsequences); in particular, \eqref{tail.uniform} and Fatou's lemma imply \eqref{SOLA.tail}.  
Moreover, since $\mathbf{W}_{s,p}[\mu] \in L^{\gamma}_{\mathrm{loc}}(\mathbb{R}^{n})$ by Proposition \ref{241020147}, it follows from Lemma \ref{equi-integrable} and \eqref{PF40} that $\{(\mathbf{W}_{s,p}[\mu_{j}])^{\gamma}\}$ is equi-integrable in $B_{M}(0)$ for any $M>1$, and so is $\{u_{j}^{\gamma}\}$. Hence, Vitali's convergence theorem implies $u_{j}^{\gamma} \rightarrow u^{\gamma}$ locally in $L^{1}(\mathbb{R}^{n})$ (up to subsequences). 
Similarly to the proof of Theorem \ref{thm.ex.rn}, we see that $u$ satisfies \eqref{SOLA2*}, hence it is a nonnegative SOLA to  \eqref{EQpRn}. Also, letting $j \to \infty$ in \eqref{PF40} yields \eqref{2hvMT4-3}, which together with Corollary \ref{2hvTH4-} (applied to $u-\inf_{\R^n}u$) leads to
\begin{align*}
u(x) \le c{\bf W}_{s,p}[\mu](x) \le c{\bf W}_{s,p}[ u^{\gamma} + \mu ](x) \le  c\left(u(x) - \inf_{\R^n}u\right) \quad \text{for a.e. } x \in \R^n.
\end{align*}
Thus we have $\inf_{\mathbb{R}^n}u = 0$, thereby concluding that $u$ is a nonnegative SOLA to \eqref{2hvMT4-2} satisfying \eqref{2hvMT4-3}.
\end{proof}

\begin{proof}[Proof of Theorem \ref{2hvMT1}]
(1) Assume that \eqref{2hvMT1a} admits a nonnegative SOLA $u$. Then, by Corollary \ref{2hvTH4}, there holds 
\[
u(x)\geq \frac{1}{C_{0}}{\bf W}^{\mathrm{dist}(x,\partial\Omega)/8}_{s,p}[P_{l,a,\beta}(u)+\mu](x) \quad \text{for a.e. } x\in  \Omega.
\]
Hence,  we achieve \eqref{2hvMT1c} from Lemma \ref{2hvTH1} (2).

\medskip

\noindent (2) Assume that \eqref{2hvMT3-1-} holds with $\delta>0$ which will be determined in a few lines.

\textit{Step 1.} We first consider the case $0 \le \mu \in L^{\infty}(\Omega)$.
Let $u_0$ be the weak solution to 
\begin{equation*}
\left\{
\begin{aligned}
- \mathcal{L}_\Phi u_0& = \mu&\text{in }&\Omega,  \\ 
u_0& = 0 & \text{in }& \mathbb{R}^n\setminus\Omega.
\end{aligned}
\right.
\end{equation*}
Then by Proposition \ref{pro1} and Corollary \ref{2hvTH4}, we have that 
\[
0 \le u_0 \leq  C_{0} {\bf W}^{R}_{s,p}[\mu] \quad \text{a.e. in }\Omega, \quad \text{where} \quad R = 2\mathrm{diam}(\Omega). 
\]
In particular, since $\mu \in L^{\infty}(\Omega)$, this inequality implies $u_{0} \in L^{\infty}(\mathbb{R}^{n})$. 
We inductively find a sequence $\{u_{m}\}_{m=1}^{\infty} \in \mathbb{W}^{s,p}(\Omega) \cap L^{\infty}(\mathbb{R}^{n})$ of weak solutions to
\begin{equation*}
\left\{
\begin{aligned}
- \mathcal{L}_\Phi u_{m}&=\mu_m &\text{in }&\Omega,\\
u_{m}&=0&\text{in }&\mathbb R^n\setminus\Omega,
\end{aligned}
\right.
\quad\text{where}\quad \mu_m \coloneqq P_{l,a,\beta}(u_{m-1})+ \mu.
\end{equation*}
Note that $\{u_{m}\}$ is nondecreasing by Proposition \ref{pro1}. Also, again by Corollary \ref{2hvTH4}, we have for any $m\in\mathbb{N}$
\[ u_{m} \le C_{0}\mathbf{W}^{R}_{s,p}[P_{l,a,\beta}(u_{m-1})+\mu] \quad \text{a.e. in }\Omega. \]
We now fix the constant $\delta>0$, depending only on $n$, $s$, $p$, $l$, $a$, $\beta$, $C_{0}$ and $R$, in a way that \eqref{2hv13062} holds with $C_{*}=C_{0}$. 
Then, an application of Lemma \ref{2hvTH3-B} gives
\begin{equation}\label{2hv1335'}
{u_m} \le 2c_pC_{0} {\bf W}_{s ,p}^R[\omega_{1}] \quad \text{a.e. in } \Omega 
\end{equation}
whenever $m \in \mathbb{N}$, where
\[ \omega_{1} = \delta\left\|{\bf M}_{sp,R}^{\frac{(p-1)(\beta-1)}{\beta}} [1] \right\|_{{L^\infty }(\mathbb{R}^n)}^{-1}+\mu. \]  
Thus, by the monotone convergence theorem, there exists a measurable function $u$ such that $u_{m} \rightarrow u$ a.e. in $\mathbb{R}^{n}$; in particular, $u$ satisfies \eqref{2hvMT1b} and so $0 \le u \in L^{\infty}(\mathbb{R}^{n})$. 
From \eqref{2hv1335'} and Lebesgue's dominated convergence theorem, it follows that $P_{l,a,\beta}(u_{m})\to P_{l,a,\beta}(u)$ in $L^{1}(\Omega)$. 
Moreover, applying \eqref{global} for any $h$ and $q$ satisfying \eqref{SOLA1}, we have that
\begin{align*}
\left(\int_{\mathbb{R}^{n}}\int_{\Omega}\frac{|u_{m}(x)-u_{m}(y)|^{q}}{|x-y|^{n+hq}}\,dx\,dy\right)^{1/q} & \le c\left(\int_{\Omega}P_{l,a,\beta}(u_{m-1})\,dx + \mu(\Omega)\right)^{1/(p-1)} \\
& \le c\left(\int_{\Omega}P_{l,a,\beta}(u)\,dx + \mu(\Omega)\right)^{1/(p-1)}
\end{align*}
holds whenever $m \in \mathbb{N}$, where $c>0$ is independent of $m$. In turn, Fatou's lemma and Lemma \ref{lemcompact} imply that $u \in W^{h,q}(\Omega)$ and $u_{m}\rightarrow u$ locally in $L^{q}(\mathbb{R}^{n})$ (up to subsequences). 
Then, similarly to \textit{Step 1} in the proof of Theorem \ref{2hvMT3} (2), we see that $u \in \mathbb{W}^{s,p}(\Omega) \cap L^{\infty}(\mathbb{R}^{n})$ and it is a weak solution to \eqref{2hvMT1a}.

\textit{Step 2.} We now consider any $\mu\in \mathcal{M}^{+}(\Omega)$ and assume that \eqref{2hvMT3-1-} holds with the constant $\delta>0$ determined in \textit{Step 1}. Then, with $\{\rho_{j}\}_{j=1}^{\infty}$ being a sequence of standard mollifiers, $\mu_{j} = \mu \ast \rho_{j} \in L^{\infty}(\Omega)$ satisfies
\begin{equation*}
\left\|{\bf M}_{sp,R}^{\frac{{(p - 1)(\beta  - 1)}}{\beta }}[\mu_{j} ]\right\|_{L^\infty({\mathbb{R}^n})} \le  \left\|{\bf M}_{sp,R}^{\frac{{(p - 1)(\beta  - 1)}}{\beta }}[\mu ]\right\|_{L^\infty({\mathbb{R}^n})} \le \delta
\end{equation*}
for any $j \in \mathbb{N}$. Thus, applying the conclusion of \textit{Step 1} to $\mu=\mu_{j}$, we find a sequence of nonnegative weak solutions $\{u_{j}\} \subset \mathbb{W}^{s,p}(\Omega) \cap L^{\infty}(\mathbb{R}^{n})$ to 
\begin{equation*}
\left\{
\begin{aligned}
- \mathcal{L}_\Phi u_{j}&= P_{l,a,\beta}(u_{j})+\mu_{j} &\text{in }&\Omega,\\
u_{j}&=0&\text{in }&\mathbb R^n\setminus\Omega
\end{aligned}
\right. 
\end{equation*}
satisfying for any $j \in \mathbb{N}$
\begin{equation}\label{PF50}
u_{j}\leq 2c_{p}C_{0}\mathbf{W}_{s,p}^R[\omega_{1,j}] \quad \text{a.e. in } \Omega, 
\end{equation}
where
\[  \omega_{1,j} = \delta\left\|{\bf M}_{ sp,R}^{\frac{(p-1)(\beta-1)}{\beta}} [1] \right\|_{{L^\infty }(\mathbb{R}^n)}^{-1}+\mu_{j}. \]
Moreover, since $\{P_{l,a,\beta}(4c_{p}C_{0}\mathbf{W}_{s,p}^R[\omega_{1,j}])\}$ is bounded in $L^{1}(\Omega)$ by Lemma \ref{2hvTH3-B} (1), it follows from Lemma \ref{equi-integrable2} and \eqref{PF50} that $\{P_{l,a,\beta}(2c_{p}C_{0}\mathbf{W}^{R}_{s,p}[\omega_{1,j}])\}$ is equi-integrable in $\Omega$ and so is $\{P_{l,a,\beta}(u_{j})\}$. Similarly to \textit{Step 2} in the proof of Theorem \ref{2hvMT3} (2), we conclude that $\{u_{j}\}$ converges (up to subsequences) to a nonnegative SOLA $u$ to \eqref{2hvMT1a} which satisfies \eqref{2hvMT1b}. 
\end{proof}

\begin{proof}[Proof of Theorem \ref{2hvMT2}]
(1) Assume that \eqref{2hvMT2a} admits a nonnegative SOLA $u$. Then, by Corollary \ref{2hvTH4}, there holds 
\[
u(x)\geq \frac{1}{C_{0}}{\bf W}_{s,p}[P_{l,a,\beta}(u)+\mu](x) \quad \text{for a.e. } x\in  \mathbb{R}^n.
\]
Hence,  we achieve \eqref{2hvMT2c} from Lemma \ref{2hvTH1} (1). 

\medskip

\noindent (2) 
Assume that \eqref{2hvMT3-1-+} holds with $\delta>0$, which is the minimum of the two constants determined in \eqref{2hvMT3-1-} and \eqref{2hv13061} (with $C_{*}=C_{0}$). Consider $\mu_{j} = \mu \ast \rho_{j}$ for a sequence $\{\rho_{j}\}_{j=1}^{\infty}$ of standard mollifiers; note that $\mu_{j}(\mathbb{R}^{n}) = \mu(\mathbb{R}^{n})$. 
Then \eqref{2hvMT3-1-} with $\mu$ replaced by $\mu_{j}$ holds for any $j \in \mathbb{N}$. 
Thus, applying the conclusion of \textit{Step 2} in the proof of Theorem \ref{2hvMT1} (2), we find a sequence $\{u_{j}\}_{j=1}^{\infty}$ of nonnegative weak solutions to
\begin{equation*}
\left\{
\begin{aligned}
- \mathcal{L}_\Phi u_{j}& =  P_{l,a,\beta}(u_{j})+\mu_{j} &\text{in }&B_{j}(0),  \\ 
u_{j}& =  0 & \text{in }& \mathbb{R}^n\setminus B_{j}(0)
\end{aligned}
\right.
\end{equation*}
satisfying for any $j \in \mathbb{N}$
\begin{equation}\label{PF40-}
u_{j} \leq 2c_{p}C_{0}{\bf W}^{4j}_{s ,p}[\omega_{2,j}] \le 2c_{p}C_{0}\mathbf{W}_{s,p}[\omega_{2,j}] \quad \text{a.e. in } \mathbb{R}^{n},
\end{equation} 
where
\[ \omega_{2,j} = \delta\left\|{\bf M}_{s p}^{\frac{(p-1)(\beta-1)}{\beta}}[\chi_{B_R}]\right\|_{L^\infty(\mathbb{R}^n)}^{-1}\chi_{B_R}+\mu_{j}. \]
Since $\{P_{l,a,\beta}(4c_{p}C_{0}\mathbf{W}_{s,p}[\omega_{2,j}])\}$ is bounded in $L^{1}(\mathbb{R}^{n})$ by Lemma \ref{2hvTH3}, it follows from Lemma \ref{equi-integrable2} and \eqref{PF40-} that $\{P_{l,a,\beta}(2c_{p}C_{0}\mathbf{W}_{s,p}[\omega_{2,j}])\}$ is equi-integrable in $\mathbb{R}^{n}$ and so is $\{P_{l,a,\beta}(u_{j})\}$. Similarly to the proof of Theorem \ref{2hvMT4} (2), this time using \eqref{PF40-}, we conclude that $\{u_j\}$ converges (up to subsequences) to a nonnegative SOLA $u$ to \eqref{2hvMT2a} which satisfies \eqref{2hvMT2b}.
\end{proof}

\section*{Acknowledgments} 

The authors wish to thank the referee for a careful reading of an earlier version of the paper and for several helpful comments that led to an improved version.

\end{document}